\documentclass[hidelinks,onefignum,onetabnum]{siamart250211}

\ifpdf
  \DeclareGraphicsExtensions{.eps,.pdf,.png,.jpg}
\else
  \DeclareGraphicsExtensions{.eps}
\fi

\newsiamremark{remark}{Remark}
\newsiamremark{hypothesis}{Hypothesis}
\crefname{hypothesis}{Hypothesis}{Hypotheses}
\newsiamthm{claim}{Claim}
\newsiamremark{fact}{Fact}
\crefname{fact}{Fact}{Facts}

\newtheorem{thrm}{Theorem}[section]

\newtheorem{lmm}[thrm]{Lemma}
\newtheorem{rmk}[thrm]{Remark}
\newtheorem{xmpl}[thrm]{Example}
\newcommand{\f}{\frac}
\newcommand{\eps}{\varepsilon}

\newcommand{\cE}{\mathbb{E}}
\newcommand{\p}{\partial}
\newcommand{\al}{\alpha}
\newcommand{\N}{\aleph}
\newcommand{\wtNh}{\widetilde{\N}_h}
\newcommand{\urone}{\uppercase\expandafter{\romannumeral1}}
\newcommand{\urtwo}{\uppercase\expandafter{\romannumeral2}}
\allowdisplaybreaks
\usepackage{lipsum}
\usepackage{amsfonts}
\usepackage{graphicx}
\usepackage{epstopdf}
\usepackage{algorithmic}
\usepackage{subcaption} 
\usepackage{float}
\usepackage{charter}
\usepackage[charter]{mathdesign}

\ifpdf
  \DeclareGraphicsExtensions{.eps,.pdf,.png,.jpg}
\else
  \DeclareGraphicsExtensions{.eps}
\fi

\crefname{hypothesis}{Hypothesis}{Hypotheses}
\crefname{fact}{Fact}{Facts}

\newcommand{\correspondingnote}{%
  \raggedright Corresponding author.%
}

\headers{Stochastic multiscale subdiffusion}{J. Dong, N. Du, X. Guo, M. Liu and X. Zheng}

\title{Analysis and numerical approximation of stochastic multiscale subdiffusion driven by fractional Gaussian noise\thanks{Submitted to the editors DATE.
\funding{This work was partially supported by the Natural Science Foundation of Shandong Province (Nos. ZR2025QB01, ZR2024JQ008, ZR2023ZD33), the National Natural Science Foundation of China (Nos. 12301555, 12131014, 12271303), the National Key R\&D Program of China (Nos. 2023YFA1008903, 2023YFA1009200), the Taishan Scholars Program of Shandong Province of China (No. tstp20250521)}}.}

\makeatletter
\newcommand{\silentthanks}[2]{%
  \g@addto@macro\@thanks{%
    \begingroup
     \long\def\@makefntext##1{%
          \noindent
          \hb@xt@1.8em{\hss\textparagraph}%
          \ignorespaces##1%
        }
      \footnotetext[#1]{#2}%
    \endgroup
  }%
}
\makeatother

\author{Jincheng Dong\thanks{\raggedright School of Mathematics, Shandong University, Jinan 250100, China\protect\\
  (\email{dongjincheng@mail.sdu.edu.cn},\email{duning@sdu.edu.cn},\email{liumengmeng423@163.com}, \email{xzheng@sdu.edu.cn}).}
\and Ning Du\footnotemark[2]
\and Xu Guo\thanks{Geotechnical and Structural Engineering Research Center, Shandong University,
17923 JingshiRoad, Jinan, China (\email{guoxu@sdu.edu.cn}).}
\and Mengmeng Liu\footnotemark[2]
\and
Xiangcheng Zheng\textsuperscript{\textdagger\,\textparagraph}
\silentthanks{1}{\correspondingnote}
}

\usepackage{amsopn}

\begin{document}
\maketitle
% REQUIRED
\begin{abstract}
This paper investigates a stochastic multiscale subdiffusion model driven by fractional Gaussian noise, where the multiscale Abel kernel with variable exponent $\alpha(t)\in(0,1)$ is used to capture multiscale and crossover behavior in anomalous diffusion. The main difficulties of this model lie in the complexity of the multiscale Abel kernel (e.g. non-monotonicity and non-coercivity) and the low regularity caused by the noise. Concerning these issues, we prove the well-posedness and regularity of the mild solutions by means of solution operator approach and a perturbation technique for multiscale Abel kernel. Then both the semidiscrete-in-time and fully-discrete numerical schemes are proposed and analyzed under the low-regularity numerical analysis framework, with proved temporal and spatial convergence rates. Numerical experiments are presented to substantiate the theoretical results.
\end{abstract}

% REQUIRED

% REQUIRED

\section{Introduction}
Subdiffusion is widely used in modeling, e.g., anomalous diffusion \cite{SerMen,GliEliRanChe,IdoEdw,XuFenShaSeeCha}, and has been extensively studied, cf. the book \cite{Jin} and the references therein. Numerical methods for anomalous diffusion problems have also received considerable attention \cite{HuAliEfeLeu,ZengZhaKar2016,ZengZhaKar}. More generally, numerical methods for nonlocal transport equations with memory effects in multiscale media were developed in
\cite{EfeLeuLiPunVab}. In the real world, external stochastic noise is ubiquitous across various systems and environments, which suggests considering stochastic ingredients to describe phenomena where noise significantly affects spatial and temporal dynamics \cite{CleDa,DicGao,KovPri,MijNan}. White noise is widely used in stochastic models \cite{ZhaKar}, while it assumes uncorrelated fluctuations in space and time, which is often unrealistic. Instead, stochastic processes with long-range dependence enable modeling of non-local interactions and multiscale heterogeneity \cite{FenYaoLiWan,PhiRal}. Considerable progress has recently been made in the study of stochastic subdiffusion problems, and 
a typical model takes the following form \cite{NieDeng2022,NieSunDenSiamJNA}
\begin{equation*}
	\left\{  
	\begin{aligned}  
		&\p_tG+\p_t^{1-\al}AG=\dot{W}^H_Q ~ \text{on }(x,t)\in D\times (0,T], \\  
		&G(\cdot,0)=G_0 ~\text{on}~D;~~~G=0 ~~\text{on}~~\p D\times(0,T],  
	\end{aligned}
	\right.  
\end{equation*}
where $D\subset \mathbb{R}^d~(d=1,~2, ~3)$ is a convex polygonal domain, $T>0$, $\p_t^{1-\al}$ with $0<\alpha<1$ refers to the fractional differential operator defined via the Abel kernel $t^{\al-1}/\Gamma(\al)$
\begin{equation*}
	\begin{aligned}
	\p_t^{1-\al}f(t):=\f{\p}{\p t}\left[\f{t^{\al-1}}{\Gamma(\al)}*f(t)\right]=\f{\p}{\p t}\int_0^t\f{(t-s)^{\al-1}}{\Gamma(\al)}f(s)ds,
    \end{aligned}
\end{equation*}
and $A:=-\Delta$. The $\dot{W}^H_Q$ refers to the fractional Gaussian noise, which will be rigorously defined in the subsequent section. Apart from the fractional Gaussian noise,  other noises have also been adopted in the literature of stochastic subdiffusion problems, including the fractionally integrated multiplicative noise \cite{WuYan2025}, fractionally integrated additive noise \cite{JinYanZho} and fractional Brownian sheet noise \cite{NieSunDeng,SunNieDeng2023}.
% fractionally integrated additive fractional Brownian motion \cite{LiDew}.

In most studies of subdiffusion, the exponent $\alpha$ is taken as a constant, which corresponds to the single-scale power law of the mean squared displacement. However, in complex subdiffusive processes there often exist transitions between diffusive regimes, which the single-scale subdiffusion models could not properly depict. A natural way to characterize such multiscale subdiffusive behavior is to replace the constant exponent $\alpha$ by a variable exponent $\alpha(t)$, resulting in the multiscale Abel kernel $t^{\al(t)-1}/\Gamma(\al(t))$ and the corresponding multiscale subdiffusion model. For instance, the work \cite{SunZhaCheRee} compares single-scale and multiscale subdiffusion based on experimental data of uranine transport, which indicates that the multiscale subdiffusion provides a better description. A main reason, as stated in the abstract of reference \cite{SunZhaCheRee}, is that the growth of contaminant plumes may not exhibit constant scaling and may instead transition between diffusive states at various transport scales. Recently, the multiscale nature of the multiscale Abel kernel has been studied in reference \cite{ZheCMS}, which substantiates the applicability of multiscale Abel kernel in modeling crossover behavior.

%Thus, the multiscale model serves a competitive instrument for modeling such multiscale behavior. In the multiscalet case, a primary challenge lies in the fact that the kernel lacks the necessary properties of constant-exponent case to apply Laplace transform. Consequently, Zheng \cite{Zheng} proposed a perturbation method for multiscale problems.

 Motivated by the above discussions, we consider the following stochastic multiscale subdiffusion equation driven by the fractional Gaussian noise
\begin{equation}\label{model_2}
	\left\{  
	\begin{aligned}  
		&\p_tG+\p_t^{1-\al(t)}AG=\dot{W}^H_Q ~ \text{on }(x,t)\in D\times (0,T],  \\  
		&G(\cdot,0)=G_0 ~~\text{on}~~D;~~~~ G=0 ~~\text{on}~~\p D\times(0,T],    
	\end{aligned}
	\right.
\end{equation}
where $\p_t^{1-\al(t)}$ with $0<\al(t)< 1$ refers to the multiscale fractional differential operator, %%%原来的$0<\al(t)\leq 1$改成了$0<\al(t)< 1$
\begin{align*} 
	\p_t^{1-\al(t)}f(t):=\f{\p}{\p t}\left[\f{t^{\al(t)-1}}{\Gamma(\al(t))}*f(t)\right].
\end{align*}Throughout the work we assume that $\al'(t),~\al''(t)$ exists and is bounded over $[0,T]$. This assumption implies the continuity of $\alpha$ on $[0,T]$ such that $\al_*=\inf_{t\in[0,T]}\al(t)>0$. Furthermore, we set  $\al_0=\al(0)$ for simplicity.

\underline{Main difficulties}: Different from the classical Abel kernel, the multiscale Abel kernel lacks favorable properties (e.g. the complete monotonicity and the explicit representation of its integral transform) that are critical for theoretical studies of the corresponding differential equations. For this reason, the existing analysis tools for single-scale subdiffusion problems do not apply directly. Recently, the work \cite{Zheng} proposed a perturbation method to circumvent this difficulty, and this idea has subsequently been applied to perform rigorous mathematical and numerical analysis for multiscale subdiffusion and viscoelastic problems \cite{ZheACOM,ZheCMS,ZheMMS}. However, the stochastic noise in (\ref{model_2}) leads to low regularity and thus causes fundamental differences in both mathematical analysis and numerical approximation of (\ref{model_2}) from the studies in references \cite{ZheACOM,ZheCMS,ZheMMS}, where the forcing terms are assumed to be relatively smooth. Furthermore, the perturbation method introduces an additional nested convolution structure: the convolution of the solution operator and a convolutional perturbation term involving the unknown solution, which is not encountered in the conventional studies of stochastic subdiffusion. 

Concerning the aforementioned issues, a comprehensive study for (\ref{model_2}) is carried out in this work. Specifically, the well-posedness and solution regularity have been proved, based on which the semi-discrete and fully-discrete numerical schemes are developed and analyzed, providing a substantial complement to the studies of stochastic subdiffusion models.
The rest of the paper is organized as follows: In Section \ref{sec:model_ana}, we introduce preliminaries and prove the well-posedness of the proposed model. The spatial regularity estimate and the temporal H\"older estimate are proved in Section \ref{sec:regular}. Section \ref{sec:timediscre} presents the temporal discretization scheme, and Section \ref{sec:TemDisAux} establishes auxiliary estimates and spatial regularity of numerical solutions to support error estimates in Section \ref{sec:err_temp_discre}. Section \ref{sec:spatial_discre} introduces the finite element spatial discretization and analyzes the resulting fully discrete scheme. 
In Section \ref{sec:simulation}, several numerical experiments are carried out to validate the theoretical results. A conclusion is given in Section \ref{sec:conclu}.% and outline potential directions for future research 

% Throughout this paper, $C$, which may vary across occurrences, denotes a generic constant.
\section{Modeling issues and well-posedness}\label{sec:model_ana}

\subsection{Preliminaries}
We first follow references \cite{Adams2003sobolev,Chow2007stochastic,Jin,JinYanZho} to introduce the following spaces and norms. 
Let $C^m(D)$ denote the vector space of all functions $u$ on $D$ such that $u$ and all its partial derivatives up to order $m$ are continuous. Let $C^\infty(D):=\bigcap_{m=0}^\infty C^m(D)$ and $C_0^\infty(D)$ consist of functions in $C^\infty(D)$ with compact support in $D$. For $1\leq p\leq \infty$, let $L^p(D)$ denote the $p$-th power Lebesgue integrable functions on $D$. In particular, we denote $\mathbb{H}:=L^2(D)$ with the corresponding inner product $(\cdot, \cdot)$. Then we define $W^{m,p}(D):=\{v\in L^p(D):D^l v\in L^p(D) ~\text{for}~0\leq|l|\leq m\}$ for some positive integer $m$ and $1\leq p\leq \infty$ where $D^l$ denotes the weak partial derivative of order $l$. Let $H^m(D):=W^{m,2}(D)$ and $H^m_0(D)$ be the closure of $C_0^\infty(D)$ in $W^{m,2}(D)$. Throughout the work,
$\|\cdot\|$ denotes the function or operator norm on $\mathbb{H}$, $C$ denotes a generic positive constant that may assume different values at different occurrences, and $\eps>0$ denotes an arbitrarily small positive number that may vary and is chosen smaller than the finitely many required thresholds.

Let $\{\lambda_j,\varphi_j\}_{j=1}^\infty$ be eigenpairs of $A:H^2(D)\cap H^1_0(D)\rightarrow \mathbb{H}$ with the zero boundary condition. For any $s \geq 0$, we define the Hilbert space $\hat{H}^s(D)$ as $\hat{H}^s(D): = \{g \in L^2(D): \sum_{j=1}^{\infty} \lambda_j^s (g, \varphi_j)^2 < \infty\}$ with the norm
$\|g\|_{\hat{H}^s(D)}^2 := \sum_{j=1}^{\infty} \lambda_j^s (g, \varphi_j)^2$. Furthermore, it is clear that $\|g\|_{\hat{H}^s(D)}^2=\|A^{s/2}g\|_{L^2(D)}^2$ for $s\geq 0$ and $\hat{H}^0(D)=\mathbb H$.
Let $\mathscr{L}(\mathbb{U};\mathbb{V})$ be a space consisting of all bounded linear operators from $\mathbb{U}$ to $\mathbb{V}$ where $\mathbb{U}$ and $\mathbb{V}$ are two separable Hilbert spaces with norms $\|\cdot\|_{\mathbb{U}}$ and $\|\cdot\|_{\mathbb{V}}$, respectively.  The covariance operator $Q$ is a self-adjoint, nonnegative linear operator on $\mathbb{H}$ and its square-root operator $Q^{1/2}$ is a unique nonnegative, self-adjoint operator on $\mathbb{H}$ such that $(Q^{1/2})^2=Q$. %$A^{-\rho} Q^{1/2}$, with $\rho\in\mathbb{R}$, is assumed to be a bounded operator on $L^2(D)$. 
Furthermore, let $\{\mu_{j}\}_{j\in\mathbb{N}}$ be an orthonormal basis in the separable Hilbert space $\mathbb{U}$. The space $\mathcal{L}_{2}(\mathbb{U};\mathbb{V})\subset\mathcal{L}(\mathbb{U};\mathbb{V})$ consists of all Hilbert-Schmidt operators with norm and inner product
\begin{equation*}
	\|S_1\|^{2}_{\mathcal{L}_{2}(\mathbb{U},\mathbb{V})}=\sum_{j\in\mathbb{N}}\|S_1\mu_{j}\|^{2}_{\mathbb{V}},\quad\langle S_1,S_2\rangle_{\mathcal{L}_{2}(\mathbb{U},\mathbb{V})}=\sum_{j\in\mathbb{N}}(S_1\mu_{j},S_2\mu_{j})_{\mathbb{V}},\quad S_1,S_2\in \mathcal{L}_{2}(\mathbb{U},\mathbb{V}).
\end{equation*}
% Let $\{ \phi_k \}_{k \in \mathbb{N}}$ be the orthonormal eigenfunctions of $Q$ with eigenvalues $\{ \Lambda_k \}_{k \in \mathbb{N}}$. The operators $Q^{1/2}$ and $Q^{-1/2}$ are defined as
% \begin{equation*}
%     Q^{1/2}\mu = \sum_{k=1}^\infty\Lambda_k^{1/2} (\mu, \omega_k) \omega_k,~~Q^{-1/2}\mu = \sum_{k: \Lambda_k>0} \Lambda_k^{-1/2} (\mu, \omega_k) \omega_k.
% \end{equation*}
Let $\mathbb{H}_0$ be the image $Q^{1/2}(\mathbb{H})$ with inner product $(\mu, \nu)_{\mathbb{H}_0} = (Q^{-1/2}\mu, Q^{-1/2}\nu)$ for $\mu, \nu \in \mathbb{H}_0$, where $Q^{-1/2}$ denotes the Moore–Penrose inverse of $Q^{1/2}$. We abbreviate the operators $\langle \cdot,\cdot\rangle_{\mathcal{L}_{2}},~\mathcal{L}_2(\mathbb{H}, \mathbb{H})$ and $\mathcal{L}_2(\mathbb{H}_0, \mathbb{H})$ as \ $\langle \cdot,\cdot\rangle,~\mathcal{L}_2$ and $\mathcal{L}_2^0$, respectively. 

The fractional Brownian motion with covariance operator $Q$ is given by $
    W^{H}_{Q}(x,t):=\sum_{k=1}^{\infty}\sqrt{\Lambda_{k}}\phi_{k}(x)W^{H}_{k}(t)
$ \cite{NieDeng2022,NieSunDenSiamJNA},
where the covariance operator $Q$ on a filtered probability space $(\Omega,\mathcal{F},\mathbb{P},\{\mathcal{F}_t\}_{t\geq 0})$ generalizes spatial correlations beyond the Dirac delta and $W^{H}_{k}$ ($k=1,2,\ldots$) are independent one-dimensional  fractional Brownian motion \cite{BenJoh,Yul} with Hurst index $H\in (1/2,1)$, satisfying $W^{H}_{k}(0)=0$, $\mathbb{E}[W^{H}_{k}(t)]=0$ for any $t>0$ and
\begin{equation*}
	\begin{aligned}
		\text{Cov}(t_1,t_2):=E(W^H_k(t_1)W^H_k(t_2))=\frac{1}{2}\left(|t_1|^{2H}+|t_2|^{2H}-|t_1-t_2|^{2H}\right),\quad \forall t_1,t_2\in [0,\infty).
	\end{aligned}
\end{equation*}
The $\dot{W}^H_Q$ denotes the fractional Gaussian noise defined as the derivative of the fractional Brownian motion. Define the space of $\mathbb{H}$-valued squared integrable random variables by %\cite{JinYanZhou}
$
L^2(\Omega;\mathbb{H})=\left\{g:\mathbb{E}[\|g\|^2]=\int_\Omega\|g(\omega)\|^2\mathrm{d}\mathbb{P}(\omega)<\infty\right\},
$
equipped with norm $\|g\|_{L^2(\Omega;\mathbb{H})}=(\mathbb{E}[\|g\|^2])^{\frac{1}{2}}$, where $\mathbb{E}$ denotes the expectation.

We finally introduce some useful sectors and the corresponding estimates.

\begin{lmm}(cf. \cite[(C.2)]{Gunzburger2019};\cite[(2.5)]{NieSunDenSiamJNA})\label{thm_contour}
For $\kappa > 0$ and $\pi/2 < \theta < \pi$, define sector $\Gamma_{\theta,\kappa}$ as 
\begin{equation}\begin{aligned}\label{Def:F_contour}
        \Gamma_{\theta,\kappa}:=\{re^{-\mathrm{i}\theta}:r\geq\kappa\}\cup\{\kappa e^{\mathrm{i}\psi}:|\psi|\leq\theta\}\cup\{re^{\mathrm{i}\theta}:r\geq\kappa\},
	\end{aligned}
\end{equation}
and $\Sigma_\theta = \{z \in \mathbb{C}: z \neq 0, |\arg z| \leq \theta\}$.
Then we have $\left|\frac{e^{z\tau}-1}{\tau^\gamma}\right|\leq C|z|^\gamma$ on $\Gamma_{\theta,\kappa}$ for $\gamma\in(0,1]$ and $\|A^\beta (z^{\al_0-1}(z^{\al_0}+A)^{-1})\| \leq C|z|^{\beta \alpha_0 - 1}$ for any $ z \in \Sigma_{\theta} $ and $ \beta \in [0, 1]$.
\end{lmm}
%$\|A^\beta \tilde{F}(z)\| \leq C|z|^{\beta \alpha_0 - 1}$
% \begin{lmm}[cf. \cite{JinYan}]\label{lem:zal0A}
% For $\kappa > 0$ and $\pi/2 < \theta < \pi$, define the sector as
% $\Sigma_\theta = \{z \in \mathbb{C}: z \neq 0, |\arg z| \leq \theta\}$ and the estimate $\left\|(z^{\al_0}+A)^{-1}\right\|\leq C|z|^{-\al_0}$ holds for any $z\in\Sigma_\theta$.
% \end{lmm}

\begin{lmm}(cf. \cite[Lemma 2.2]{NieSunDenSiamJNA})\label{lmm:nieDW}
Let $f \in L^2([0, T]; L^2_0)$ and $H \in (1/2, 1)$. Then we have 
\begin{equation*}
\mathbb{E} \Big[\big\|\int_0^t f(s)  dW_Q^H(s)\big\|^2 \Big]=H(2H-1) \int_0^t \int_0^t 
\langle f(s) Q^{1/2}, f(r) Q^{1/2} \rangle|r-s|^{2H-2}  dr  ds.
\end{equation*}
\end{lmm}

\begin{lmm}(Gronwall's inequality cf. \cite[Theorem 4.1]{Jin})\label{lem:gron}
Let $\beta > 0$ and $b \geq 0$. Assume $a \in L^1(0,T)$ is nondecreasing and $a(t) \geq 0$ for almost all $t$. If $0\leq v \in L^1(0,T)$ satisfies 
$
v(t) \leq a(t) + \frac{b}{\Gamma(\beta)} \int_0^t \frac{v(s)  ds}{(t-s)^{1-\beta}}
$
for $0 < t \leq T$ a.e., then $v(t) \leq a(t)  E_{\beta,1}\bigl(b t^{\beta}\bigr)$ for almost all $t \in (0,T]$.
If instead $a(t) = a t^{-\gamma}$ for some constants $a > 0$ and $0 < \gamma < 1$, then $v(t) \leq a  \Gamma(1-\gamma)  E_{\beta,1-\gamma}\bigl(b t^{\beta}\bigr)  t^{-\gamma}$ for $0 < t \leq T$, where $E_{\beta, \gamma}(z) = \sum_{k=0}^{\infty} \frac{z^k}{\Gamma(\beta k + \gamma)}$ is the Mittag-Leffler function.
\end{lmm}

\begin{lmm}(Discrete Gronwall Inequality cf. \cite[Lemma 5.1]{AlKar})\label{lmm:disGrowIneq} %$Lemma 5.1.p14
Let $\beta\in(0,1)$, $N>0$ integer, $T>0$, and $t_n=(n/N)T$ for $0\leq n\leq N$ and for some positive integer $N$. Let $(\varphi_n)_{n=1}^N$ be a nonnegative sequence. Assume that there exist 
$\eta_1, \eta_2 \in [0,1)$ and $A_1, A_2, B \geq 0$ such that $\varphi_n \leq A_1 t_n^{-\eta_1} + A_2 t_n^{-\eta_2} 
+ B\tau \sum_{j=1}^{n-1} t_{n-j}^{-1+\beta} \varphi_j,~ 1 \leq n \leq N.$
Then there exists a constant $C = C(\eta_1, \eta_2, \beta, B, T)$ such that $\varphi_n \leq C(A_1 t_n^{-\eta_1} + A_2 t_n^{-\eta_2})$.

\end{lmm}

% {\color{red}\begin{lmm}[{Discrete Gronwall Inequality} cf. \cite{JinYan}] %$Lemma 9.8$p244
% Let $ \varphi^n \geq 0 $ for $ 0 \leq t_n \leq T $. If 
% $
% \varphi^n \leq a t_n^{-\mu} + b \tau \sum_{j=1}^n \varphi^j, 
% \quad \text{for } 0 < t_n \leq T,
% $
% for some $ a, b \geq 0 $, and $ b \tau < \frac{1}{2} $, $ 0 \leq \mu < 1 $, then there is $ c = c(b, T) $ such that
% $
% \varphi^n \leq \frac{ac}{1 - \mu} t_n^{-\mu}, \quad 0 < t_n \leq T.
% $
% \end{lmm}}

\subsection{Model reformulation}%\label{subsec:Int_mod_re}
Model \eqref{model_2} can be decomposed into a deterministic problem
\begin{equation}\label{model_3_det}
	\left\{   
	\begin{aligned}  
		&\p_tv+\p_t^{1-\al(t)}Av=0 ~~~~~~~\text{on}~~~~D\times(0,T],  \\  
		&v(\cdot,0)=G_0 ~\text{on}~D, v=0 ~\text{on}~\p D\times (0,T],    
	\end{aligned}
	\right.
\end{equation}
and a stochastic problem
\begin{equation}\label{model_3_sto}
	\left\{  
	\begin{aligned}  
		&\p_tu+\p_t^{1-\al(t)}Au=\dot{W}^H_Q ~~~ \text{on}~~~D\times (0,T],  \\  
		&u(\cdot,0)=0 ~ \text{on}~D,~~u=0 ~\text{on}~\p D\times(0,T].
	\end{aligned}
	\right.
\end{equation}

To handle the difficulties caused by the variable exponent, we apply the perturbation method proposed in \cite{Zheng} for model reformulation. Specifically, we split the kernel as follows
	\begin{align*}  
		\f{t^{\al(t)-1}}{\Gamma(\al(t))}=\f{t^{\al_0-1}}{\Gamma(\al_0)}+\int_{0}^{t}\p_z\f{t^{\al(z)-1}}{\Gamma(\al(z))}dz=:\f{t^{\al_0-1}}{\Gamma(\al_0)}+g(t),
	\end{align*}
and direct calculations yield
\begin{align*} 
		g(t)=\int_{0}^{t}\f{t^{\al(z)-1}}{\Gamma(\al(z))}\left(\al'(z)\ln t-\digamma(\al(z))\al'(z)\right)dz,~~\digamma(z)=\f{\Gamma'(z)}{\Gamma(z)},
	\end{align*}
	and 
\begin{equation}\label{eq:g'}
	\begin{aligned}
		g'(t)&=\f{t^{\al(t)-1}}{\Gamma(\al(t))}\left(\al'(t)\ln t-\digamma(\al(t))\al'(t)\right)\\
		&~+\int_{0}^{t}\f{(\al(z)-1)t^{\al(z)-2}}{\Gamma(\al(z))}\left(\al'(z)\ln t-\digamma(\al(z))\al'(z)\right)+\f{t^{\al(z)-2}}{\Gamma(\al(z))}\al'(z)dz.
	\end{aligned}
\end{equation}
It is shown in {\cite[Estimates (5.7)--(5.8)]{Zheng}} that for $0<\eps\ll 1$,
\begin{align}\label{Bound_g_g'}
	|g|\leq Ct^{\al_0}(1+|\ln t|),~~|g'|\leq Ct^{\al_0-1}(1+|\ln t|)\leq Ct^{\al_0-1-\eps}.
\end{align}

Based on the kernel splitting and the properties of $g$, the governing equations in (\ref{model_3_det}) and (\ref{model_3_sto}) could be equivalently reformulated as 
\begin{equation}\label{jy1}
		\p_tu+\p_t^{1-\al_0}Au=-g'(t)*Au+\dot{W}^H_Q \text{ and }
\p_tv+\p_t^{1-\al_0}Av=-g'(t)*Av.
\end{equation}

Based on the Laplace transform of the equations in \eqref{jy1}, we define mild solutions of \eqref{model_3_det} and \eqref{model_3_sto} by
\begin{align}\label{v_ref}
	v=F(t)G_0-F(t)*(g'*Av)(t),
\end{align}

\begin{align}\label{u_ref}
	u=F(t)*(\dot{W}^H_Q(t)-g'*Au(t)),
\end{align}
where 
\begin{align}\label{def:F}
	F(t):=\f{1}{2\pi i}\int_{\Gamma_{\theta,\kappa}}e^{zt}z^{\al_0-1}(z^{\al_0}+A)^{-1}dz,\quad\kappa>0,~~\pi/2<\theta<\pi,
\end{align}
and $\Gamma_{\theta,\kappa}$ is defined in \eqref{Def:F_contour}. It is known that for $\gamma\leq\beta\leq\gamma+2$ with $\gamma,~\beta\in\mathbb{R}$, we have {\cite[Corollary 6.2]{Jin}}
\begin{equation}\label{ineq:Ffleqf}
        \|F(t)f\|_{\hat{H}^\beta(D)}^2\leq Ct^{(\gamma-\beta)\al_0}\|f\|_{\hat{H}^\gamma(D)}^2,
\end{equation}
and the following estimate holds.
\begin{lmm}[cf. {\cite[Theorem 2.3]{NieSunDeng}}]\label{lemjy1}
Let $F$ be defined as in \eqref{def:F} and assume that $\|A^{-\rho}\|_{\mathcal{L}_2^0}<\infty$ with $\rho<\min\{\f{H}{\al_0},1+\eps\}$, then we have
$\cE\big[\big\|\int_{0}^{t}A^\sigma F(t-s)\dot{W}^H_Qds\big\|^2\big]\leq C$ for $\sigma\in[-\rho, \min\{1-\rho,\f{H}{\al_0}-\rho-\eps\}]$, where $0<\eps\ll\f{H}{\al_0}-\rho$.
\end{lmm}

\begin{lmm}[cf. {\cite[Theorem 2.4]{NieSunDeng}}]\label{lemjy2}
Let $F$ be defined as in \eqref{def:F} and assume that $\|A^{-\rho}\|_{\mathcal{L}_2^0}<\infty$ with $\rho\in[0,\f{H}{\al_0})\cap[0,1]$, then we have 
\begin{equation*}
    \begin{aligned}
        \cE\bigg[\Big\|\f{\int_{0}^{t}F(t-s)\dot{W}_Q^H(s)ds-\int_{0}^{t-\tau}F(t-\tau-s)\dot{W}_Q^H(s)ds}{\tau^\gamma}\Big\|^2\bigg]\leq C,~~\gamma\in[0,H-\rho\al_0).
    \end{aligned}
\end{equation*}
\end{lmm}

Then the mild solutions to \eqref{model_2} could be defined by $G=F(t)G_{0}+F(t)*\dot{W}_{Q}^{H}-F(t)*g^{\prime}*AG.$

\subsection{Well-posedness}%\label{subsec:wellpose}
%According to the resolvent estimate $\|(z + A)^{-1}\|\leq C|z|^{-1},~\|A(z^{\al_0}+A)^{-1}\|\leq C$ for any $ z \in \Sigma_{\theta},\theta\in(0,\pi)$ ~\cite{Jin,JinYan}, then $\|\tilde{F}(z)\| \leq C|z|^{-1},~|A\tilde{F}(z)\| \leq C|z|^{\al_0-1},\forall z \in \Sigma_{\theta}$, where $\widetilde{\cdot}$ denote the Laplace transform.
%% where \(\tilde{F}(z)\) is the Laplace transform of \(F(t)\). 
%Furthermore,
%
%\begin{align}\label{betaal_1}
%	.
%\end{align}

Define the space $\mathcal{X}=L^{q}(0,T;L^{2}(\Omega;\mathbb{H}))$ for some $\varsigma\geq0$ and $q\in[2,+\infty)$, equipped with the equivalent norm
\begin{equation*}
    \|f\|_{\mathcal{X},q,\varsigma}:=\Big\|e^{-\frac{\varsigma t}{2}}\left(\cE\left[\|f\|^2\right]\right)^{1/2}\Big\|_{L^q(0,T)}=\left\|e^{-\varsigma t}\left(\cE\left[\|f\|^2\right]\right)\right\|_{L^{q/2}(0,T)}^{1/2}.
\end{equation*}

%It is straightforward to obtain the following lemma.
%\begin{lmm}
%	Let $q\in[2,+\infty)$, the norm
%	$$\|f\|_{\mathcal{X},q,\varsigma}=\left\|e^{-\frac{\varsigma t}{2}}\left(\cE\left[\|f\|^2\right]\right)^{1/2}\right\|_{L^q(0,T)}=\left\|e^{-\varsigma t}\left(\cE\left[\|f\|^2\right]\right)\right\|_{L^{q/2}(0,T)}^{1/2}.$$
%\end{lmm}
%\begin{proof}
%For $q\in[2,+\infty)$,
%\begin{equation*}\begin{aligned}
%		&\left\|e^{-\frac{\varsigma t}{2}}\left(\cE\left[\|f\|^2\right]\right)^{1/2}\right\|_{L^q(0,T)}
%		=\left(\int_{0}^{T}\left(e^{-\frac{\varsigma t}{2}}\left(\cE\left[\|f\|^2\right]\right)^{1/2}\right)^qdt\right)^{1/q}\\
%		&=\left(\left(\int_{0}^{T}\left(e^{-\varsigma t}\cE\left[\|f\|^2\right]\right)^{q/2}dt\right)^{2/q}\right)^{1/2}
%		=\left\|e^{-\varsigma t}\left(\cE\left[\|f\|^2\right]\right)\right\|_{L^{q/2}(0,T)}^{1/2}.
%\end{aligned}\end{equation*}
%Thus we finish the proof.
%%For $q=\infty$,
%%\begin{equation}\begin{aligned}
%%		&\|e^{-\frac{\varsigma t}{2}}\left(\cE\left[\|f\|_\mathbb{H}^2\right]\right)^{1/2}\|_{L^{\infty}(0,T)}
%%		=\sup\limits_{t\in[0,T]}\left(e^{-\frac{\varsigma t}{2}}\left(\cE\left[\|f\|_\mathbb{H}^2\right]\right)^{1/2}\right)\\
%%		=&\left(\sup\limits_{t\in[0,T]}\left(e^{-\varsigma t}\cE\left[\|f\|_\mathbb{H}^2\right]\right)\right)^{1/2}
%%		=\left\|e^{-\varsigma t}\left(\cE\left[\|f\|^2\right]\right)\right\|_{L^{\infty}(0,T)}^{1/2}.
%%\end{aligned}\end{equation}
%\end{proof}
\begin{thrm}
   If $\|A^{-\rho}\|_{\mathcal{L}_2^0}<\infty$ with $\rho\in[0,1]\cap[0,\frac{H}{\alpha_0})$ and $G_0\in \mathbb{H}$, then problem \eqref{model_2} admits a unique mild solution $G$ in  $\mathcal{X}$.
\end{thrm}
\begin{proof}
For each $w\in\mathcal{X}$, define $v=\mathcal{M}w:=F(t)G_{0}+F(t)*\dot{W}_{Q}^{H}-F(t)*g^{\prime}*Aw$. First, we demonstrate that the mapping $\mathcal{M}$ is well-defined on $\mathcal{X}$.  Using \eqref{ineq:Ffleqf}, we obtain $\|F(t)G_0\|_{\mathcal{X},q,\varsigma} \leq C \|G_0\| < \infty$. Applying \eqref{Bound_g_g'}, $\|g\|_{\hat{H}^s(D)}^2=\|A^{s/2}g\|^2$ for $s\geq 0$ and $(z^{\al_0}+A)^{-1}=A(z^{\al_0}+A)^{-1}A^{-1}$, we have
% \begin{equation*}
% 	\begin{aligned}
% 		\cE&\left[\left\| F(t)*g'*Aw\right\|^2\right]
% 		=\cE\left[\int_D\left(\int_{0}^{t}AF(t-s)(g'* w)(s)ds\right)^2dx\right]\\
% 		&\leq  Ct^{1-\eps}\left[\int_{0}^{t}\left\|w(\xi)\right\|^2\int_{\xi}^{t}(t-s)^{-\al_0}\cdot(s-\xi)^{\al_0-1-\eps}dsd\xi\right]\\
% 		&\leq  Ct^{1-\eps}\left[\int_{0}^{t}(t-\xi)^{-\eps}\cE\left\| w(\xi)\right\|^2d\xi\right],
% 	\end{aligned}
% \end{equation*}

\begin{equation}\label{ineq:FgAw}
	\begin{aligned}
		\cE&\left[\left\| F(t)*g'*Aw\right\|^2\right]
		=\cE\bigg[\int_D\Big(\int_{0}^{t}AF(t-s)(g'* w)(s)ds\Big)^2dx\bigg]\\
        &\leq\cE\left[\int_D\int_{0}^{t}(t-s)^{-\al_0}ds\int_{0}^{t}(t-s)^{\al_0}\left(AF(t-s)(g'* w)(s)\right)^2dsdx\right]\\
       % &=\f{t^{1-\al_0}}{1-\al_0}\cE\left[\int_{0}^{t}(t-s)^{\al_0}\left\|AF(t-s)(g'* w)(s)\right\|^2ds\right]\\
		&=\f{t^{1-\al_0}}{1-\al_0}\cE\bigg[\int_{0}^{t}(t-s)^{\al_0}\big\|F(t-s)(g'* w)(s)\big\|_{\hat{\mathbb{H}}^2}^2ds\bigg]\\
        &\leq C t^{1-\al_0}\cE\bigg[\int_{0}^{t}(t-s)^{-\al_0}\big\|(g'* w)(s)\big\|^2ds\bigg]\\
        % &\leq  Ct^{1-\eps}\cE\left[\int_{0}^{t}(t-s)^{-\alpha_0}\int_{0}^{s}|g'(s-\xi)|\cdot\left\| w(\xi)\right\|^2d\xi ds\right]\\
        &\leq Ct^{1-\eps}\cE\bigg[\int_{0}^{t}\left\| w(\xi)\right\|^2\int_{\xi}^{t}(t-s)^{-\al_0}\cdot|g'(s-\xi)|dsd\xi\bigg]\\
		&\leq  Ct^{1-\eps}\bigg[\int_{0}^{t}\cE\left[\left\| w(\xi)\right\|^2\right]\int_{\xi}^{t}(t-s)^{-\al_0}\cdot(s-\xi)^{\al_0-1-\eps}dsd\xi\bigg]\\
        &\leq  Ct^{1-\eps}\bigg[\int_{0}^{t}(t-\xi)^{-\eps}\cE\left[\left\| w(\xi)\right\|^2\right]d\xi\bigg],
	\end{aligned}
\end{equation}
where $0<\eps\ll\al_0$ and we used 
\begin{equation*}
		\int_{\xi}^{t}(t-s)^{-\al_0}(s-\xi)^{\al_0-1-\eps}ds
		=\int_{0}^{t-\xi}(t-\xi-\eta)^{-\al_0}\eta^{\al_0-1-\eps}d\eta=\frac{\Gamma(1-\alpha_0)\Gamma(\alpha_0-\varepsilon)}{\Gamma(1-\varepsilon)}(t-\xi)^{-\eps}.
\end{equation*}
%where $\left\|g'(t)* w(t)\right\|^2$ is bounded by
%	\begin{equation*}
%		\begin{aligned}
%			\left\|g'(t)* w(t)\right\|^2\leq Ct^{\al_0-\eps}\int_{0}^{t}|g'(t-s)|\cdot\left\| w(s)\right\|^2ds.
%		\end{aligned}
%	\end{equation*}

By Young's convolution inequality,  $\int_{0}^{t}e^{-\varsigma y}y^{-\varepsilon}dy\leq\varsigma^{-(1-\varepsilon)}\Gamma(1-\varepsilon)$ and $w\in\mathcal{X}$, we obtain
\begin{equation}\label{ineq:FcovGcovAw}
\begin{aligned}
    &\|F(t)*(g'*Aw)(t)\|_{\mathcal{X},q,\varsigma}^2 =\left\|e^{-\varsigma t}\left(\cE\left[\|F(t)*(g'*Aw)(t)\|^2\right]\right)\right\|_{L^{q/2}(0,T)}\\
    % &\leq CT^{1-\varepsilon}\left\|e^{-\varsigma t}\left(\int_0^t(t-\xi)^{-\varepsilon}\cE[\left\|w(\xi)\right\|^2]d\xi\right)\right\|_{L^{q/2}(0,T)}\\
    &\quad\leq CT^{1-\varepsilon}\Big\|\Big(\int_0^t(t-\xi)^{-\varepsilon}e^{-\varsigma (t-\xi)}e^{-\varsigma\xi}\cE[\left\|w(\xi)\right\|^2]d\xi\Big)\Big\|_{L^{q/2}(0,T)}\\
    &\quad\leq CT^{1-\varepsilon}\|t^{-\varepsilon}e^{-\varsigma t}\|_{L^1(0,T)}\left\|e^{-\varsigma t}\cE[\left\|w(\xi)\right\|^2]\right\|_{L^{q/2}(0,T)}\leq C T^{1-\varepsilon} \varsigma^{-1+\varepsilon} \|w\|_{\mathcal{X},q,\varsigma}^2 .
\end{aligned}
\end{equation}
By taking $\sigma=0$ in Lemma \ref{lemjy1}, we have $\cE\left[\|F(t)*\dot{W}_Q^H(t)\|\right]\leq C$, which implies $\|F(t)*\dot{W}_Q^H(t)\|_{\mathcal{X},q,\varsigma}\leq C\left( \frac{2}{\varsigma q} \left( 1 - e^{-\frac{\varsigma q}{2}T} \right) \right)^{1/q}$. Thus, the mapping $\mathcal{M}$ is well-defined. %C\left\|e^{-\frac{\varsigma t}{2}}\right\|_{L^q(0,T)}=

Next, we prove its contractivity. For $v_{i} \in \mathcal{X}$ and $v_{i} := \mathcal{M} w_{i}$ for $i = 1, 2$, $e_{w} := w_{1} - w_{2}$ and $e_{v} := v_{1} - v_{2}$ satisfy $e_{v}= -F(t)*(g^{\prime}*Ae_{w})(t)$. We multiply both sides of this equation by $e^{-\varsigma t}$ to obtain $e^{-\varsigma t}e_{v}=-e^{-\varsigma t}F(t)*(g^{\prime}*Ae_{w})(t)$, 
and follow the same argument as \eqref{ineq:FgAw} to obtain
% ultilize Cauchy–Schwarz inequality and \eqref{ineq:Ffleqf} to get
\begin{equation*}%\label{contra_ineq}
\begin{aligned}
		e^{-\varsigma t}\cE\left[\|e_{v}\|^2\right]
		% =&e^{-\varsigma t}\cE\left[\int_D\left(\int_{0}^{t}AF(t-s)(g'* e_w)(s)ds\right)^2dx\right]\\
		% \leq&e^{-\varsigma t}\cE\left[\int_D\int_{0}^{t}(t-s)^{-\al_0}ds\int_{0}^{t}(t-s)^{\al_0}\left[AF(t-s)(g'* e_w)(s)\right]^2dsdx\right]\\
		% % =&\f{t^{1-\al_0}}{1-\al_0}e^{-\varsigma t}\cE\left[\int_{0}^{t}(t-s)^{\al_0}\left\|AF(t-s)(g'* e_w)(s)\right\|^2ds\right]\\
		% % =&\f{t^{1-\al_0}}{1-\al_0}e^{-\varsigma t}\cE\left[\int_{0}^{t}(t-s)^{\al_0}\left\|F(t-s)(g'* e_w)(s)\right\|_{\hat{\mathbb{H}}^2}^2ds\right]\\
		% \leq&C \f{t^{1-\al_0}}{1-\al_0}e^{-\varsigma t}\cE\left[\int_{0}^{t}(t-s)^{-\al_0}\left\|g'* e_w(s)\right\|^2ds\right]\\
		% \leq & Ct^{1-\eps}e^{-\varsigma t}\left(\int_{0}^{t}(t-s)^{-\alpha_0}\int_{0}^{s}|g'(s-\xi)|\cdot\cE\left[\left\| e_w(\xi)\right\|^2\right]d\xi ds\right)\\		
		% % =&Ct^{1-\eps}\left(\int_{0}^{t}(t-s)^{-\alpha_0}e^{-\varsigma(t-s)}\int_{0}^{s}e^{-\varsigma(s-\xi)}|g'(s-\xi)|e^{-\varsigma\xi}\cdot\cE\left[\left\| e_w(\xi)\right\|_{\mathbb{H}}^2\right]d\xi ds\right)\\	
		\leq C\left[\left(t^{-\alpha_0}e^{-\varsigma t}\right)*\left(e^{-\varsigma t}g'(t)\right)*\left(e^{-\varsigma t }\cE\left[\left\| e_w(t)\right\|^2\right]\right)\right].
\end{aligned}\end{equation*}
%where $\left\|g'*e_w(t)\right\|^2$ is bounded by using Minkowski's inequality and \eqref{Bound_g_g'},
%\begin{equation*}
%	\begin{aligned}
%		\left\|g'*e_w(t)\right\|^2
%        =\left\|\int_{0}^{t}g'(t-s)\cdot e_w(s)ds\right\|^2
%		% \leq \left(\int_{0}^{t}|g'(t-s)|\cdot\left\|e_w(s)\right\|_{\mathbb{H}}ds\right)^2\\
%		&\leq\int_{0}^{t}|g'(t-s)|ds\int_{0}^{t}|g'(t-s)|\cdot\left\|e_w(s)\right\|^2ds\\
%		&\leq Ct^{\al_0-\eps}\int_{0}^{t}|g'(t-s)|\cdot\left\|e_w(s)\right\|^2ds.
%	\end{aligned}
%\end{equation*}
By applying Young's convolution inequality and \eqref{Bound_g_g'}, and following a similar derivation as \eqref{ineq:FcovGcovAw}, we obtain
% and the fact that $\int_{0}^{t}e^{-\varsigma y}y^{-\alpha_0}dy\leq\varsigma^{-(1-\alpha_0)}\Gamma(1-\alpha_0)$, for $q\in\left[2,+\infty\right)$, 
\begin{equation*}\begin{aligned}
		&\left\|e_{v}\right\|_{\mathcal{X},q,\varsigma}^2=\Big\|e^{-\frac{\varsigma t}{2}}\left(\cE\left[\|e_{v}\|^2\right]\right)^{1/2}\Big\|^2_{L^q(0,T)}
		=\left\|e^{-\varsigma t}\left(\cE\left[\|e_{v}\|^2\right]\right)\right\|_{L^{q/2}(0,T)}\\
		&\leq C\left[\left\|t^{-\alpha_0}e^{-\varsigma t}\right\|_{L^1(0,T)}\left\|e^{-\varsigma t}g'(t)\right\|_{L^1(0,T)}\left\|e^{-\varsigma t }\cE\left[\left\| e_w(t)\right\|^2\right]\right\|_{L^{q/2}(0,T)}\right]\leq C\varsigma^{-1+\varepsilon}\|e_{w}\|_{\mathcal{X},q,\varsigma}^2.
\end{aligned}\end{equation*}
%这里原本是CT^{1-\eps}\varsigma^{-1+\varepsilon}\|e_{w}\|_{\mathcal{X},q,\varsigma}^2，节省空间省区了T^{1-\eps}
% which implies $\|e_{v}\|_{\mathcal{X},q,\varsigma}\leq CT^{(1-\eps)/2}\varsigma^{(-1+\varepsilon)/2}\|e_{w}\|_{\mathcal{X},q,\varsigma}$.
Thus $\mathcal{M}$ is a contraction from $\mathcal{X}$ to $\mathcal{X}$ for $\varsigma$ large enough such that $C\varsigma^{-1+\varepsilon}<1$, which implies the existence of a mild solution of (\ref{model_2}) in $\mathcal X$. To show the uniqueness,
assume that $\tilde{G}$ is another solution to \eqref{model_2}. Then following the same procedure as (\ref{ineq:FgAw}), we have
%Using the analogue of \eqref{contra_ineq}, we obtain the following estimate
\begin{equation*}
		\cE\left[\|G-\tilde{G}\|^2\right]\leq Ct^{1-\eps}\bigg[\int_{0}^{t}(t-\xi)^{-\eps}\cE\bigg[\Big\|G(\xi)-\tilde{G}(\xi)\Big\|^2\Big]d\xi\bigg].
\end{equation*}
Then we use Lemma \ref{lem:gron} to obtain $\cE\left[\|G-\tilde{G}\|^2\right]= 0$, which completes the proof.
\end{proof}

\section{Solution regularity}\label{sec:regular}
In this section, we investigate the regularity properties of the solution $G$ to problem \eqref{model_2} based on the decomposition $G=u+v$, where $u$ and $v$ solve the stochastic problem \eqref{model_3_sto} and the deterministic problem \eqref{model_3_det}, respectively.
\subsection{Regularity of the stochastic problem} 

\begin{thrm}\label{thm_reg}
	Let $u$ be the mild solution of problem \eqref{model_3_sto} and assume $\|A^{-\rho}\|_{\mathcal{L}_2^0}<\infty$ with $\rho<\min\left\{1+\varepsilon,\f{H}{\al_0}\right\}$, then 
	%虽然Nie文章没要求rho<1，但sigma\geq0要求rho<1
	$\cE\left[\|A^\sigma u\|^2\right]<C$
	for $\sigma\in[-\rho, \min\{1-\rho,\f{H}{\al_0}-\rho-\eps\}]$. %由于rho一定大于0所以sigma一定小于1
\end{thrm}

\begin{proof}
Applying \eqref{u_ref}, we obtain
\begin{equation}\label{Ine_reg_1}
	\begin{aligned}
		\cE\left[\|A^\sigma u\|^2\right]&=\cE\left[\left\|A^\sigma \left(F(t)*(\dot{W}^H_Q(t)-g'*Au(t))\right)\right\|^2\right]\\
		&\leq2\cE\bigg[\Big\|\int_{0}^{t}A^\sigma F(t-s)\dot{W}^H_Qds\Big\|^2\bigg]+2\cE\bigg[\Big\|\int_{0}^{t}A^\sigma F(t-s)g'*Auds\Big\|^2\bigg].
	\end{aligned}
\end{equation}
Let $0<\eps\ll\min\{1,\f{H}{\al_0}-\rho\}$, the first right-hand side term is bounded by some constant $C$ (cf. Lemma \ref{lemjy1}), and the second right-hand side term can be bounded in analogous with \eqref{ineq:FgAw}
\begin{equation*}%\label{Ine_reg_3}
	\begin{aligned}
		\cE&\bigg[\Big\|\int_{0}^{t}A^\sigma F(t-s)g'*Au\Big\|^2\bigg]
		\leq  C\left[\int_{0}^{t}(t-\xi)^{-\eps}\cE\left\|A^\sigma u(\xi)\right\|^2d\xi\right].
	\end{aligned}
\end{equation*}
%where $\left\|g'*A^\sigma u\right\|^2$ is bounded by
%	\begin{equation*}
%		\begin{aligned}
%			\left\|g'(t)*A^\sigma u(t)\right\|^2\leq Ct^{\al_0-\eps}\int_{0}^{t}|g'(t-s)|\cdot\left\|A^\sigma u(s)\right\|^2ds.
%		\end{aligned}
%	\end{equation*}
%	\begin{equation}
%		\begin{aligned}
%			\left\|g'(t)*A^\sigma u(t)\right\|_{\mathbb{H}}^2&=\left\|\int_{0}^{t}g'(t-s)\cdot A^\sigma u(s)ds\right\|_{\mathbb{H}}^2\\
%			&\leq \left(\int_{0}^{t}|g'(t-s)|\cdot\left\|A^\sigma u(s)\right\|_{\mathbb{H}}ds\right)^2 ~~(Minko.~ineq.)\\
%			&\leq\int_{0}^{t}|g'(t-s)|ds\int_{0}^{t}|g'(t-s)|\cdot\left\|A^\sigma u(s)\right\|_{\mathbb{H}}^2ds~~(using |g'|\leq Ct^{\al_0-1-\varepsilon})\\
%			&=Ct^{\al_0-\eps}\int_{0}^{t}|g'(t-s)|\cdot\left\|A^\sigma u(s)\right\|_{\mathbb{H}}^2ds
%		\end{aligned}
%	\end{equation}
    Substituting the estimate into \eqref{Ine_reg_1} and applying Lemma \ref{lem:gron} completes the proof.
\end{proof}
\begin{remark}
Throughout the remainder of the paper, we fix $\sigma$ and
choose $\eps$ as follows. If $H>\alpha_0$, we choose $0<\eps\ll\frac{H}{\alpha_0}-1$ and set $\sigma=1-\rho$. If $H\leq\alpha_0$, we choose $0<\eps\ll\f{H}{\al_0}-\rho$ and set $\sigma=\f{H}{\al_0}-\rho-\eps$. Any additional restrictions on $\eps$ will be specified in the corresponding results; otherwise, the above conditions are assumed to hold.
\end{remark}

We now consider the temporal H\"older regularity.
\begin{thrm}\label{thm_holder_reg}
	Let $u$ be the mild solution of problem \eqref{model_3_sto} and assume $\|A^{-\rho}\|_{\mathcal{L}_2^0}<\infty$ with $\rho\in[0,\frac{H}{\alpha_0})\cap [0,1]$,  then we have
	\begin{align*}
		\cE\left[\left\|\f{u(t)-u(t-\tau)}{\tau^\gamma}\right\|^2\right]\leq C,~~\gamma\in (0,H-\rho\al_0).
	\end{align*}
	%如果 sigma=1 该证明最后算围道积分就需要要求gamma>0
\end{thrm}
\begin{proof}
	According to \eqref{u_ref}, we have
	\begin{equation}\label{hol_reg_1}
		\begin{aligned}
			&\cE\bigg[\Big\|\f{u(t)-u(t-\tau)}{\tau^\gamma}\Big\|^2\bigg]\\
			=&\cE\bigg[\Big\|\f{\int_{0}^{t}F(t-s)(\dot{W}_Q^H-g'*Au)(s)ds-\int_{0}^{t-\tau}F(t-\tau-s)\left(\dot{W}_Q^H-g'*Au\right)(s)ds}{\tau^\gamma}\Big\|^2\bigg]\\
			\leq&2\cE\bigg[\Big\|\f{\int_{0}^{t}F(t-s)\dot{W}_Q^H(s)ds-\int_{0}^{t-\tau}F(t-\tau-s)\dot{W}_Q^H(s)ds}{\tau^\gamma}\Big\|^2\\
			&\qquad +\Big\|\f{\int_{0}^{t-\tau}F(t-\tau-s)(g'*Au)(s)ds-\int_{0}^{t}F(t-s)\left(g'*Au\right)(s)ds}{\tau^\gamma}\Big\|^2\bigg].\\
		\end{aligned}
	\end{equation}
	% Taking $\rho\in[0,\frac{H}{\alpha_0}]\cap[0,1]$ and $\gamma<H-\rho\alpha_0$, 
    The first right-hand side of \eqref{hol_reg_1} is bounded by some constant (cf. Lemma \ref{lemjy2}) and the second right-hand side term can be bounded as
	\begin{equation}\label{eq_holder_reg}
		\begin{aligned}
			&4\cE\bigg[\Big\|\f{\int_{t-\tau}^{t}F(t-s)(g'*Au)(s)ds}{\tau^\gamma}\Big\|^2\bigg]\\
			&\qquad+4\cE\bigg[\Big\|\f{\int_{0}^{t-\tau}(F(t-\tau-s)-F(t-s))(g'*Au)(s)ds}{\tau^\gamma}\Big\|^2\bigg]
			=:4\left(\urone+\urtwo\right).
		\end{aligned}
	\end{equation}
	Using Cauchy-Schwarz inequality and \eqref{ineq:Ffleqf}, $\urone$ can be bounded by
	%这里要满足for $0\leq\sigma\leq1$,也就是rho要\in[0,1](H/al>1),in [H/al-1,H/al]
    \begin{equation}\label{urone}
		\begin{aligned}
			\urone\leq& \tau^{-2\gamma}\cE\left[\int_D\int_{t-\tau}^{t}(t-s)^{-\al_0}ds\int_{t-\tau}^{t}(t-s)^{\al_0}(F(t-s)(g'*Au)(s))^2dsdx\right]\\
			% =&C\tau^{1-\al_0-2\gamma}\cE\left[\int_{t-\tau}^{t}(t-s)^{\al_0}\left\|F(t-s)(g'*Au)(s)\right\|_{\mathbb{H}}^2ds\right]\\
			% =&C\tau^{1-\al_0}\cE\left[\int_{t-\tau}^{t}(t-s)^{\al_0}\left\|A^{1-\sigma}F(t-s)(g'*A^\sigma u)(s)\right\|_{\mathbb{H}}^2ds\right]\\
			\leq&C\tau^{1-\al_0-2\gamma}\cE\left[\int_{t-\tau}^{t}(t-s)^{\al_0}\left\|F(t-s)(g'*A^\sigma u)(s)\right\|_{\hat{\mathbb{H}}^{2-2\sigma}}^2ds\right]\\
			\leq&C\tau^{1-\al_0-2\gamma}\cE\left[\int_{t-\tau}^{t}(t-s)^{(2\sigma-1)\al_0}\left\|(g'*A^\sigma u)(s)\right\|^2ds\right]\\
			% \leq&C\tau^{1-\al_0-2\gamma}t^{\al_0-\eps}\int_{t-\tau}^{t}(t-s)^{(2\sigma-1)\al_0}\int_{0}^{s}|g'(s-\xi)|\cdot\cE\left[\left\|A^\sigma u(\xi)\right\|_{\mathbb{H}}^2\right]d\xi ds\\
			\leq&C\tau^{1-\al_0-2\gamma}t^{\al_0-\eps}\bigg(\int_{0}^{t-\tau}\cE\left[\left\|A^\sigma u(\xi)\right\|^2\right]\int_{t-\tau}^{t}(t-s)^{(2\sigma-1)\al_0}|g'(s-\xi)|dsd\xi\\
			&\hspace{2em}+
			\int_{t-\tau}^{t}\cE\left[\left\|A^\sigma u(\xi)\right\|^2\right]\int_{\xi}^{t}(t-s)^{(2\sigma-1)\al_0}|g'(s-\xi)|dsd\xi\bigg)=:I_1+I_2,
		\end{aligned}
	\end{equation}
	where we take $\sigma=\min\left\{1-\rho,\f{H}{\al_0}-\rho-\eps\right\}$ with $0<\eps\ll \min\{\f{1}{\al_0}-1,\al_0\}$. We apply the variable substitution $y=\frac{t-s}{t-\xi}$ to obtain
    \begin{equation*}%\label{ineq:intintI1}
		\begin{aligned}
			&\int_{t-\tau}^{t}(t-s)^{(2\sigma-1)\al_0}(s-\xi)^{\al_0-1-\eps}ds= (t-\xi)^{2\sigma\al_0-\eps}\int_{0}^{\f{\tau}{t-\xi}}y^{(2\sigma-1)\al_0}(1-y)^{\al_0-1-\eps}dy\\
            &\qquad\leq(t-\xi)^{2\sigma\al_0-\eps}\left(1-\f{\tau}{t-\xi}\right)^{\al_0-1-\eps}\left(\f{\tau}{t-\xi}\right)^{(2\sigma-1)\al_0+1}= (t-\xi-\tau)^{\al_0-\eps-1}\tau^{(2\sigma-1)\al_0+1},%上限带入第二个y，剩下的求积分
		\end{aligned}
	\end{equation*}	
% 	\begin{equation}\label{ineq:intintI1}
% 		\begin{aligned}
% 			\int_{t-\tau}^{t}&(t-s)^{(2\sigma-1)\al_0}(s-\xi)^{\al_0-1-\eps}ds		=\int_{t-\xi-\tau}^{t-\xi}(t-\xi-\eta)^{(2\sigma-1)\al_0}\eta^{\al_0-1-\eps}d\eta\\
% %			=&\int_{r-\tau}^{r}(r-\eta)^{(2\sigma-1)\al_0}\eta^{\al_0-1-\eps}d\eta
% %			=r^{2\sigma\al_0-\eps}\int_{\f{r-\tau}{r}}^{1}(1-k)^{(2\sigma-1)\al_0}k^{\al_0-1-\eps}dk%\eta=rk\\
% 			=& (t-\xi)^{2\sigma\al_0-\eps}\left(B\left(\al_0-\eps,(2\sigma-1)\al_0+1\right)-B\left(1-\frac{\tau}{(t-\xi)};\al_0-\eps,(2\sigma-1)\al_0+1\right)\right)\\
% 			\leq&(t-\xi)^{2\sigma\al_0-\eps}\left(\frac{\tau}{t-\xi}\right)^{(2\sigma-1)\al_0+1}
% 			\leq C(t-\xi)^{\al_0-1-\eps}\tau^{(2\sigma-1)\al_0+1}.
% 		\end{aligned}
% 	\end{equation}	
	and we invoke this and Theorem \ref{thm_reg} to bound $I_1$ as
	$I_1\leq C\tau^{2(\sigma-1)\al_0+2-2\gamma}$. Since $\gamma\in(0,H-\rho\al_0)$, we obtain $I_1\leq C\tau^{2\left(\sigma-1\right)\al_0+2-2(H-\rho\al_0)}$. If $H>\al_0$, then $\sigma=1-\rho$ and the exponent becomes $2-2H\geq 0$. If $H\leq\al_0$, then $\sigma=\f{H}{\al_0}-\rho-\eps$, and the exponent reduces to $ 2-2(1+\eps)\al_0\geq 0$. Consequently, we have $I_1\leq C$. Furthermore, we apply (\ref{Bound_g_g'}) to get
	\begin{equation*}%\label{urone_2}
		\begin{aligned}
			I_2\leq& C\tau^{1-\al_0-2\gamma}t^{\al_0-\eps}\left(\int_{t-\tau}^{t}\cE\left[\left\|A^\sigma u(\xi)\right\|^2\right]\int_{\xi}^{t}(t-s)^{(2\sigma-1)\al_0}(s-\xi)^{\al_0-1-\eps}dsd\xi\right)\\
			\leq& C\tau^{1-\al_0-2\gamma}t^{\al_0-\eps}\left(\int_{t-\tau}^{t}\cE\left[\left\|A^\sigma u(\xi)\right\|^2\right](t-\xi)^{2\sigma\al_0-\eps}d\xi\right)
			% \\ \leq& C\tau^{(2\sigma-1)\al_0+2-2\gamma-\eps}t^{\al_0-\eps}
            \leq C\tau^{(2\sigma-1)\al_0+2-2\gamma-\eps}.%\leq C\tau^{(2\sigma-1)\al_0-2\gamma+2} .       
		\end{aligned}
	\end{equation*}
%	where using the fact that
%	\begin{equation}
%		\begin{aligned}
%			&\int_{\xi}^{t}(t-s)^{(2\sigma-1)\al_0}(s-\xi)^{\al_0-1-\eps}ds\\
%			\leq&(t-\xi)^{2\sigma\al_0-\eps}\int_{0}^{1}(1-k)^{(2\sigma-1)\al_0}k^{\al_0-1-\eps}dk
%			\leq C(t-\xi)^{2\sigma\al_0-\eps}.%~~(\sigma>\frac{\varepsilon-1}{2\alpha_0})
%		\end{aligned}
%	\end{equation}
	%算出来I_1比较大，这也比较符合直观感觉
Note that the exponent $(2\sigma-1)\al_0+2-2\gamma-\eps$ is nonnegative under the above assumptions. Indeed, if $H>\al_0$, then the exponent becomes $2-2H+\al_0-\varepsilon$. If $H\leq\al_0$, the exponent becomes $2-\al_0-\varepsilon(1+2\al_0)$. Thus we have $I_2<C$. Now, we estimate $\urtwo$ by using Lemma \ref{thm_contour} and Minkowski's inequality, %and \eqref{betaal_1}
	\begin{equation}\label{urtwo}
		\begin{aligned}
            \hspace{-2.5mm}\urtwo=&\tau^{-2\gamma}\cE\bigg[\Big\|\int_{0}^{t-\tau}(F(t-\tau-s)-F(t-s))\cdot(g'*Au)(s)ds\Big\|^2\bigg]\\
%			&\leq\tau^{-2\gamma}\cE\left[\left(\int_{0}^{t-\tau}\left\|(F(t-\tau-s)-F(t-s))\cdot(g'*Au)(s)\right\|_{\mathbb{H}}ds\right)^2\right]\\
%			&=\cE\left[\left(\int_{0}^{t-\tau}\left\|\int_{\Gamma_{\theta,\kappa}}\left(e^{z(t-\tau-s)}-e^{z(t-s)}\right)z^{\al_0-1}\left(z^{\al_0}+A\right)^{-1}A^{1-\sigma}dz(g'*A^\sigma u(s))\right\|_{\mathbb{H}}ds\right)^2\right]\\
			%		&\leq \tau^{-2\gamma}\cE\left[\left(\int_{0}^{t-\tau}\left\|\int_{\Gamma_{\theta,\kappa}}\left(e^{z(t-\tau-s)}-e^{z(t-s)}\right)z^{\al_0-1}\left(z^{\al_0}+A\right)^{-1}A^{1-\sigma}(g'*A^\sigma u)(s)dz\right\|_{\mathbb{H}}ds\right)^2\right]\\
			\leq&\tau^{-2\gamma}\cE\bigg[\Big(\int_{0}^{t-\tau}\!\!\!\int_{\Gamma_{\theta,\kappa}}\left|e^{z(t-\tau-s)}-e^{z(t-s)}\right|\left\|z^{\al_0-1}\left(z^{\al_0}+A\right)^{-1}A^{1-\sigma}\right\|\cdot\left\|g'*A^\sigma u(s)\right\||dz|ds\Big)^2\bigg]\\
			\leq& \tau^{-2\gamma}\cE\bigg[\Big(\int_{0}^{t-\tau}\left\|g'*A^\sigma u(s)\right\|\int_{\Gamma_{\theta,\kappa}}|e^{z(t-\tau-s)}|\cdot\left|1-e^{z\tau}\right|\left\|A^{1-\sigma}z^{\al_0-1}(z^{\al_0}+A)^{-1}\right\||dz|ds\Big)^2\bigg]\\ 
			\leq& C\tau^{-2\gamma}\cE\bigg[\Big(\int_{0}^{t-\tau}\left\|g'*A^\sigma u(s)\right\|\int_{\Gamma_{\theta,\kappa}}\left|e^{z(t-\tau-s)}\right|\tau^\gamma|z|^{\gamma+(1-\sigma)\al_0-1}|dz|ds\Big)^2\bigg]\\ 
			\leq& C \cE\bigg[\int_{0}^{t-\tau}\left(t-\tau-s\right)^{-\varrho}\left\|g'*A^\sigma u(s)\right\|^2ds\\
			&\hspace{1in}\cdot\int_{0}^{t-\tau}\left(t-\tau-s\right)^{\varrho}\Big(\int_{\Gamma_{\theta,\kappa}}\left|e^{z(t-\tau-s)}\right||z|^{\gamma+(1-\sigma)\al_0-1}|dz|\Big)^2ds\bigg],\\ 
		\end{aligned}
	\end{equation}
%	in which we use \eqref{g'Asigm_au} to obtain
%	\begin{equation}
%		\begin{aligned}
%			&\int_{0}^{t-\tau}\left(t-\tau-s\right)^{-\varrho}\cE\left[\left\|g'*A^\sigma u(s)\right\|_{\mathbb{H}}^2\right]ds\\
%			\leq&\int_{0}^{t-\tau}\left(t-\tau-s\right)^{-\varrho}s^{\al_0-\varepsilon}\int_{0}^{s}|g'(s-\xi)|\cE\left[\left\|A^\sigma u(\xi)\right\|_{\mathbb{H}}^2\right]d\xi ds\\
%%			\leq& C \int_{0}^{t-\tau}\left(t-\tau-s\right)^{-\varrho}s^{\al_0-\varepsilon}\int_{0}^{s}(s-\xi)^{\al_0-1-\varepsilon}d\xi ds\\
%%			\leq C \int_{0}^{t-\tau}\left(t-\tau-s\right)^{-\varrho}s^{2(\al_0-\varepsilon)}ds\\
%			\leq&C(t-\tau)^{2(\al_0-\varepsilon)-\varrho+1}\int_{0}^{1}(1-s)^{-\varrho}s^{2(\al_0-\varepsilon)}ds\leq C  %(\varrho<1)
%		\end{aligned}
%	\end{equation}
	where we used the following estimate based on the substitution $r=-\cos(\theta)(t-\tau-s)|z|$ 	\begin{equation*}
		\begin{aligned}
			&\int_{\Gamma_{\theta,\kappa}}\left|e^{z(t-\tau-s)}\right||z|^{\gamma+(1-\sigma)\al_0-1}|dz|\\
			=&
			2\int_{\kappa}^\infty e^{|z|\cos\theta(t-\tau-s)}|z|^{\gamma+(1-\sigma)\al_0-1}|dz|+\kappa^{\gamma+(1-\sigma)\al_0}\int_{-\theta}^{\theta}e^{\kappa(t-\tau-s)\cos\iota}d\iota\\
			\leq&\frac{2}{(-\cos(\theta)(t-\tau-s))^{\gamma+(1-\sigma)\al_0}}\int_{-\cos(\theta)(t-\tau-s)\kappa}^{\infty}e^{-r}r^{\gamma+(1-\sigma)\al_0-1}dr+C\\
			\leq&\frac{2\Gamma(\gamma+(1-\sigma)\al_0,-\kappa\cos(\theta)(t-\tau-s))}{(-\cos(\theta)(t-\tau-s))^{\gamma+(1-\sigma)\al_0}}+C,~~%\qquad(\gamma+(1-\sigma)\al_0>0)
		\end{aligned}
	\end{equation*}
 %$\gamma+(1-sigma)al_0>0$
   
    To estimate $\urtwo$, we choose $\varrho$ in \eqref{urtwo} such that $\varrho<1$. In particular, we take $\varrho=1-2\varepsilon$, where $0<\varepsilon\ll\min\{1-\alpha_0,1-H\}$ under the condition $\gamma<H-\rho\alpha_0$, we obtain\begin{equation*}
	    \begin{aligned}
	        \urtwo\leq C \int_{0}^{t-\tau}\left(t-\tau-s\right)^{\varrho-2(\gamma+(1-\sigma)\al_0)}ds\leq C.
            %(\varrho-2(\gamma+(1-\sigma)\al_0)>-1)
	    \end{aligned}
	\end{equation*}
 %    where we require $-(1+\alpha_0)/2<(\sigma-1)\alpha_0<\gamma<1+(\sigma-1)\alpha_0-\varepsilon$. According to Theorem \ref{thm_reg} and $\rho\in[0,\frac{H}{\alpha_0})\cap [0,1]$, $\sigma=\min\left\{1-\rho,\f{H}{\al_0}-\rho-\eps\right\}$, then it requires 
	% $\gamma<\min\left\{1-\rho\al_0,1-\rho\al_0+H-\al_0\right\}$, which is naturally restricted by the condition $\gamma<H-\rho\al_0$. 
    We finally complete the proof in combination with $\urone\leq C$ and \eqref{eq_holder_reg}.
	%由于sigma一定小于1，而holder连续的gamma一定大于0，所以gamma左边的限制可以拿去，而且e^z\tau-1的控制也要强调\gamma>0
\end{proof}

\subsection{Regularity of the deterministic problem}
\begin{thrm}\label{thm:Asig_v}
	If $G_0\in\hat{H}^{2q}(D)$, $q\in[0,\sigma]$, let $v$ be the mild solution of problem \eqref{model_3_det}, then we have %这里$q\in[\sigma-1,\sigma]$， 但由于\hat{H}^s空间s>=0,所以q>=0
		$\|A^\sigma v\|^2< Ct^{(2q-2\sigma)\al_0}\left\|G_0\right\|_{\hat{H}^{2q}(D)}^2$ for $\sigma\in[0,1].
		$
\end{thrm}
\begin{proof}
	From \eqref{v_ref}, we obtain
	\begin{equation*}
		\begin{aligned}
			\|A^\sigma v\|&=\left\|A^\sigma \left(F(t)G_0-F(t)*(g'*Av)(t)\right)\right\|\\
			&\leq\left\|A^\sigma F(t)G_0\right\|+\Big\|\int_{0}^{t}A^\sigma F(t-s)(g'*Av)(s)ds\Big\|.
		\end{aligned}
	\end{equation*}
	
	Using \eqref{ineq:Ffleqf}, for $\sigma\geq 0$ and $q\in[0,\sigma]$, we estimate 
	$
			\left\|A^\sigma F(t)G_0\right\|=\left\|F(t)G_0\right\|_{\hat{H}^{2\sigma}(D)}\leq C t^{(q-\sigma)\al_0}\left\|G_0\right\|_{\hat{H}^{2q}(D)}.
$
	%where $\sigma\geq0$.
	
	Similar to the method used in deriving \eqref{ineq:FgAw}, using Minkowski's inequality, we have 
	% \begin{equation*}
	% 	\begin{aligned}
	% 		\left\| \int_{0}^{t}A^\sigma F(t-s)(g'*Av)(s)ds\right\|^2\leq Ct^{1-\eps}\left[\int_{0}^{t}(t-\xi)^{-\eps}\left\|A^\sigma v(\xi)\right\|^2d\xi\right],
	% 	\end{aligned}
	% \end{equation*}
    \begin{equation*}
    \begin{aligned}
        &\left\|A^{\sigma}F(t)*(g'*Av)(t)\right\|
        =\Big\|\int_0^t AF(t-s)(g'*A^{\sigma}v)(s)ds\Big\|\\
        &\leq \int_0^t \left\|AF(t-s)(g'*A^{\sigma}v)(s)\right\|ds
        \leq C\int_0^t (t-s)^{-\alpha_0}\|(g'*A^{\sigma}v)(s)\|ds\\
        % &\leq C\int_0^t (t-s)^{-\alpha_0}
        % \int_0^s |g'(s-\xi)|\|A^{\sigma}v(\xi)\|d\xi ds
        % \quad \text{(convolution bound + Fubini)}\\
        % &= C\int_0^t \|A^{\sigma}v(\xi)\|
        % \int_\xi^t (t-s)^{-\alpha_0}|g'(s-\xi)|dsd\xi
        % \quad \text{(Fubini)}\\
        % &\leq C\int_0^t \|A^{\sigma}v(\xi)\|
        % \int_\xi^t (t-s)^{-\alpha_0}(s-\xi)^{\alpha_0-1-\varepsilon}dsd\xi
        % \quad \text{(bound on } g')\\
        % &= C B(1-\alpha_0,\alpha_0-\varepsilon)
        % \int_0^t (t-\xi)^{-\varepsilon}\|A^{\sigma}v(\xi)\|d\xi
        % \quad \text{(Beta function)}\\
        &\leq C\int_0^t (t-\xi)^{-\varepsilon}
        \|A^{\sigma}v(\xi)\|d\xi,
    \end{aligned}
    \end{equation*}
	which implies
	$
			\|A^\sigma v\|\leq C t^{(q-\sigma)\al_0}\left\|G_0\right\|_{\hat{H}^{2q}(D)}+C\Big[\int_{0}^{t}(t-\xi)^{-\eps}\left\|A^\sigma v(\xi)\right\|d\xi\Big].
		$
	According to Lemma \ref{lem:gron} with $q\in[\sigma-1,\sigma)$, we finish the proof.
\end{proof}

	\begin{thrm}\label{thm_holder_reg_det}
		Let $v$ be the solution of problem \eqref{model_3_det} with $G_0\in\mathbb{H}$. Then $v$ satisfies the following H\"older regularity
		% \begin{align*}
		% 	\left\|\f{v(t)-v(t-\tau)}{\tau^\gamma}\right\|_{\mathbb{H}}\leq C+Ct^{-\gamma}\left\|G_0\right\|_{\mathbb{H}},
		% \end{align*}
        $
			\big\|\f{v(t)-v(t-\tau)}{\tau^\gamma}\big\|\leq C+Ct^{-\gamma}$ for $\gamma\in(0,1).
		$
     %根据后面证明(3.12) M1 证明需要\gamma>0
	\end{thrm}
	
	\begin{proof}
		According to \eqref{v_ref}, we have
		\begin{equation}\label{M1M2}
			\begin{aligned}
				&\Big\|\f{v(t)-v(t-\tau)}{\tau^\gamma}\Big\|\\
				=&\Big\|\f{F(t)G_0-F(t)*(g'*Av)(t)-F(t-\tau)G_0+F(t-\tau)*(g'*Av)(t-\tau)}{\tau^\gamma}\Big\|\\
				=&\tau^{-\gamma}\left\|(F(t)-F(t-\tau))G_0+F(t-\tau)*(g'*Av)(t-\tau)-F(t)*(g'*Av)(t)\right\|\\
				\leq&\tau^{-\gamma}\big(\left\|(F(t)-F(t-\tau))G_0\right\|\\
				&\hspace{3em}+\left\|F(t-\tau)*(g'*Av)(t-\tau)-F(t)*(g'*Av)(t)\right\|\big)=:M_1+M_2.
			\end{aligned}
		\end{equation}
		
		Using Lemma \ref{thm_contour}, we estimate
		\begin{equation}\label{M1}
			\begin{aligned}
				M_1
                % =\left\|(F(t)-F(t-\tau))G_0\right\|
	           =&\tau^{-\gamma}\left\|G_0\right\|\cdot\Big\|\Big(\int_{\Gamma_{\theta,\kappa}}\left(e^{zt}-e^{z(t-\tau)}\right)z^{\al_0-1}\left(z^{\al_0}+A\right)^{-1}dz\Big)\Big\|\\%后面再加上
%				\leq& \tau^{-\gamma}\left\|G_0\right\|_{\mathbb{H}}\cdot\left(\tau^\gamma\int_{\Gamma_{\theta,\kappa}}|e^{z(t-\tau)}| |z|^{\al_0+\gamma-1}\left\|\left(z^{\al_0}+A\right)^{-1}\right\|_{\mathbb{H}}|dz|\right)\\
				\leq&
				%\tau^{-\gamma} \left\|G_0\right\|_{\mathbb{H}}\cdot\left(\tau^\gamma\int_{\Gamma_{\theta,\kappa}}|e^{z(t-\tau)}| |z|^{\gamma-1}|dz|\right)\leq C(t-\tau)^{-\gamma}\tau^\gamma\left\|G_0\right\|_{\mathbb{H}}. %~~~~(\gamma>0)\\
               \tau^{-\gamma} \left\|G_0\right\|\cdot\Big(\tau^\gamma\int_{\Gamma_{\theta,\kappa}}|e^{z(t-\tau)}| |z|^{\gamma-1}|dz|\Big)\leq Ct^{-\gamma}. %~~~~(\gamma>0)\\
			\end{aligned}
		\end{equation}
        % For the case $t\in(\tau,2\tau)$, applying \eqref{ineq:Ffleqf}, it holds that $M_1\leq\left\|F(t)G_0\right\|+\left\|F(t-\tau)G_0\right\|\leq C\left\|G_0\right\|$
        % $\tau^{-\gamma}(\left\|(F(t)-F(t-\tau))G_0\right\|)\leq \tau^{-\gamma}\left(\left\|F(t)G_0\right\|+\left\|F(t-\tau)G_0\right\|\right) \leq2C\tau^{-\gamma}\left\|G_0\right\|\leq 2^{\gamma+1}Ct^{-\gamma}\left\|G_0\right\|\leq Ct^{-\gamma}\left\|G_0\right\|$.
	The  $M_2$ can be bounded by $M_{21}+M_{22}$ where 
		\begin{equation}\label{M2}
			M_{21}\!\!=\!\!\tau^{-\gamma}\Big\|\int_{t-\tau}^{t}\!F(t-s)(g'*Av)(s)ds\Big\|,M_{22}\!\!=\!\!\tau^{-\gamma}\Big\|\int_{0}^{t-\tau}\!\!(F(t-\tau-s)-F(t-s))(g'*Av)(s)ds\Big\|.
		\end{equation}
	
	    % Similar to the procedure applied to derive \eqref{urone}, we have $M_{21}\leq\tau^\gamma$ with $\gamma<(\sigma-1)\al_0+1$,  which is naturally restricted by the condition $\gamma<H-\rho\al_0$. Meanwhile, following the similar approach as for \eqref{urtwo}, for $\gamma<H-\rho\al_0$, we have
        
        % Similar to the procedure applied to derive \eqref{urone}, we have $M_{21}\leq\tau^{(\sigma-1)\alpha_0+1}$, which means $\gamma<(\sigma-1)\alpha_0+1$. Meanwhile, following the similar approach as for \eqref{urtwo}, we have

        Following the same argument used to derive \eqref{urone} and applying Theorem \ref{thm:Asig_v},  we obtain $M_{21}\leq C\tau^{(\sigma-1)\al_0+1-\gamma}+C\tau^{(2\sigma-1)\al_0/2+1-\varepsilon-\gamma}$ for $0<\eps\ll\f{\al_0}{2}$. Taking $\sigma=1$ and $\gamma<1$, we have $M_{21}\leq C$. Meanwhile, for $\gamma<1$, we have
	    %由于sigma\in[0,1],这里G_0\in H,所以q=0,sigma=1代入(\sigma-1)\alpha_0+1}得到1 \eqref{urtwo}
		\begin{equation*}%\label{M22}
			\begin{aligned}
				M_{22}
				\leq&\tau^{-\gamma}\int_{0}^{t-\tau}\left\|(F(t-\tau-s)-F(t-s))(g'*Av)(s)\right\|ds\\
				\leq&\tau^{-\gamma}\int_{0}^{t-\tau}\int_{\Gamma_{\theta,\kappa}}\left|e^{z(t-\tau-s)}-e^{z(t-s)}\right|\left\|z^{\al_0-1}\left(z^{\al_0}+A\right)^{-1}A^{1-\sigma}\right\|\cdot\left\|g'*A^\sigma v(s)\right\||dz|ds\\
				\leq& \tau^{-\gamma}\int_{0}^{t-\tau}\left\|g'*A^\sigma v(s)\right\|\int_{\Gamma_{\theta,\kappa}}|e^{z(t-\tau-s)}|\cdot\left|1-e^{z\tau}\right|\left\|A^{1-\sigma}z^{\al_0-1}(z^{\al_0}+A)^{-1}\right\||dz|ds\\ 
				\leq& C\tau^{-\gamma}\bigg(\int_{0}^{t-\tau}\left\|g'*A^\sigma v(s)\right\|\int_{\Gamma_{\theta,\kappa}}\left|e^{z(t-\tau-s)}\right|\tau^\gamma|z|^{\gamma+(1-\sigma)\al_0-1}|dz|ds\bigg)\\ 
				\leq&  C\tau^{-\gamma}\bigg(\int_{0}^{t-\tau}\int_{0}^{s}|g'(s-\xi)|\cdot\left\|A^\sigma v(\xi)\right\|d\xi\int_{\Gamma_{\theta,\kappa}}\left|e^{z(t-\tau-s)}\right|\tau^\gamma|z|^{\gamma+(1-\sigma)\al_0-1}|dz|ds\bigg)\\ 
				\leq& C\tau^{-\gamma}\bigg(\int_{0}^{t-\tau}\int_{0}^{s}(s-\xi)^{\al_0-1-\varepsilon}d\xi\int_{\Gamma_{\theta,\kappa}}\left|e^{z(t-\tau-s)}\right|\tau^\gamma|z|^{\gamma+(1-\sigma)\al_0-1}|dz|ds\bigg)\\
%				\leq& C\tau^\gamma\left(\int_{0}^{t-\tau}s^{\al_0-\varepsilon}\int_{\Gamma_{\theta,\kappa}}\left|e^{z(t-\tau-s)}\right||z|^{\gamma+(1-\sigma)\al_0-1}|dz|ds\right)\\
%				\leq&C\tau^\gamma\left(\int_{0}^{t-\tau}s^{\al_0-\varepsilon}(t-\tau-s)^{\gamma+(1-\sigma)\al_0}ds\right)\\%~~~~(\gamma+(1-\sigma)\al_0>0)
				\leq&C\tau^{-\gamma}(t-\tau)^{\al_0-\varepsilon-\gamma-(1-\sigma)\al_0+1}B(\alpha_0-\varepsilon+1,1-\gamma-(1-\sigma)\alpha_0)\tau^\gamma\leq C.
			\end{aligned}
		\end{equation*}
		Thus, incorporating \eqref{M1} and \eqref{M2} into \eqref{M1M2}, we complete the proof.
\end{proof}
\subsection{Regularity of original problem}
A direct consequence of Theorems \ref{thm_reg} and \ref{thm:Asig_v} is the following spatial regularity estimate.
\begin{thrm}\label{thm_reg_G}
Let $G$ be the mild solution of problem \eqref{model_2}. Suppose that $\|A^{-\rho}\|_{\mathcal{L}_2^0}<\infty$ with $\rho<\min\left\{1+\varepsilon,\frac{H}{\alpha_0}\right\}$ and $G_0\in\hat{H}^{2q}(D),~q\in[0,\sigma)$. We have
\begin{equation*}
	\begin{aligned}
		\cE\left[\|A^\sigma G\|^2\right]<C+ Ct^{(2q-2\sigma)\al_0}\left\|G_0\right\|_{\hat{H}^{2q}(D)}^2,~\sigma=\min\left\{1, 1-\rho,\f{H}{\al_0}-\rho-\eps\right\}.
	\end{aligned}
\end{equation*} 
%由于需要sigma-1<=q<=sigma,只有sigma<1才有q\in[0,\sigma]
\end{thrm}

Combining Theorems \ref{thm_holder_reg} and \ref{thm_holder_reg_det} yields the following H\"older regularity of mild solution for problem \eqref{model_2}.
\begin{thrm}\label{thm_holder_G}
	 Let $G$ be the mild solution of problem \eqref{model_2}, $G_0\in\mathbb{H}$. If $\|A^{-\rho}\|_{\mathcal{L}_2^0}<\infty$ with $\rho\in[0,\f{H}{\al_0})\cap [0,1]$, then
	% \begin{align*}
	% 	\left(\cE\left[\left\|\f{G(t)-G(t-\tau)}{\tau^\gamma}\right\|_{\mathbb{H}}^2\right]\right)^{1/2}\leq C+Ct^{-\gamma}\left\|G_0\right\|_{\mathbb{H}},
	% \end{align*}
    \begin{align*}
		\bigg(\cE\Big[\Big\|\f{G(t)-G(t-\tau)}{\tau^\gamma}\Big\|^2\Big]\bigg)^{1/2}\leq C+Ct^{-\gamma},~~\gamma\in(0,H-\rho\al_0).
	\end{align*}
\end{thrm}
\begin{rmk}
When $\alpha(t)\equiv \alpha$ for some $0<\alpha<1$,  the results in Theorems \ref{thm_reg_G} and \ref{thm_holder_G} are exactly consistent with those for the constant-exponent case, cf. \cite[Theorem 2.5]{NieSunDenSiamJNA}.	%In comparison to the constant-exponent case presented in {\cite[Theorem 2.5]{NieSunDenSiamJNA}}:
   % Let $G$ be the mild solution of \eqref{model:const} and assume $\|A^{-\rho}\|_{\mathcal{L}_2^0} < \infty $,
%    \begin{enumerate}
%        \item[(1)] If $ \rho < \min\left(\frac{H}{\alpha}, 1 + \varepsilon\right) $ and $ G_0 \in \hat{H}^q(D) $ with $ q \in [0, \sigma] $, then
%        \begin{equation*}
%        \mathbb{E} \left[\|A^\sigma G\|^{2}\right] \leq C + Ct^{(2q-2\sigma)\alpha} \|G_0\|^2_{\hat{H}^{2q}(D)},
%        \end{equation*}
%        where $ \sigma = \min\left\{1, 1 - \rho, \frac{H}{\alpha} - \rho - \eps\right\} $.
%    
%        \item[(2)] If $ \rho \in [0, \frac{H}{\alpha}) \cap [0, 1] $ and $ G_0 \in \mathbb{H} $, then
%        \begin{equation*}
%        \left(\mathbb{E}\left[\left\| \frac{G(t) - G(t-\tau)}{\tau^\gamma} \right\|^{2}\right]\right)^{\frac{1}{2}}
%        \leq C + Ct^{-\gamma} \|G_0\|,
%        \end{equation*}
%        where $ \gamma \in [0, H - \rho\alpha) $.
%    \end{enumerate}
%    Our  lead to a conclusion under the substitution of the constant parameter $\alpha$ with the initial value $\alpha_0$ of the variable-exponent function.
\end{rmk}

    % \section{Temporal semi-discretization}
     
\section{Temporal discretization}\label{sec:timediscre}
We construct semidiscrete-in-time schemes for equations in \eqref{jy1}.   For a given $ N \in \mathbb{N}^+ $, define the uniform time step $ \tau = T / N $ and set $ t_i = i\tau $ for $ i = 0, 1, \dots, N $, so that $ 0 = t_0 < t_1 < \cdots < t_N = T $. For notational convenience, set $\tau_{n,k}:=t_n-t_k$.
For the ease of computation, we rewrite $g'(t)$ in \eqref{eq:g'} as
	\begin{equation*}%\label{eq:rewrt_g'}
		\begin{aligned}
			g'(t)&=\f{t^{\al(t)-1}}{\Gamma(\al(t))}(\al'(t)\ln t-\digamma(\al(t))\al'(t))\\
			&\quad+\int_{0}^{1}\f{(\al(zt)-1)t^{\al(zt)-1}}{\Gamma(\al(zt))}\left(\al'(zt)\ln (t)-\digamma(\al(zt))\al'(zt)\right)+\f{t^{\al(zt)-1}}{\Gamma(\al(zt))}\al'(zt)dz\\
			&=:\f{t^{\al(t)-1}}{\Gamma(\al(t))}H(t)+\int_{0}^{1}\f{t^{\al(zt)-1}}{\Gamma(\al(zt))}\cdot U(z,t)dz,
		\end{aligned}
	\end{equation*}
   where the definitions of $H(t)$ and $U(z,t)$ are given by
    \begin{align*}
		&H(t):=\al'(t)\ln t-\digamma(\al(t))\al'(t),~~
		% &U(z,t):=(\al(zt)-1)\al'(zt)\ln(t)-(\al(zt)-1)\digamma(\al(zt))\al'(zt)+\al'(zt).\\%原始写法
  %       &U(z,t) := (\alpha(zt) - 1)H(zt) + \alpha'(zt)\left[1 - (\alpha(zt) - 1)\ln z\right].\\
        U(z,t) := (\alpha(zt) - 1)(H(zt)-\alpha'(zt)\ln z) + \alpha'(zt).
	\end{align*}
  Then we discretize $g'*Au$ at $t=t_n$ for $1\leq n\leq N$ by
	\begin{equation}\label{eq:gpconAu}
		\begin{aligned}
			g'*Au(t_n)
            % =\int_{t_0}^{t_n}g'(t_n-s)Au(s)ds
            =\sum_{k=1}^{n}\int_{t_{k-1}}^{t_k}g'(t_n-s)Au(s)ds
			% &=\sum_{k=1}^{n}\int_{t_{k-1}}^{t_k}\bigg(\f{(t_n-s)^{\al(t_n-s)-1}H(t_n-s)}{\Gamma(\al(t_n-s))}\\
   %          &\hspace{4em}+\int_{0}^{1}\f{(\al(z(t_n-s))-1)(t_n-s)^{\al(z(t_n-s))-1}}{\Gamma(\al(z(t_n-s))}\\
			% &\hspace{5em}\times\left(\al'(z(t_n-s))\ln ((t_n-s))-\digamma(\al(z(t_n-s)))\al'(z(t_n-s))\right)\\
			% &\hspace{5em}+\f{(t_n-s)^{\al(z(t_n-s))-1}}{\Gamma(\al(z(t_n-s)))}\left(\al'(z(t_n-s))\right)dz\bigg)Au(s)ds\\
			=:I_\tau Au(t_n)+J_n+\hat{J}_n.
		\end{aligned}
	\end{equation}
	Here
	\begin{equation}\label{eq:ItauAu}
		\begin{aligned}
			I_\tau Au(t_n)
   %          &=\sum_{k=1}^{n} \int_{t_{k-1}}^{t_k}\bigg(\f{(t_n-s)^{\al(\tau_{n,k})-1}\breve{H}_{n,k}(s)}{\Gamma(\al(\tau_{n,k}))}\\
   %          &\hspace{6em}+\int_{0}^{1}\f{(t_n-s)^{\al(z(\tau_{n,k}))-1}(\al(z(\tau_{n,k}))-1)\tilde{H}_{n,k}(z,s)}{\Gamma(\al(z(\tau_{n,k})))}\\
			% &\hspace{7em}+\f{(t_n-s)^{\al(z(\tau_{n,k}))-1}}{\Gamma(\al(z(\tau_{n,k})))}\al'(z(\tau_{n,k}))dz\bigg)dsAu(x,t_k)\\
			&=\sum_{k=1}^{n} \int_{t_{k-1}}^{t_k}\bigg(\f{(t_n-s)^{\al(\tau_{n,k})-1}\breve{H}_{n,k}(s)}{\Gamma(\al(\tau_{n,k}))}\\
            &\hspace{6em}+\int_{0}^{1}\f{(t_n-s)^{\al(z(\tau_{n,k}))-1}\breve{U}_{n,k}(z,s)}{\Gamma(\al(z(\tau_{n,k})))}dz\bigg)dsAu(x,t_k)\\
			&=:\sum_{k=1}^{n}\int_{t_{k-1}}^{t_k}\mathcal{K}_{n,k}(s)dsAu(x,t_k)=:\sum_{k=1}^{n}b_{n,k}Au(x,t_k),
		\end{aligned}
	\end{equation}
	where $\breve{H}_{n,k}(s):=\tilde{H}_{n,k}(1,s)$ and $ \breve{U}_{n,k}(z, s) := (\alpha(z \tau_{n,k}) - 1) \tilde{H}_{n,k}(z, s) + \alpha'(z \tau_{n,k})$ with $\tilde{H}_{n,k}(z,s):=\al'(z\tau_{n,k})(\ln(t_n-s)-\digamma(\al(z\tau_{n,k})))$. The truncation errors $J_n$ and $\hat{J}_n$ in \eqref{eq:gpconAu} are defined as 
	\begin{equation*}
			J_n\!:=\sum_{k=1}^{n}\int_{t_{k-1}}^{t_k}\mathcal{K}_{n,k}(s)(Au(x,s)-Au(x,t_k))ds,\,
			\hat{J}_n\!
            :=\sum_{k=1}^{n}\int_{t_{k-1}}^{t_k}(g'(t_n-s)-\mathcal{K}_{n,k}(s))Au(x,s)ds.
	\end{equation*}	
	 
    For the convenience of subsequent analysis, we decompose $b_{n,k}:=\hat{b}_{n,k}+\tilde{b}_{n,k}$ where
	\begin{equation*}
		\begin{aligned}
		% b_{n,k}&=\hat{b}_{n,k}+\tilde{b}_{n,k},\\
		\hat{b}_{n,k}&=\int_{t_{k-1}}^{t_k}\f{(t_n-s)^{\al(\tau_{n,k})-1}\breve{H}_{n,k}(s)}{\Gamma(\al(\tau_{n,k}))}ds
		% &=\int_{t_{k-1}}^{t_k}\f{(t_n-s)^{\al(\tau_{n,k})-1}\cdot[\al'(\tau_{n,k})
		% 	\ln(t_n-s)-\digamma(\al(\tau_{n,k}))\al'(\tau_{n,k})]}{\Gamma(\al(\tau_{n,k})}ds\\
		%&=\f{\al'(\tau_{n,k})\int_{t_{k-1}}^{t_k}\f{\ln (t_n-s)ds}{(t_n-s)^{1-\al(\tau_{n,k})}}-\digamma(\al(\tau_{n,k}))\al'(\tau_{n,k})\int_{t_{k-1}}^{t_k}(t_n-s)^{\al(\tau_{n,k})-1}ds}{\Gamma(\al(\tau_{n,k}))}\\
		&=:\f{\al'(\tau_{n,k})\hat{b}_{n,k}^1-\digamma(\al(\tau_{n,k}))\al'(\tau_{n,k})\hat{b}_{n,k}^2}{\Gamma(\al(\tau_{n,k}))},
		\end{aligned}
	\end{equation*}
	with
	\begin{equation*}
		\begin{aligned}
			\hat{b}_{n,k}^1:=&\int_{t_{k-1}}^{t_k}\f{\ln(t_n-s)ds}{(t_n-s)^{1-\al(\tau_{n,k})}}
			=\f{\tau_{n,k-1}^{\al(\tau_{n,k})}\ln(\tau_{n,k-1})}{\al(\tau_{n,k})}
            -\f{\tau_{n,k}^{\al(\tau_{n,k})}\ln(\tau_{n,k})}{\al(\tau_{n,k})}
			+\f{\tau_{n,k}^{\al(\tau_{n,k})}-\tau_{n,k-1}^{\al(\tau_{n,k})}}{\left(\al(\tau_{n,k})\right)^2},
		\end{aligned}
	\end{equation*}	
	\begin{equation*}
		\begin{aligned}
			\hat{b}_{n,k}^2:=&\int_{t_{k-1}}^{t_k}(t_n-s)^{\al(\tau_{n,k})-1}ds
			=\f{\tau_{n,k-1}^{\al(\tau_{n,k})}-\tau_{n,k}^{\al(\tau_{n,k})}}{\al(\tau_{n,k})},
		\end{aligned}
	\end{equation*}	
	and
	\begin{equation*}
		\begin{aligned}
			\tilde{b}_{n,k}:=&\int_{0}^{1}\int_{t_{k-1}}^{t_k}\f{(\al(z\tau_{n,k})-1)(t_n-s)^{\al(z\tau_{n,k})-1}\tilde{H}_{n,k}(z,s)}{\Gamma(\al(z\tau_{n,k}))}+\f{(t_n-s)^{\al(z\tau_{n,k})-1}}{\Gamma(\al(z\tau_{n,k}))}\al'(z\tau_{n,k})dsdz\\
			=&\int_{0}^{1}\tilde{b}_{n,k}^1(z)\int_{t_{k-1}}^{t_k}(t_n-s)^{\al(z\tau_{n,k})-1}\ln(t_n-s)ds\\
            &\hspace{2em}+\tilde{b}_{n,k}^2(z)\int_{t_{k-1}}^{t_k}(t_n-s)^{\al(z\tau_{n,k})-1}dsdz
            =\int_{0}^{1}\tilde{b}_{n,k}^1(z)\tilde{c}_{n,k}^1(z)+\tilde{b}_{n,k}^2(z)\tilde{c}_{n,k}^2(z)dz,
		\end{aligned}
	\end{equation*}
    with
	\begin{equation*}
	   \tilde{b}_{n,k}^1(z):=\f{(\al(z\tau_{n,k})-1)\al'(z\tau_{n,k})}{\Gamma(\al(z\tau_{n,k}))},~~
        \tilde{b}_{n,k}^2(z):=\frac{\al'(z\tau_{n,k})\left[ 1 - (\al(z\tau_{n,k}) - 1)\digamma(\al(z\tau_{n,k})) \right]}{\Gamma(\al(z\tau_{n,k}))} ,
    \end{equation*}
    \begin{equation*}
		\tilde{c}_{n,k}^1(z):=\f{\ln(\tau_{n,k-1})\tau_{n,k-1}^{\al(z\tau_{n,k})}-\ln(\tau_{n,k})\tau_{n,k}^{\al(z\tau_{n,k})}-\tilde{c}_{n,k}^2(z)}{\al(z\tau_{n,k})},~~
		\tilde{c}_{n,k}^2(z):=\f{\tau_{n,k-1}^{\al(z\tau_{n,k})}-\tau_{n,k}^{\al(z\tau_{n,k})}}{\al(z\tau_{n,k})}.
	\end{equation*}
We combine (\ref{eq:gpconAu}) with   the backward Euler method and the convolution quadrature \cite{Lub1,Lub2} to obtain the semidiscrete-in-time scheme of the first equation of \eqref{jy1}
	\begin{equation}\label{scheme}
		\begin{aligned}
			\f{u^n-u^{n-1}}{\tau}+\sum_{i=0}^{n-1}d_i^{(1-\al_0)}Au^{n-i}=\overline{\p}_\tau W^H_Q(t_n)-I_\tau Au^n,\quad u^0=0,
		\end{aligned}
	\end{equation}
	where $\overline{\partial}_\tau W^H_Q(t)=0$ for $t=0$ or $t>T$ and 
    $\overline{\partial}_\tau W^H_Q(t)=\tfrac{W^H_Q(t_n)-W^H_Q(t_{n-1})}{\tau}$ 
    for $t\in (t_{n-1},t_n]$ ($1\leq n\leq N$), and the coefficients $\{d_i^{(1-\alpha_0)}\}$ are determined by \begin{equation}\label{def:delta_tau}
		(\delta_\tau(\xi))^{1-\al_0}=\sum_{i=0}^{\infty}d_i^{(1-\al_0)}\xi^i,\text{ where }\delta_\tau(\xi)=\f{1-\xi}{\tau}.
	\end{equation}
	% \begin{equation*}
	% 		\overline{\partial}_\tau W^H_Q(t) = \left\{
	% 		\begin{aligned}
	% 		&0, && t = t_0, \\
	% 		&\dfrac{W^H_Q(t_n) - W^H_Q(t_{n-1})}{\tau}, && t \in (t_{n-1}, t_n], \\
	% 		&0, && t > t_N.
	% 	\end{aligned}\right.
	% \end{equation*}

	Similarly, we obtain the semidiscrete-in-time scheme of the second equation of \eqref{jy1},
	\begin{equation}\label{scheme_v}
	\begin{aligned}
		\f{v^n-v^{n-1}}{\tau}+\sum_{i=0}^{n-1}d_i^{(1-\al_0)}Av^{n-i}+I_\tau Av^n=0,\quad v^0=G_0.
	\end{aligned}
    \end{equation}

    Now the semidiscrete-in-time numerical approximation to the stochastic multiscale subdiffusion equation (\ref{model_2}) can be defined as $G^n := u^n + v^n$.

    We  provide estimates for the coefficients $b_{n,k}$ in \eqref{eq:ItauAu} for future use.
    \begin{lmm}\label{lem:bnkbound}
	   Let $\alpha(t)$ satisfy $0<\al_*\le\alpha(t)<1$ and $|\alpha'(t)|\le L$ on $[0,T]$. Then there exists a constant $C$ independent of $n,k,\tau$, such that
	\begin{equation*}
		|b_{n,k}|\le Ct_{n-k+1}^{\al_*-1}(1+|\ln(t_{n-k+1})|)\tau\leq Ct_{n-k+1}^{\al_*-1-\tilde{\eps}}\tau,\qquad k=1,\dots,n.
	\end{equation*}
    where $0<\tilde{\eps}<\al_*$. Furthermore, $\sum_{k=1}^{n}|b_{n,k}|\le C$.
    \end{lmm}
    \begin{proof}
	Under the uniform mesh, $t_n-t_k=(n-k)\tau=t_{n-k}$ for every $k$. For $k<n$, we have $t_n-t_k\ge\tau$. By the mean value theorem,
	\begin{equation*}
		\hat b_{n,k}^2=\frac{\tau_{n,k-1}^{\alpha(t_{n-k})}-\tau_{n,k}^{\alpha(t_{n-k})}}{\alpha(t_{n-k})}=\xi^{\alpha(t_{n-k})-1}\tau,\qquad \xi\in[\tau_{n,k},\tau_{n,k-1}]\subset[\tau,T].
	\end{equation*}
     Since $\alpha(t_{n-k})-1\in(-1,0)$, we have
	$
		|\hat b_{n,k}^2|\le \xi^{\alpha(t_{n-k})-1}\tau\le t_{n-k}^{\alpha(t_{n-k})-1}\tau.
$
	Similarly, using $\hat b_{n,k}^1=\xi^{\alpha(t_{n-k})-1}\ln(\xi)\tau$ for some $\xi\in[\tau_{n,k},\tau_{n,k-1}]$ and $|\ln\xi|\le C(1+|\ln (t_{n-k+1})|)$, we have
	$
		|\hat b_{n,k}^1|\le Ct_{n-k}^{\alpha(t_{n-k})-1}\tau\big(1+|\ln (t_{n-k+1})|\big).
	$
	The same bounds hold for $\tilde c_{n,k}^1(z)$ and $\tilde c_{n,k}^2(z)$ uniformly in $z\in[0,1]$:
 $
		|c_{n,k}^1(z)|\le Ct_{n-k}^{\alpha(z(t_{n-k}))-1}(1+|\ln (t_{n-k+1})| )\tau,~|c_{n,k}^2(z)|\le Ct_{n-k}^{\alpha(z(t_{n-k}))-1}\tau.
	$
    Collecting all bounded factors ($\alpha',1/\Gamma,\digamma$) we obtain
	$
		|\tilde{b}_{n,k}|\le C(1+|\ln (t_{n-k+1})| ) t_{n-k}^{\al_*-1}\tau.
	$
    Then $|{b}_{n,k}|\le C(1+|\ln (t_{n-k+1})| )t_{n-k+1}^{\al_*-1}\tau$.
    
	For $k=n$, %Using $\lim_{u\to0^+}u^{\alpha_0}\ln u=0$,
	we have $
		\hat b_{n,n}^2=\frac{\tau^{\alpha_0}}{\alpha_0}$ and $
		\hat b_{n,n}^1=\frac{\tau^{\alpha_0}\ln\tau}{\alpha_0}-\frac{\tau^{\alpha_0}}{\alpha_0^2}.
	$
	Since $R_{n,n}=-\digamma(\alpha_0)\alpha'(0)$,
	\begin{equation*}
		|\hat b_{n,n}|=\bigg|\frac{\alpha'(0)}{\Gamma(\alpha_0)}\left[\frac{\tau^{\alpha_0}\ln\tau}{\alpha_0}-\frac{\tau^{\alpha_0}}{\alpha_0^2}-\digamma(\alpha_0)\frac{\tau^{\alpha_0}}{\alpha_0}\right]\bigg|\leq C(1+|\ln (\tau)| ) \tau^{\al_0}.
	\end{equation*}
	Since $\tilde c_{n,n}^1(z)\equiv\hat b_{n,n}^1$ and $\tilde c_{n,n}^2(z)\equiv\hat b_{n,n}^2$ (independent of $z$), and $\tilde b_{n,n}^1(z),\tilde b_{n,n}^2(z)$ are likewise constants in $z$, we have
	$
		|\tilde b_{n,n}|=|\tilde b_{n,n}^1(1)\hat b_{n,n}^1+\tilde b_{n,n}^2(1)\hat b_{n,n}^2(1)|\le C(1+|\ln (\tau)| ) \tau^{\al_0}.
	$
    Then we obtain $|b_{n,n}|\le C(1+|\ln (\tau)| ) \tau^{\al_0}\le C(1+|\ln (\tau)| ) t_{1}^{\al_*-1}\tau$. Now we estimate $(1+|\ln(t)|)\le C_{\tilde{\eps}}t^{-\tilde{\eps}}$, for $t\in(0,T],~\tilde{\eps}>0$, where $C_{\tilde{\eps}}\leq \f{e^{\tilde{\eps}}-1}{\tilde{\eps}}+T^{\tilde{\eps}} \left( 1 + \max[\ln T, 0] \right)$.%=\sup\limits_{t\in(0,T]}t^{\eps}(1+|\ln(t)|)
    Then taking $0<\tilde{\eps}<\al_*$, we have
    \begin{equation*}
        \begin{aligned}
            \sum_{k=1}^n |b_{n,k}| \leq C  \tau \sum_{k=1}^n t_{n-k+1}^{\alpha_*-1}(1+|\ln (t_{n-k+1})| )\leq CC_{\tilde{\eps}}  \tau \sum_{k=1}^n t_{n-k+1}^{\alpha_*-1-\tilde{\eps}} \le C .
        \end{aligned}
    \end{equation*}
    We finish the proof.
\end{proof}

\section{Auxiliary estimates}\label{sec:TemDisAux}
We derive integral representations of the solutions to the schemes \eqref{scheme} and \eqref{scheme_v}, based on which we establish the spatial regularity of the solutions to the schemes \eqref{scheme} and \eqref{scheme_v}, which will be used in error estimates.
\subsection{Integral representations of numerical solutions}\label{subsec51}
	Multiplying both sides of \eqref{scheme} by $\xi^n$ and summing $n$ from $1$ to $\infty$ lead to
    \begin{equation*}
	\begin{aligned}
	 % &\sum_{n=1}^{\infty}\f{u^n-u^{n-1}}{\tau}\xi^n+\sum_{n=1}^{\infty}\sum_{i=0}^{n-1}d_i^{(1-\al_0)}Au^{n-i}\xi^n=\sum_{i=0}^{n-1}\overline{\p}_\tau\notag W^H_Q(t_n)\xi^n-\sum_{n=1}^{\infty}\sum_{i=0}^{n-1}b_{n,n-i}Au^{n-i}\xi^n.\\
	 % &\left(\f{1-\xi}{\tau}\right)\sum_{n=1}^{\infty}u^n\xi^n+\sum_{i=0}^{\infty}d_i^{(1-\al_0)}\xi^i\sum_{n=1}^{\infty}Au^{n}\xi^n=\sum_{n=1}^{\infty}\overline{\p}_\tau W^H_Q(t_n)\xi^n-\sum_{n=1}^{\infty}\sum_{i=0}^{n-1}b_{n,n-i}Au^{n-i}\xi^n.\\
	 % &\left(\delta_\tau(\xi)+(\delta_\tau(\xi))^{1-\al_0}A\right)\sum_{n=1}^{\infty}u^{n}\xi^n=\sum_{n=1}^{\infty}\overline{\p}_\tau W^H_Q(t_n)\xi^n-\sum_{n=1}^{\infty}\sum_{i=0}^{n-1}b_{n,n-i}Au^{n-i}\xi^n.\\
	 &\sum_{n=1}^{\infty}u^{n}\xi^n=\left(\delta_\tau(\xi)+(\delta_\tau(\xi))^{1-\al_0}A\right)^{-1}\Big(\sum_{n=1}^{\infty}\overline{\p}_\tau W^H_Q(t_n)\xi^n-\sum_{n=1}^{\infty}\sum_{i=0}^{n-1}b_{n,n-i}Au^{n-i}\xi^n\Big).
	\end{aligned}
    \end{equation*}
     Let $\widetilde{\overline{\p}_\tau W^H_Q}$ denote the Laplace transform of $\overline{\p}_\tau W^H_Q$. By definition, we have
		\begin{equation*}
			\begin{aligned}
				\f{z}{e^{z\tau}-1}\widetilde{\overline{\p}_\tau W^H_Q}(\cdot,z)
                % &=\f{z}{e^{z\tau}-1}\mathcal{L}\left(\sum_{k=1}^{\infty}\frac{W_Q^H(t_k)-W_Q^H(t_{k-1})}{\tau}\left(q(t-t_{k-1})-q(t-t_k)\right)\right)\\
				&=\f{z}{e^{z\tau}-1}\sum_{k=1}^{\infty}\frac{W_Q^H(t_k)-W_Q^H(t_{k-1})}{z\tau}\left(e^{-zt_{k-1}}-e^{-zt_k}\right),
				%\\ &=\sum_{k=1}^{\infty}\overline{\p}_\tau W^H_Q(t_k)e^{-zt_k},
			\end{aligned}%感觉有点啰嗦，所以把最后等式注释掉
		\end{equation*}
     which leads to the following equality %, which is established in the proof of {\cite[Proposition~3.2]{GunzMathComp}} (see the equation below~(3.26)),
	   \begin{equation}\label{lapW}
			\sum_{n=1}^{\infty}\overline{\p}_\tau W^H_Q(t_n)e^{-zt_n}=\f{z}{e^{z\tau}-1}\widetilde{\overline{\p}_\tau W^H_Q}.
		\end{equation}

	% Define the Heaviside function by $q(t) = 1$ for $t > 0$ and  $q(t) = 0$ for $ t \leq 0$. and denote $\mathcal{L}$ as Laplace transform, then

    Define
	\begin{equation*}
    	\Gamma_{\theta,\kappa}^\tau := \{ z \in \mathbb{C} : \kappa \leq |z| \leq \frac{\pi}{\tau \sin(\theta)},  |\arg z| = \theta \} \cup \left\{ z \in \mathbb{C} : |z| = \kappa,  |\arg z| \leq \theta \right\},
	\end{equation*}
	where $\kappa\leq\frac{\pi}{t_n|\sin(\theta)|}$ and $\theta\in\left(\frac{\pi}{2},\operatorname{arccot}\left(-\frac{2}{\pi}\right)\right)$. Applying \eqref{lapW} and taking $\xi=e^{-z\tau}$, the solution of scheme \eqref{scheme} can be expressed via the Cauchy integral principle
	\begin{equation}\label{u_nemeric_0}
		\begin{aligned}
			u^n=\f{1}{2\pi i}\int_{\Gamma_{\theta,\kappa}^\tau}e^{zt_n}&\left(\delta_\tau(e^{-z\tau})+(\delta_\tau(e^{-z\tau}))^{1-\al_0}A\right)^{-1}\\
			&\times\Big(\f{z\tau}{e^{z\tau}-1}\widetilde{\overline{\p}_\tau W^H_Q}(z)-\tau\sum_{j=1}^{\infty}\sum_{i=0}^{j-1}b_{j,j-i}Au^{j-i}e^{-zt_j}\Big)dz.
		\end{aligned}
	\end{equation}
    
	Define a continuous auxiliary function
	\begin{equation*}
		\begin{aligned}
			\hat{u}(t)=\f{1}{2\pi i}\int_{\Gamma_{\theta,\kappa}^\tau}e^{zt}&\left(\delta_\tau(e^{-z\tau})+(\delta_\tau(e^{-z\tau}))^{1-\al_0}A\right)^{-1}\\
			&\times\Big(\f{z\tau}{e^{z\tau}-1}\widetilde{\overline{\p}_\tau W^H_Q}(z)-\tau\sum_{j=1}^{\infty}\sum_{i=0}^{j-1}b_{j,j-i}Au^{j-i}e^{-zt_j}\Big)dz,
		\end{aligned}
	\end{equation*}
	then $\hat{u}(t_n)=u^n$. Define 
	\begin{align*}
		\overline{F}(t)=\f{1}{2\pi i}\int_{\Gamma_{\theta,\kappa}^\tau}e^{zt}\left(\delta_\tau(e^{-z\tau})+(\delta_\tau(e^{-z\tau}))^{1-\al_0}A\right)^{-1}\f{z\tau}{e^{z\tau}-1}dz,
	\end{align*}
 %    and denote $\mathcal{L}^{-1}$ as inverse Laplace transform,  then
	% \begin{equation*}
	% 	\mathcal{L}(\hat{u}(t))=\mathcal{L}\left(\overline{F}(t)\right)\cdot\mathcal{L}\left(\overline{\p}_\tau W^H_Q(t)\right)-\f{\mathcal{L}\left(\overline{F}(t)\right)(e^{z\tau}-1)}{z}\sum_{j=1}^{\infty}\sum_{i=0}^{j-1}b_{j,j-i}Au^{j-i}e^{-zt_j},
	% \end{equation*}
	% \begin{equation}\label{u_nemeric}
	% 	\begin{aligned}
	% 	\hat{u}(t)&=\overline{F}(t)*\overline{\p}_\tau W^H_Q(t)-\sum_{j=1}^{\infty}\sum_{i=0}^{j-1}b_{j,j-i}Au^{j-i}\cdot\overline{F}(t)*\mathcal{L}^{-1}\left(e^{-zt_j}\cdot\f{e^{z\tau}-1}{z}\right)\\
	% 	&=\overline{F}(t)*\left(\overline{\p}_\tau W^H_Q(t)-\sum_{j=1}^{\infty}\sum_{i=0}^{j-1}b_{n,n-i}Au^{j-i}\left(q(t-t_{j-1})-q(t-t_j)\right)\right)\\
	% 	&=\overline{F}(t)*\left(\overline{\p}_\tau W^H_Q(t)-\sum_{j=1}^{\infty}I_\tau Au^j\left(q(t-t_{j-1})-q(t-t_j)\right)\right),
	% 	\end{aligned}
	% \end{equation}
	then 
    \begin{equation}\label{u_nemeric}
        \hat{u}(t_n)=\overline{F}(t_n)*\left(\overline{\p}_\tau W^H_Q(t_n)-I_\tau Au^n\right).
    \end{equation}
    
    %下为更早的注释
	% \begin{equation*}
	% 	\begin{aligned}
	% 		u^n=\hat{u}(t_n)=\overline{F}(t_n)*\left(\overline{\p}_\tau W^H_Q(t_n)-I_\tau Au^n\right)
	% 	\end{aligned}
	% \end{equation*}
	% with the fact that $\mathcal{L}^{-1}\left(e^{-za}/z\right)=q(t-a)$.

Similar to the derivation of \eqref{u_nemeric_0}, we obtain
\begin{equation}\label{vn_numer}
    \begin{aligned}
    v^n&=\f{1}{2\pi i}\int_{\Gamma_{\theta,\kappa}^\tau}e^{zt_n}\left(\delta_\tau(e^{-z\tau})+(\delta_\tau(e^{-z\tau}))^{1-\al_0}A\right)^{-1}\Big(e^{-z\tau}G_0-\tau\sum_{j=1}^{\infty}\sum_{i=0}^{j-1}b_{j,j-i}Av^{j-i}e^{-zt_j}\Big)dz\\
    &=\breve{F}(t_{n-1})G_0-\overline{F}(t_n)*I_\tau Av^n,
    \end{aligned}
\end{equation}
where $\breve{F}$ is defined by
\begin{equation}\label{def:breF}
    \breve{F}(t):=\f{1}{2\pi i}\int_{\Gamma_{\theta,\kappa}^\tau}e^{zt}\left(\delta_\tau(e^{-z\tau})+(\delta_\tau(e^{-z\tau}))^{1-\al_0}A\right)^{-1}dz.
\end{equation}

\begin{lmm}[cf. {\cite[Lemma 3.4]{GunzMathComp}}]\label{lem:delta} %JinYan p51 lmm3.1 NIESUN J.N.A. p18 lemma3.1
    Let $\alpha_0>0$, $\varrho\in(0,1)$, and
    $\theta\in\left(\frac{\pi}{2},\operatorname{arccot}\left(-\frac{2}{\pi}\right)\right)$.
    Suppose that $\delta_\tau(\zeta)$ is given by \eqref{def:delta_tau}. If $0<\kappa\leq\min\left(\frac{1}{T},-\frac{\ln(\varrho)}{\tau}\right),$ then $\delta_\tau(e^{-z\tau})$ and $(\delta_\tau(e^{z\tau})^{\alpha_0}+A)^{-1}$ are analytic functions of $z$ in the region 
    \begin{equation*}
    \Omega_{\theta,\kappa,\varrho}^{\tau}:=\left\{z\in\mathbb C:|z|\geq\kappa,\quad|\arg z|\leq\theta,\quad|\operatorname{Im}z|\leq\frac{\pi}{\tau},\quad\operatorname{Re}z\leq-\frac{\ln\varrho}{\tau}\right\},
    \end{equation*}
    and for all $z\in\Gamma_{\theta,\kappa}^{\tau}$ we have $\delta_{\tau}\left(e^{-z\tau}\right)\in\Sigma_{\theta}$,
    		$c_{0}|z|\leq\left|\delta_{\tau}\left(e^{-z\tau}\right)\right|\leq c_{1}|z|$, $\left|\delta_{\tau}\left(e^{-z\tau}\right)-z\right|\leq c\tau|z|^{2}$ and $
            \left|(\delta_{\tau}\left(e^{-z\tau}\right))^{\alpha_0}-z^{\alpha_0}\right|\leq c\tau|z|^{1+\alpha_0}$.
\end{lmm}

According to Lemmas \ref{thm_contour} and \ref{lem:delta}, for $z\in\Gamma_{\theta,\kappa}^{\tau}$, we estimate
    % 好像不用那么麻烦用插值，可以直接估计
	% \begin{equation*}%\label{BoundFInInt}
	% 	\begin{aligned}
	% 		&\Big\|\left(\delta_\tau(e^{-z\tau})+\left(\delta_\tau(e^{-z\tau})\right)^{1-\al_0}A\right)^{-1}\f{z\tau}{e^{z\tau}-1}\Big\| \\
	% 		=&\Big\|-\left(z+z^{1-\al_0}A\right)^{-1}+\left(\delta_\tau(e^{-z\tau})+(\delta_\tau(e^{-z\tau}))^{1-\al_0}A\right)^{-1}\f{z\tau}{e^{z\tau}-1}\Big\|\\
	% 		&~~\qquad~~+\Big\|\left(z+z^{1-\al_0}A\right)^{-1}\Big\|
	% 		\leq C\left(\tau+|z|^{-1}\right)
	% 	\end{aligned}
	% \end{equation*}
	% and
	% \begin{equation*}%\label{BoundLapFA}
	% 	\begin{aligned}
	% 		&\Big\|\left(\delta_\tau(e^{-z\tau})+(\delta_\tau(e^{-z\tau}))^{1-\al_0}A\right)^{-1}\f{z\tau}{e^{z\tau}-1}A\Big\|\\
	% 		=&\Big\|\Big(-\left(z+z^{1-\al_0}A\right)^{-1}+\left(\delta_\tau(e^{-z\tau})+(\delta_\tau(e^{-z\tau}))^{1-\al_0}A\right)^{-1}\f{z\tau}{e^{z\tau}-1}\Big)A\Big\|\\
	% 		&~~\qquad~~+\Big\|\left(z+z^{1-\al_0}A\right)^{-1}A\Big\|
	% 		\leq C\Big(\tau|z|^{\alpha_0}+|z|^{\alpha_0-1}\Big).
	% 	\end{aligned}
	% \end{equation*}
	
	% By interpolation property $\|A^\beta v\|\leq\|Av\|^\beta\|v\|^{1-\beta}$, we obtain
	\begin{equation*}
	 	\begin{aligned}
	 		\big\|\left(\delta_\tau(e^{-z\tau})+(\delta_\tau(e^{-z\tau}))^{1-\al_0}A\right)^{-1}\f{z\tau}{e^{z\tau}-1}A^{\beta}\big\|
	 		\leq C|z|^{\alpha_0\beta-1},~~\beta\in[0,1],
	 	\end{aligned}
	 \end{equation*}
	which implies
	\begin{align}\label{Bound_FA1_sigma}
			&\left\|\overline{F}(t_n-r)A^{1-\sigma}\right\|
   %          =&\Big\|\f{1}{2\pi i}\int_{\Gamma_{\theta,\kappa}^\tau}e^{z(t_n-r)}\left(\delta_\tau(e^{-z\tau})+(\delta_\tau(e^{-z\tau}))^{1-\al_0}A\right)^{-1}\f{z\tau}{e^{z\tau}-1}dzA^{1-\sigma}\Big\|\\
			% \leq& C\int_{\Gamma_{\theta,\kappa}^\tau}|e^{z(t_n-r)}|\cdot\Big\|\left(\delta_\tau(e^{-z\tau})+(\delta_\tau(e^{-z\tau}))^{1-\al_0}A\right)^{-1}\f{z\tau}{e^{z\tau}-1}A^{1-\sigma}\Big\||dz|\\
			\leq C\int_{\Gamma_{\theta,\kappa}^\tau}|e^{z(t_n-r)}|\cdot|z|^{\alpha_0(1-\sigma)-1} |dz|,~~\sigma\in[0,1],
		\\
\label{ineq:disint}
			&\Big(\int_0^{t_n}\left\|\overline{F}(t_n-s)A\right\|ds\Big)^2=\bigg(\int_0^{t_n}\int_{\Gamma_{\theta,\kappa}^\tau}|e^{z(t_n-s)}||z|^{\al_0-1}|dz|ds\bigg)^2\leq C.
		\end{align}
  \subsection{Spatial regularity of numerical solutions}  
We establish the spatial regularity for the solution to scheme \eqref{scheme} in the following lemma.
\begin{lmm}\label{lem:AsigUn_bound}
	Assume that $\|A^{-\rho}\|_{\mathcal{L}_2^0}<\infty$ with $\rho<\min\{\f{H}{\al_0},1+\eps\}$, and let $u^n$ be the solution of the semidiscrete scheme \eqref{scheme}. Then there exists a constant $C$, independent of $n$ and $\tau$, such that $\cE\left[\left\|A^{\sigma}u^{n}\right\|^{2}\right]\leq C,~ n=1,\dots,N,$ for $\sigma\in\big[-\rho,\ \min\{1-\rho,\ \f{H}{\al_0}-\rho-\eps\}\big]$.
\end{lmm}

\begin{proof}
	By \eqref{u_nemeric}, we have $u^n=\overline{F}(t_n)*\overline{\p}_\tau W_Q^H(t_n)-\overline{F}(t_n)*I_\tau Au^n$, and we intend to bound its right-hand side terms acted by $A^\sigma$. It is clear that
	\begin{equation*}
		\begin{aligned}
			\cE&\left[\left\|A^\sigma\overline{F}(t_n)*\overline{\p}_\tau W_Q^H(t_n)\right\|^2\right]
			\leq 2\cE\left[\left\|A^\sigma F(t_n)*\dot{W}_Q^H(t_n)\right\|^2\right]\\
			&\hspace{1.5cm}+2\cE\left[\left\|A^\sigma\Big(\overline{F}(t_n)*\overline{\p}_\tau W_Q^H(t_n)-F(t_n)*\dot{W}_Q^H(t_n)\Big)\right\|^2\right]
			=:2R_1+2R_2.
		\end{aligned}
	\end{equation*}
	By Lemma \ref{lemjy1}, $R_1\leq C$ uniformly for $\sigma\in[-\rho,\min\{1-\rho,\f{H}{\al_0}-\rho-\eps\}]$. Split $R_2$ into
    \begin{equation*}%\label{ineq:R2split}
    	\begin{aligned}
    		R_2\leq&2\cE\bigg[\Big\|A^\sigma\int_{0}^{t_n}F(t_n-s)\left(\dot{W}_Q^H(s)-\overline{\p}_\tau W_Q^H(s)\right)ds\Big\|^2\bigg]\\
    		&+2\cE\bigg[\Big\|A^\sigma\int_{0}^{t_n}\big(F(t_n-s)-\overline{F}(t_n-s)\big)\overline{\p}_\tau W_Q^H(s)ds\Big\|^2\bigg]
    		=:2R_{21}+2R_{22}.
    	\end{aligned}
    \end{equation*}

    % Applying Lemma \ref{thm_contour} and set $\beta=\rho+\sigma,~\gamma+\beta\al_0<1,~\gamma, ~\beta\in[0,1],~r>s$, we obtain
    % \begin{equation}\label{ineq:F_holder}
    % 	\begin{aligned}
    % 		\big\|A^{\beta}[F(t_n-s)-F(t_n-r)]\big\|
    % 		&=\Big\|\f{1}{2\pi i}\int_{\Gamma_{\theta,\kappa}}e^{z(t_n-r)}\big(e^{z(r-s)}-1\big)A^\beta\tilde{F}(z)dz\Big\|\\
    % 		&\hspace{-3cm}\leq C\int_{\Gamma_{\theta,\kappa}}\big|e^{z(t_n-r)}\big|\cdot\big|e^{z(r-s)}-1\big|\cdot\big\|A^\beta\tilde F(z)\big\||dz|
    % 		% &\leq C(r-s)^{\gamma}\int_{\Gamma_{\theta,\kappa}}\big|e^{z(t_n-r)}\big||z|^{\beta\al_0+\gamma-1}|dz|\\
    % 		\leq C(r-s)^{\gamma}(t_n-r)^{-\gamma-\beta\al_0},%\qquad\text{provided }\gamma+\beta\al_0<1,\ 0<s<r<t.
    % 	\end{aligned}
    % \end{equation}

    % Fix $k\in\{1,\dots,n\}$. For $s,r\in(t_{k-1},t_k]$ we have $|s-r|\le\tau$, so applying \eqref{ineq:F_holder} with $t=t_n-r$, $h=|s-r|\le\tau$ (and $t-h\ge t_n-t_k$), and
    Let $\Psi(s):=\tau^{-1}\int_{t_{k-1}}^{t_k}A^\sigma\big[F(t_n-s)-F(t_n-r)\big]dr,~ s\in(t_{k-1},t_k],~ k=1,\dots,n$. Then we have
	\begin{equation*}
		\begin{aligned}
			R_{21}=&\cE\bigg[\Big\|A^{\sigma}\sum_{k=1}^{n}\big(\int_{t_{k-1}}^{t_k}F(t_n-s)\dot{W}_Q^H(s)ds-\int_{t_{k-1}}^{t_k}F(t_n-s)\overline{\p}_\tau W_Q^H(s)ds\big)\Big\|^2\bigg]\\
			=&\cE\bigg[\Big\|A^{\sigma}\sum_{k=1}^{n}\bigg(\int_{t_{k-1}}^{t_k}F(t_n-s)\dot{W}_Q^H(s)ds-\tau^{-1}\Big(\int_{t_{k-1}}^{t_k}F(t_n-r)dr\Big)\int_{t_{k-1}}^{t_k}\dot{W}_Q^H(s)ds\bigg)\Big\|^2\bigg]\\
		%	=&\cE\bigg[\Big\|\tau^{-1}\sum_{k=1}^{n}\big(\int_{t_{k-1}}^{t_k}\int_{t_{k-1}}^{t_k}A^{\sigma}\big[F(t_n-s)-F(t_n-r)\big]dr\;\dot{W}_Q^H(s)ds\big)\Big\|^2\bigg]\\
            =&\cE\bigg[\Big\|\int_0^{t_n}\Psi(s)dW_Q^H(s)\Big\|^2\bigg].
		\end{aligned}
	\end{equation*}
    According to Lemma \ref{lmm:nieDW}, we have
   	\begin{equation}\label{ineq:isometry}
    \begin{aligned}
        R_{21}
		&\le CH(2H-1)\int_0^{t_n}\int_0^{t_n}\big\langle\Psi(s)Q^{1/2},\Psi(r)Q^{1/2}\big\rangle|s-r|^{2H-2}drds\\
        &\le CH(2H-1)\int_0^{t_n}\Big\langle\Psi(s)Q^{1/2},\int_s^{t_n}\Psi(r)Q^{1/2}|r-s|^{2H-2}dr\Big\rangle  ds\\%由于被积函数对称
        &\leq CH(2H-1)\int_0^{t_n}\|\Psi(s)Q^{1/2}\|_{\mathcal{L}_2}\cdot\Big\|\int_0^{t_n-s}\Psi(s+r)Q^{1/2}|r|^{2H-2}dr\Big\|_{\mathcal{L}_2}ds.\\
    \end{aligned}
	\end{equation}
  Now we turn to bound integrands in (\ref{ineq:isometry}).  By $\|A^{-\rho}\|_{\mathcal{L}_2^0}<\infty$, we have $\|\Psi(s)Q^{1/2}\|_{\mathcal{L}_2}\leq\|\Psi(s)A^\rho\|\cdot\|A^{-\rho}\|_{\mathcal{L}_2^0}\le C\|\Psi(s)A^\rho\|$. Applying Lemma \ref{thm_contour} and setting $\beta=\rho+\sigma$, and recalling that $\beta\al_0<1$ and $\beta\in[0,1]$, we obtain for $s<r<t_n$
    \begin{equation}\label{ineq:F_holder}
    	\begin{aligned}
    		\big\|A^{\beta}[F(t_n-s)-F(t_n-r)]\big\|
    		&=\Big\|\f{1}{2\pi i}\int_{\Gamma_{\theta,\kappa}}e^{z(t_n-r)}\big(e^{z(r-s)}-1\big)A^\beta\tilde{F}(z)dz\Big\|\\
    		&\hspace{-3cm}\leq C\int_{\Gamma_{\theta,\kappa}}\big|e^{z(t_n-r)}\big|\cdot\big|e^{z(r-s)}-1\big|\cdot\big\|A^\beta\tilde F(z)\big\||dz|
    		\leq C(t_n-r)^{-\beta\al_0}.
    	\end{aligned}
    \end{equation}
 Thus, for $s\in(t_{k-1},t_k]$ we get
    \begin{equation*}
        \begin{aligned}
        \|\Psi(s)A^\rho\|=&\big\|\tau^{-1}\int_{t_{k-1}}^{t_k}A^{\beta}[F(t_n-s)-F(t_n-r)]dr\big\|
        \leq C\tau^{-1}\int_{t_{k-1}}^{t_k}(t_n-\max\{s,r\})^{-\beta\alpha_0}dr\\
        \leq&C(t_n-s)^{-\beta\alpha_0}
        +C\tau^{-1}((t_n-s)^{1-\beta\alpha_0}
        -(t_n-t_k)^{1-\beta\alpha_0})\\
        \leq&C(t_n-s)^{-\beta\alpha_0}
        +C\tau^{-1}(t_k-s)(t_n-s)^{-\beta\alpha_0} \leq C(t_n-s)^{-\beta\alpha_0}.
    \end{aligned}
    \end{equation*}
We combine this with
	\begin{equation*}
		\begin{aligned}
			\Big\|\int_0^{t_n-s}\Psi(s+r)Q^{1/2}|r|^{2H-2}dr\Big\|_{\mathcal{L}_2}
			&\leq\|A^{-\rho}\|_{\mathcal{L}_2^0}\int_0^{t_n-s}r^{2H-2}(t_n-s-r)^{-\beta\al_0}dr\\
			&\le C(t_n-s)^{2H-1-\beta\al_0}B(2H-1,1-\beta\al_0),
		\end{aligned}
	\end{equation*}
 to obtain $R_{21}\le C\int_0^{t_n}(t_n-s)^{2H-1-2\beta\al_0}ds\le C.$

 % This term $R_{22}$ is treated by the same contour-splitting device used for $\zeta_{1,1}$ and $\zeta_{1,2}$ in the proof of Theorem \ref{time_dic_err}, but with the stochastic integral in place of the deterministic convolution against $g'*Au$.
 To bound $R_{22}$,  we write
\begin{equation*}
	\begin{aligned}
		R_{22}\leq& 2\cE\bigg[\Big\|A^\sigma\int_{0}^{t_n}\int_{\Gamma_{\theta,\kappa}\backslash\Gamma_{\theta,\kappa}^\tau}e^{z(t_n-s)}\widetilde{F}(z)dz\overline{\p}_\tau W_Q^H(s)ds\Big\|^2\bigg]\\
		&+2\cE\bigg[\Big\|A^\sigma\int_{0}^{t_n}\int_{\Gamma_{\theta,\kappa}^\tau}e^{z(t_n-s)}\mathcal{E}_\tau(z)dz\overline{\p}_\tau W_Q^H(s)ds\Big\|^2\bigg]
		=:2R_{221}+2R_{222}.
	\end{aligned}
\end{equation*}
Define $F_{1}(t):=\f{1}{2\pi i}\int_{\Gamma_{\theta,\kappa}\backslash\Gamma_{\theta,\kappa}^\tau}e^{zt}z^{\al_0-1}(z^{\al_0}+A)^{-1}dz.$ For $s\in(t_{k-1},t_k]$, $k=1,\dots,n$, set $\Psi_{1}(s):=\tau^{-1}\int_{t_{k-1}}^{t_k}A^\sigma F_{1}(t_n-\eta)d\eta$ and $\Psi_{1}^{(k)}:=\Psi_{1}(s)|_{s\in(t_{k-1},t_{k}]}$. Then we have
\begin{equation}\label{ineq:phikQ}
    \begin{aligned}
        \|\Psi_{1}^{(k)}(s)Q^{1/2}\|_{\mathcal{L}_2}\leq\|\Psi_{1}^{(k)}(s)A^\rho\|\cdot\|A^{-\rho}\|_{\mathcal{L}_2^0}\le C\tau^{-1}[(t_n-t_{k-1})^{1-\beta\al_0}-(t_n-t_{k})^{1-\beta\al_0}].
    \end{aligned}
\end{equation}
Similar to \eqref{ineq:isometry}, we have
\begin{equation}\label{ineq:R221}\begin{aligned}
R_{221}=&\cE\bigg[\Big\|\int_0^{t_n}\Psi_{1}(s)dW_Q^H(s)\Big\|^2\bigg]\\
\le & CH(2H-1)\int_0^{t_n}\|\Psi_{1}(s)Q^{1/2}\|_{\mathcal{L}_2}\cdot\Big\|\int_0^{t_n-s}\Psi_{1}(s+r)Q^{1/2}|r|^{2H-2}dr\Big\|_{\mathcal{L}_2}ds\\
\le &C\sum_{k=1}^n \|\Psi_{1}^{(k)}(s)Q^{1/2}\|_{\mathcal{L}_2} \int_{t_{k-1}}^{t_k}
\Big\| \int_0^{t_n-s} \Psi_{1}(s+r) Q^{1/2} |r|^{2H-2}  dr \Big\|_{\mathcal{L}_2}  ds.
\end{aligned}
\end{equation}
By
\begin{equation*}
\int_{0}^{t_n - s} \Psi_{1}(s + r) Q^{1/2} r^{2H-2}  dr 
= \Psi_{1}^{(k)} Q^{1/2} \int_{0}^{t_k - s} r^{2H-2}  dr 
+ \sum_{j=k+1}^n \Psi_{1}^{(j)} Q^{1/2} \int_{t_{j-1} - s}^{t_j - s} r^{2H-2}  dr,
\end{equation*}
we have
\begin{equation}\label{ineq:R2212}\begin{aligned}
&\int_{t_{k-1}}^{t_k}\Big\| \int_0^{t_n-s} \Psi_{1}(s+r)Q^{1/2}r^{2H-2}  dr \Big\|_{\mathcal{L}_2} ds
\\
\leq& \int_{t_{k-1}}^{t_k}\frac{(t_k-s)^{2H-1}}{2H-1} \| \Psi_{1}^{(k)}Q^{1/2} \|_{\mathcal{L}_2}
+ \sum_{j=k+1}^n \frac{(t_j-s)^{2H-1} - (t_{j-1}-s)^{2H-1}}{2H-1}\| \Psi_{1}^{(j)}Q^{1/2} \|_{\mathcal{L}_2}ds\\
=& \frac{\tau^{2H}}{2H(2H-1)} \bigg(\| \Psi_{1}^{(k)}Q^{1/2}\|_{\mathcal{L}_2}  +  \sum_{m=1}^{n-k} \| \Psi_{1}^{(k+m)}Q^{1/2} \|_{\mathcal{L}_2} \left[ (m+1)^{2H} - 2m^{2H} + (m-1)^{2H} \right]\bigg).
\end{aligned}
\end{equation}
Applying \eqref{ineq:phikQ}, \eqref{ineq:R221} and \eqref{ineq:R2212}, we obtain $R_{221}\le C$.

Similar to \eqref{ineq:R221}, 
\begin{equation}\label{ineq:R222}
	\begin{aligned}
			R_{222} \le &C\sum_{k=1}^n \|\Psi_{2}^{(k)}(s)Q^{1/2}\|_{\mathcal{L}_2} \int_{t_{k-1}}^{t_k}
\Big\| \int_0^{t_n-s} \Psi_{2}(s+r) Q^{1/2} |r|^{2H-2}  dr \Big\|_{\mathcal{L}_2}  ds.\\
	\end{aligned}
\end{equation}
where $\Psi_{2}(s):=\tau^{-1}\int_{t_{k-1}}^{t_k}A^\sigma F_{2}(t_n-\eta)d\eta,~F_{2}(t):=\f{1}{2\pi i}\int_{\Gamma_{\theta,\kappa}^{\tau}}e^{zt}\mathcal{E}_\tau(z)A^{\rho}dz$ and $\Psi_{2}^{(k)}:=\Psi_{2}(s)|_{s\in(t_{k-1},t_{k}]}$. Using $\|\mathcal{E}_\tau(z)A^{\rho}\|\leq C\tau|z|^{\rho\al_0}$ and $\|\Psi_{2}^{(k)}(s)Q^{1/2}\|_{\mathcal{L}_2}\leq\|\Psi_{2}^{(k)}(s)A^\rho\|\cdot\|A^{-\rho}\|_{\mathcal{L}_2^0}$, we have
\begin{equation}\label{phi2Q}
\|\Psi_{2}^{(k)}(s)Q^{1/2}\|_{\mathcal{L}_2}\leq 
C \left[ (t_n - t_k)^{-\beta \alpha_0} - (t_n - t_{k-1})^{-\beta \alpha_0} \right], ~~k = 1, \dots, n-1,  
\end{equation}
and $\|\Psi_{2}^{(k)}(s)Q^{1/2}\|_{\mathcal{L}_2}\leq C \tau^{-\beta \alpha_0},~k = n.$ Applying \eqref{phi2Q}, we obtain
\begin{equation}\label{ineq:R2222}\begin{aligned}
&\int_{t_{k-1}}^{t_k}\Big\| \int_0^{t_n-s} \Psi_{2}(s+r)Q^{1/2}r^{2H-2}  dr \Big\|_{\mathcal{L}_2} ds
\\
\leq& \frac{\tau^{2H}}{2H(2H-1)} \bigg(\| \Psi_{2}^{(k)}Q^{1/2} \|_{\mathcal{L}_2}  +  \sum_{m=1}^{n-k} \| \Psi_{2}^{(k+m)}Q^{1/2}\|_{\mathcal{L}_2} \left[ (m+1)^{2H} - 2m^{2H} + (m-1)^{2H} \right]\bigg).
\end{aligned}
\end{equation}
Substituting \eqref{phi2Q} and \eqref{ineq:R2222} into \eqref{ineq:R222} yields $R_{222}\le C$. Then we have 
\begin{equation}\label{ineq:AsigFW_bound}
    \begin{aligned}
        \cE\left[\left\|A^\sigma\overline{F}(t_n)*\overline{\p}_\tau W_Q^H(t_n)\right\|^2\right]\le C.
    \end{aligned}
\end{equation}

Using \eqref{ineq:disint} and Lemma \ref{lem:bnkbound}, we obtain
	\begin{equation}\label{ineq:drift_bound}
		\begin{aligned}
			&\cE\left[\left\|A^\sigma\overline{F}(t_n)*I_\tau Au^n\right\|^2\right]
			\leq C\Big(\int_0^{t_n}\left\|\overline{F}(t_n-s)A\right\|ds\Big)^2\cE\left[\left\| I_\tau A^\sigma u^n \right\|^2\right]\\
			\leq& C\sum_{k=1}^{n}|b_{n,k}|\cE\left[\left\|A^\sigma u^k\right\|^2\right]
			\leq C\tau\sum_{k=1}^{n}t_{n-k+1}^{\al_*-1-\tilde{\eps}}\cE\left[\left\|A^\sigma u^k\right\|^2\right].
		\end{aligned}
	\end{equation}%这里用了cauchy不等式，lem:AsigVn_bound没用

	Combining \eqref{ineq:AsigFW_bound} and \eqref{ineq:drift_bound} yields
	\begin{equation*}
		\cE\left[\left\|A^\sigma u^n\right\|^2\right]\leq C+C\tau\sum_{k=1}^{n}t_{n-k+1}^{\al_*-1-\tilde{\eps}}\cE\left[\left\|A^\sigma u^k\right\|^2\right],\qquad n=1,\dots,N.
	\end{equation*}
	Applying Lemma \ref{lmm:disGrowIneq} to complete the proof.
\end{proof}

We now establish the following regularity result for the solution to scheme \eqref{scheme_v}.
\begin{lmm}\label{lem:AsigVn_bound}
	Assume that $G_0\in\mathbb{H}$. Let $v^n$ be the solution of the semidiscrete scheme \eqref{scheme_v}. Then there exists a constant $C$, independent of $n$ and $\tau$, such that $\left\|A^{\sigma}v^{n}\right\|\leq Ct_{n}^{-\sigma\al_0},~~n=1,\dots,N,$ for $\sigma\in[0,1]$.
\end{lmm}
\begin{proof}
By \eqref{vn_numer}, $v^n=\breve{F}(t_{n-1})G_0-\overline{F}(t_n)*I_\tau Av^n$. Via \eqref{ineq:Ffleqf} and Lemma \ref{lem:delta}, for $\gamma\leq\beta\leq\gamma+2$ with $\gamma,\beta\in\mathbb{R}$, we have $\|\breve{F}(t)f\|_{\hat{H}^\beta(D)}^2\leq Ct^{(\gamma-\beta)\al_0}\|f\|_{\hat{H}^\gamma(D)}^2.$ Taking $\gamma=0$ and $\beta=2\sigma$ with $ \sigma\in[0,1]$ leads to
$
    \|A^{\sigma}\breve F(t)G_0\|^2\leq Ct^{-2\sigma\al_0}\|G_0\|^2$ for $ t>0.
$
Thus, for $n=1$ we have $\|A^\sigma\breve F(0)G_0\|\leq C\|G_0\|\int_{\Gamma_{\theta,\kappa}^\tau}|z|^{\sigma\al_0-1}|dz|\leq C\tau^{-\sigma\al_0}$, and for $n\geq2$ we obtain $\|A^\sigma\breve F(t_{n-1})G_0\|\leq Ct_{n-1}^{-\sigma\al_0}\|G_0\|\leq Ct_{n}^{-\sigma\al_0}$. Similar to \eqref{ineq:drift_bound}, we have
$
        \left\|A^\sigma\overline{F}(t_n)*I_\tau Av^n\right\|
        \leq C\tau\sum_{k=1}^{n}t_{n-k+1}^{\al_*-1-\tilde{\eps}}\left\|A^\sigma v^k\right\|.
$
Consequently, $\left\|A^\sigma v^n\right\|\leq Ct_{n}^{-\sigma\al_0}+C\tau\sum_{k=1}^{n}t_{n-k+1}^{\al_*-1-\tilde{\eps}}\left\|A^\sigma v^k\right\|$ for $n=1,\dots,N$, and we apply Lemma \ref{lmm:disGrowIneq} to complete the proof.
\end{proof}

\section{Error estimate of semi-discrete approximations}\label{sec:err_temp_discre}
The main result of this section is the following error estimate for the semi-discrete approximation $G^n$ to the stochastic multiscale subdiffusion model (\ref{model_2}), the proof of which is a direct consequence of error estimates for schemes \eqref{scheme} and \eqref{scheme_v} proved subsequently. 
\begin{thrm}\label{time_dic_err_G}
   Suppose $\|A^{-\rho}\|_{\mathcal{L}_{2}^{0}}<\infty$, $G_0\in\mathbb{H}$ and $\rho\in[0,\frac{H}{\alpha_0})\cap[0,1]$. We have
   $
          \cE\left[\|G(t_n)-G^n\|^2\right]\leq C\tau^{2H-2\rho\al_0}+ Ct_n^{-2}\tau^{2}.
 $
\end{thrm}

\subsection{Error estimate for scheme (\ref{scheme})}
Based on the integral representations of numerical solutions in Section \ref{subsec51}, we subtract \eqref{u_nemeric} from \eqref{u_ref}
 and apply $\cE[\|F(t_n)*\dot{W}_Q^H(t_n)-\overline{F}(t_n)*\overline{\p}_t W_Q^H(t_n)\|^2]\le C\tau^{2H-2\rho\al_0}$ (cf. \cite[Theorem 3.4]{NieSunDenSiamJNA}) to obtain
\begin{equation}\label{error_semidiscrete}
    \begin{aligned}
        &\cE\left[\|u(t_n)-u^n\|^2\right]\\
        % =&\cE\left[\left\|F(t_n)*\left(\dot{W}_Q^H(t_n)-g'(t_n)*Au(t_n)\right)-\overline{F}(t_n)*\left(\overline{\p}_t W_Q^H(t_n)-I_\tau Au^n\right)\right\|^2\right]\\
        =&\cE\left[\left\|F(t_n)*\dot{W}_Q^H(t_n)-\overline{F}(t_n)*\overline{\p}_t W_Q^H(t_n)+\overline{F}(t_n)*I_\tau Au^n-F(t_n)*g'(t_n)*Au(t_n)\right\|^2\right]\\
        \leq& C\tau^{2H-2\rho\al_0}+2\cE\left[\left\|\overline{F}(t_n)*I_\tau Au^n-F(t_n)*g'(t_n)*Au(t_n)\right\|^2\right].
    \end{aligned}
\end{equation}
Then we split $\cE\left\|\overline{F}(t_n)*I_\tau Au^n-F(t_n)*g'(t_n)*Au(t_n)\right\|^2$ into two parts
\begin{equation*}
    \begin{aligned}
        &\cE\left[\left\|\overline{F}(t_n)*I_\tau Au^n-F(t_n)*g'(t_n)*Au(t_n)\right\|^2\right]\\
        =&2\cE\bigg[\Big\|\int_{0}^{t_n}\overline{F}(t_n-s)(g'(s)*Au(s))-F(t_n-s)(g'(s)*Au(s))ds\Big\|^2\bigg]\\
        &~~+2\cE\bigg[\Big\|\overline{F}(t_n)*I_\tau Au^n-\overline{F}(t_n)*g'(t_n)*Au(t_n)\Big\|^2\bigg]=:2\zeta_1+2\zeta_2,
    \end{aligned}
\end{equation*}	
and bound $\zeta_1$ and $\zeta_2$ separately.
\begin{lmm}\label{lmm:bound_zeta1}
    Assume that $\|A^{-\rho}\|_{\mathcal{L}_{2}^{0}}<\infty$ for $\rho\in[0,\frac{H}{\alpha_0})\cap[0,1]$. We have $\zeta_1\leq C\tau^{2(\sigma\al_0-\al_0+1)-\eps},$ where $\sigma=\min\{1-\rho,\f{H}{\al_0}-\rho-\eps\}$.
\end{lmm}
\begin{proof}
    We split $\zeta_1$ into two parts,
	\begin{equation*}
		\begin{aligned}
			\zeta_1
            % &=\cE\left[\left\|\int_{0}^{t_n}\left(\overline{F}(t_n-s)-F(t_n-s))(g'(s)*Au(s)\right)\right\|^2\right]\\
			% &=\cE\left[\left\|\int_{0}^{t_n}\left(\int_{\Gamma_{\theta,\kappa}}e^{z(t_n-s)}\widetilde{F}(z)dz-\int_{\Gamma_{\theta,\kappa}^\tau}e^{z(t_n-s)}\widetilde{\overline{F}}(z)dz\right)g'(s)*Au(s)ds\right\|^2\right]\\
			&\leq 2\cE\bigg[\Big\|\int_{0}^{t_n}\int_{\Gamma_{\theta,\kappa}\backslash\Gamma_{\theta,\kappa}^\tau}e^{z(t_n-s)}\widetilde{F}(z)dzg'(s)*Au(s)ds\Big\|^2\bigg]\\
			&~~~~~+2\cE\bigg[\Big\|\int_{0}^{t_n}\int_{\Gamma_{\theta,\kappa}^\tau}e^{z(t_n-s)}\left(\widetilde{F}(z)-\widetilde{\overline{F}}(z)\right)dzg'(s)*Au(s)ds\Big\|^2\bigg]=:2\zeta_{1,1}+2\zeta_{1,2}.
		\end{aligned}
	\end{equation*}	
    
    Let $0<\eps\ll1-\al_0$ and taking $\sigma=\min\{1-\rho,\f{H}{\al_0}-\rho-\eps\}$ in Theorem \ref{thm_reg} with $\rho\in[0,\frac{H}{\alpha_0})\cap[0,1]$, we have $(1-\sigma)\al_0+1>0$. So that $4(1-(1-\sigma)\al_0)<1-\al_0$, hence the same $\eps$ automatically satisfies $0<\eps\ll4(1-(1-\sigma)\al_0)$. Set $l_1^{'}=1-\eps,~l_1^{''}=1-\eps/2$, then $l_1^{'}<1,~l_1^{'}<l_1^{''},~2(1-\sigma)\al_0-1<l_1^{''}$. Using Cauchy-Schwarz inequality, Lemma \ref{thm_contour}, Theorem \ref{thm_reg} and \eqref{ineq:Ffleqf}, we obtain
    \begin{equation*}%\label{Eq:v11}
       	\begin{aligned}
			\zeta_{1,1}&\leq C \cE\bigg[\int_{0}^{t_n}(t_n-s)^{l_1^{'}}\Big(\int_{\Gamma_{\theta,\kappa}\backslash\Gamma_{\theta,\kappa}^\tau}e^{z(t_n-s)}\left\|A\widetilde{F}(z)g'(s)*u(s)\right\|dz\Big)^2ds\bigg]\\
			&\leq C\cE\bigg[\int_{0}^{t_n}(t_n-s)^{l_1^{'}}\int_{\Gamma_{\theta,\kappa}\backslash\Gamma_{\theta,\kappa}^\tau}|e^{2z(t_n-s)}||z|^{l_1^{''}}|dz|\\
			&\hspace{5em}\times\int_{\Gamma_{\theta,\kappa}\backslash\Gamma_{\theta,\kappa}^\tau}|z|^{-l_1^{''}}\cdot\left\|A\widetilde{F}(z)g'(s)*u(s)\right\|^2|dz|ds\bigg]\notag\\
			% &=C \cE\bigg[\int_{0}^{t_n}(t_n-s)^{l_1^{'}}\int_{\Gamma_{\theta,\kappa}\backslash\Gamma_{\theta,\kappa}^\tau}|e^{2z(t_n-s)}||z|^{l_1^{''}}|dz|\\
			% &\hspace{5em}\times\int_{\Gamma_{\theta,\kappa}\backslash\Gamma_{\theta,\kappa}^\tau}\left\|g'(s)*A^\sigma u(s)\right\|^2|z|^{2(1-\sigma)\al_0-2-l_1^{''}}|dz|ds\bigg]\\
			&\leq C \int_{0}^{t_n}(t_n-s)^{l_1^{'}}\int_{\Gamma_{\theta,\kappa}\backslash\Gamma_{\theta,\kappa}^\tau}|e^{2z(t_n-s)}||z|^{l_1^{''}}|dz|\\
            &\hspace{5em}\times\int_{\Gamma_{\theta,\kappa}\backslash\Gamma_{\theta,\kappa}^\tau}s^{\al_0-\eps}\int_{0}^{s}|g'(s-r)|\cE\left[\left\|A^\sigma u(r)\right\|^2\right]dr|z|^{2(1-\sigma)\al_0-2-l_1^{''}}|dz|ds\notag\\
			&\leq C \int_{0}^{t_n}(t_n-s)^{l_1^{'}}\int_{\Gamma_{\theta,\kappa}\backslash\Gamma_{\theta,\kappa}^\tau}\left|e^{2z(t_n-s)}\right||z|^{l_1^{''}}|dz|\int_{\Gamma_{\theta,\kappa}\backslash\Gamma_{\theta,\kappa}^\tau}|z|^{-l_1^{''}+2(1-\sigma)\al_0-2}|dz|ds\\
			%&\hspace{3in}(-l_1^{''}+2(1-\sigma)\al_0-2<-1)
            % &\leq C\tau^{l_1^{''}-2(1-\sigma)\al_0+1}\int_{\Gamma_{\theta,\kappa}\backslash\Gamma_{\theta,\kappa}^\tau}|z|^{l_1^{''}}\int_{0}^{t_n}(t_n-s)^{l_1^{'}}e^{2|z|(t_n-s)}ds|dz|\\
			&\leq C\tau^{l_1^{''}-2(1-\sigma)\al_0+1}\int_{\Gamma_{\theta,\kappa}\backslash\Gamma_{\theta,\kappa}^\tau}|z|^{l_1^{''}-l_1^{'}-1}|dz|\leq C\tau^{l_1^{'}-2(1-\sigma)\al_0+1}\leq C\tau^{2(\sigma\al_0-\al_0+1)-\eps}. %这里sigma=0的话要求2(1-al_0)-eps>0
			%&\hspace{3in}(l_1'>-1,~l_1''-l_1'<0)
		\end{aligned} 
    \end{equation*}
   
    % Due to that $2(1-\sigma)\alpha_0-1<l_1^{''}<l_1^{'}<1$, we need to restrict $(1-\sigma)\alpha_0<1$ and let $l_1^{'}=1-\eps$, which lead to $\zeta_{1,1}\leq C\tau^{2\sigma\alpha_0-2\alpha_0+2-\eps}$. %C\tau^{2-2(1-\sigma)\al_0-\eps}=

   Now we introduce the auxiliary operator
    \begin{equation}\label{def:Etau}
        \begin{aligned}
            \mathcal{E}_{\tau}(z):=z^{\al_0-1}(z^{\al_0}+A)^{-1}-(\delta_{\tau}(e^{-z\tau}))^{\al_0-1}((\delta_{\tau}(e^{-z\tau}))^{\al_0}+A)^{-1}\frac{z\tau}{e^{z\tau}-1}. 
        \end{aligned}
    \end{equation}
    Using Cauchy-Schwarz inequality and set $l_2^{'}<1$ to estimate $\zeta_{1,2}$ by
	\begin{equation*}
		\begin{aligned}
			\zeta_{1,2}&\leq \cE\bigg[\Big\|\int_{0}^{t_n}\int_{\Gamma_{\theta,\kappa}^\tau}e^{z(t_n-s)}\left(\widetilde{F}(z)-\widetilde{\overline{F}}(z)\right)dzg'(s)*Au(s)ds\Big\|^2\bigg]\\
			&\leq C \cE\bigg[\Big(\int_{0}^{t_n}\Big\|\int_{\Gamma_{\theta,\kappa}^\tau}e^{z(t_n-s)}\mathcal{E}_\tau(z)dzg'(s)*Au(s)\Big\|ds\Big)^2\bigg]\\
			&\leq C \cE\bigg[\int_{0}^{t_n}(t_n-s)^{l_2^{'}}\int_{\Gamma_{\theta,\kappa}^\tau}|e^{2z(t_n-s)}||z|^{l_2^{''}}|dz|\int_{\Gamma_{\theta,\kappa}^\tau}|z|^{-l_2^{''}}\Big\|\mathcal{E}_\tau(z)g'(s)*Au(s)\Big\|^2|dz|ds\bigg]\\
			&= C \cE\bigg[\int_{0}^{t_n}(t_n-s)^{l_2^{'}}\int_{\Gamma_{\theta,\kappa}^\tau}|e^{2z(t_n-s)}||z|^{l_2^{''}}|dz|\\
            &\hspace{1.3in}\times\int_{\Gamma_{\theta,\kappa}^\tau}|z|^{-l_2^{''}}\left\|\mathcal{E}_\tau(z)A^{1-\sigma}g'(s)*A^\sigma u(s)\right\|_\mathbb{H}^2|dz|ds\bigg].
   		\end{aligned}
	\end{equation*}
   According to \cite[(3.7)]{NieSunDenSiamJNA},  we have $\left\|\mathcal{E}_\tau A^{\beta}\right\|\leq C\tau|z|^{\beta\al_0}$ for $0\le\beta\le1$. Consequently, for \(0\leq \sigma\leq 1\), we have
$
            \left\|\mathcal{E}_\tau A^{1-\sigma}(g'*A^{\sigma}u)(s)\right\|^2
            % &\leq C\|g'*A^\sigma u(s)\|^2\cdot\left\|\mathcal{E}_\tau A^{1-\sigma}\right\|^2\\
            \leq C\|g'*A^\sigma u(s)\|^2\cdot |z|^{2(1-\sigma)\al_0}\tau^2.
$ We combine this and set $l_2^{'}=1-\eps,~l_2^{''}=1-\eps/2$ (then $-1<l_2^{'}<1,~l_2^{''}<2(1-\sigma)\alpha_0+1$ and $l_2'<l_2^{''}$ with $0<\eps\ll 2(\sigma\al_0-\al_0+1)$) to estimate $\zeta_{1,2}$ by
   	\begin{equation*}
		\begin{aligned}
           \zeta_{1,2}&\leq C \cE\bigg[\int_{0}^{t_n}(t_n-s)^{l_2^{'}}\int_{\Gamma_{\theta,\kappa}^\tau}|e^{2z(t_n-s)}||z|^{l_2^{''}}|dz|\\
                    &\hspace{1.2in}\times\int_{\Gamma_{\theta,\kappa}^\tau}|z|^{-l_2^{''}+2(1-\sigma)\al_0}\tau^2\left\|g'*A^\sigma u(s)\right\|^2dzds\bigg]\\
                    %%%%%%%%%%%%
        			% &\leq C\tau^2 \int_{0}^{t_n}(t_n-s)^{l_2^{'}}\int_{\Gamma_{\theta,\kappa}^\tau}|e^{2z(t_n-s)}||z|^{l_2^{''}}|dz|\\
           %          &\hspace{0.3in}\times\int_{\Gamma_{\theta,\kappa}^\tau}s^{\al_0-\eps}\int_{0}^{s}|g'(s-r)|\cE\left[\left\|A^\sigma u(r)\right\|_\mathbb{H}^2\right]dr|z|^{-l_2^{''}+2(1-\sigma)\al_0}|dz|ds\\
        			&\leq C\tau^2 \int_{0}^{t_n}(t_n-s)^{l_2^{'}}\int_{\Gamma_{\theta,\kappa}^\tau}|e^{2z(t_n-s)}||z|^{l_2^{''}}|dz|\int_{\Gamma_{\theta,\kappa}^\tau}|z|^{-l_2^{''}+2(1-\sigma)\al_0}|dz|ds\\
                    % &\leq C\tau^2 \int_{0}^{t_n}(t_n-s)^{l_2^{'}}\int_{\Gamma_{\theta,\kappa}^\tau}|e^{2z(t_n-s)}||z|^{l_2^{''}}|dz|\\
           %          &\hspace{0.3in}\times\int_{\Gamma_{\theta,\kappa}^\tau}s^{\al_0-\eps}\int_{0}^{s}|g'(s-r)|\cE\left[\left\|A^\sigma u(r)\right\|_\mathbb{H}^2\right]dr|z|^{-l_2^{''}+2(1-\sigma)\al_0}|dz|ds\\
        			% &\leq C\tau^2 \int_{0}^{t_n}(t_n-s)^{l_2^{'}}\int_{\Gamma_{\theta,\kappa}^\tau}|e^{2z(t_n-s)}||z|^{l_2^{''}}|dz|\int_{\Gamma_{\theta,\kappa}^\tau}|z|^{-l_2^{''}+2(1-\sigma)\al_0}|dz|ds\\
        			% &\leq C\tau^2 \tau^{l_2^{''}-2(1-\sigma)\al_0-1}\int_{0}^{t_n}(t_n-s)^{l_2^{'}}\int_{\Gamma_{\theta,\kappa}^\tau}|e^{2z(t_n-s)}||z|^{l_2^{''}}|dz|ds\\%~~(-l_2^{''}+2(1-\sigma)\alpha_0>-1)\\
        			% &\leq C\tau^2 \tau^{l_2^{''}-2(1-\sigma)\al_0-1}\int_{\Gamma_{\theta,\kappa}^\tau}|z|^{l_2^{''}}\int_{0}^{t_n}(t_n-s)^{l_2^{'}}|e^{z(t_n-s)}|ds|dz|\\\\要不懂就看这行
        			&\leq C\tau^2 \tau^{l_2^{''}-2(1-\sigma)\al_0-1}\int_{\Gamma_{\theta,\kappa}^\tau}|z|^{l_2^{''}-l_2^{'}-1}|dz|\leq C\tau^{l_2^{'}-2(1-\sigma)\al_0+1}\leq C\tau^{2(\sigma\al_0-\al_0+1)-\eps}.%~~\qquad(l_2^{''}-l_2^{'}-1>-1)
		\end{aligned}
	\end{equation*}
	We finish the proof.
\end{proof}	

\begin{lmm}\label{lmm:bound_zeta2}
  Assume that $\|A^{-\rho}\|_{\mathcal{L}_{2}^{0}}<\infty$ for $\rho\in[0,\frac{H}{\alpha_0})\cap[0,1]$. We have $\zeta_2\le C(\tau^{2\gamma}+\tau\sum_{k=1}^{n}t_{n-k+1}^{\al_*-1-\tilde{\eps}}\cE\left[\left\|u(t_k)-u^k\right\|^2\right])$, where $\gamma\in(0,H-\rho\al_0)$.
\end{lmm}
\begin{proof}
   We apply \eqref{eq:gpconAu} and the Cauchy–Schwarz inequality to estimate $\zeta_2$,
	\begin{equation*}%\label{split:v2}
		\begin{aligned}
			\zeta_2
   %          &=\cE\left[\left\|\overline{F}(t_n)*I_\tau Au^n-\overline{F}(t_n)*g'(t_n)*Au(t_n)\right\|_\mathbb{H}^2\right]\\
			% &=\cE\left[\left\|\overline{F}(t_n)*[I_\tau Au^n-I_\tau Au(t_n)+I_\tau Au(t_n)-g'(t_n)*Au(t_n)]\right\|_\mathbb{H}^2\right]\\
			% &\leq 2\cE\left[\left\|\overline{F}(t_n)*[I_\tau Au^n-I_\tau Au(t_n)]\right\|_\mathbb{H}^2\right]+2\cE\left[\left\|\overline{F}(t_n)*[J_n+\hat{J}_n]\right\|_\mathbb{H}^2\right]\\&
			\leq2\cE\left[\left\|\overline{F}(t_n)*[I_\tau Au^n-I_\tau Au(t_n)]\right\|^2\right]
			 +4\cE\left[\left\|\overline{F}(t_n)*J_n\right\|^2\right]+4\cE\left[\left\|\overline{F}(t_n)*\hat{J}_n\right\|^2\right].
		\end{aligned}
	\end{equation*}
	
	Let $0<\eps\ll 1-\al_0$ and set $l_3^{'}=1-\eps,~l_3^{''}=-1+2\eps$, we have $l_3^{'}+l_3^{''}>0,~1>l_3^{'},~1>l_3^{''},~-1>2\al_0-2+l_3^{''}$. Apply Minkowski's inequality, \eqref{Bound_FA1_sigma} and Theorem \ref{thm_holder_reg} to estimate
	\begin{equation*}%\label{Bound_Jn}
		\begin{aligned}
			&\cE\left[\left\|\overline{F}(t_n)*J_n\right\|^2\right]
            % &=\cE\left[\int_D\left(\int_{0}^{t_n}\overline{F}(t_n-s)J_n(s)ds\right)^2dx \right]\\
			% &\leq \cE\left[\left(\int_{0}^{t_n}\left(\int_D\left(\overline{F}(t_n-s)J_n(s)\right)^2dx\right)^{\f{1}{2}}ds\right)^2\right]\\%~~(Min.Ineq.)
			\leq\cE\bigg[\Big(\int_{0}^{t_n}\left\|\overline{F}(t_n-r)J_n(r)\right\|dr\Big)^2\bigg]\\
            &\leq C\cE\bigg[\Big(\int_{0}^{t_n}\Big\|\overline{F}(t_n-r)A\sum_{k=1}^n\int_{t_{k-1}}^{t_k}(t_n-s)^{\al_*-1}|\ln(t_n-s)|ds\int_{t_{k-1}}^{t_k}|\partial_\theta u(\cdot,\theta)|d\theta \Big\|dr\Big)^2\bigg]\\
			% &\leq C\cE\bigg[\bigg(\int_{0}^{t_n}\bigg\|\overline{F}(t_n-r)A\bigg\|\\
   %          &\hspace{2cm}\times\bigg\|\sum_{k=1}^n\int_{t_{k-1}}^{t_k}(t_n-s)^{\al_*-1}|\ln(t_n-s)|ds\left|u(\cdot,t_k)-u(\cdot,t_{k-1})\right| \bigg\|_{{\mathbb{H}}}dr\bigg)^2\bigg]\\
			% &\leq C\cE\bigg[\int_{0}^{t_n}(t_n-r)^{l_3'}\bigg\|\overline{F}(t_n-r)A\bigg\|^2dr\\
   %          &\hspace{2cm}\times\left(\sum_{k=1}^n\int_{t_{k-1}}^{t_k}(t_n-s)^{\al_*-1-\varepsilon}ds\left\|u(\cdot,t_k)-u(\cdot,t_{k-1}) \right\|_{{\mathbb{H}}}\right)^2\bigg]\\
			&\leq C\Big(\int_{0}^{t_n}(t_n-r)^{l_3'}\|\overline{F}(t_n-r)A\|^2dr\Big)\\
            &\hspace{2cm}\times\bigg(\sum_{k=1}^n\Big(\int_{t_{k-1}}^{t_k}(t_n-s)^{\al_*-1-\varepsilon}ds\Big)^2\cE\left[\left\|u(\cdot,t_k)-u(\cdot,t_{k-1}) \right\|^2\right]\bigg)\\
			% &\leq C\tau^{2\gamma}\left(\int_{0}^{t_n}(t_n-r)^{l_3'}\left(\int_{\Gamma_{\theta,\kappa}^\tau}|e^{z(t_n-r)}|\cdot\left(\tau|z|^{\alpha_0}+|z|^{\alpha_0-1}\right)|dz|\right)^2dr\right)\\
   %          &\hspace{2cm}\times\left(\sum_{k=1}^n\int_{t_{k-1}}^{t_k}(t_n-s)^{\al_*-1-\varepsilon}ds\right)^2\\
			&\leq C\tau^{2\gamma}\bigg(\int_{0}^{t_n}(t_n-r)^{l_3'}\int_{\Gamma_{\theta,\kappa}^\tau}|z|^{-l_3^{''}}|e^{2z(t_n-r)}|dz\cdot\int_{\Gamma_{\theta,\kappa}^\tau}|z|^{2\alpha_0-2+l_3^{''}}|dz|dr\bigg)\\
			&\leq C\tau^{2\gamma}\int_{0}^{t_n}(t_n-r)^{l_3'+l_3^{''}-1}dr\leq C\tau^{2\gamma},
		\end{aligned}%这里之所以要出来C\tau^{2\gamma+2\eps}是由于gamma<H-rho\al_0,但是我们想证出H-rho\al_0阶
	\end{equation*}
    where $\gamma\in(0,H-\rho\al_0)$ provided in Theorem \ref{thm_holder_reg}. Similarly, taking $l_4^{'}$<1, we have
%	\begin{equation}\label{BoundFqJ2}
%		\begin{aligned}
%			\left\|\overline{F}(t_n-s)\right\|&=\left\|\f{1}{2\pi i}\int_{\Gamma_{\theta,\kappa}^\tau}e^{z(t_n-s)}\left(\delta_\tau(e^{-z\tau})+(\delta_\tau(e^{-z\tau}))^{1-\al_0}A\right)^{-1}\f{z\tau}{e^{z\tau}-1}dz\right\|\\
%%			&\leq C\left\|\int_{\Gamma_{\theta,\kappa}^\tau}e^{z(t_n-s)}\left(\delta_\tau(e^{-z\tau})+(\delta_\tau(e^{-z\tau}))^{1-\al_0}A\right)^{-1}\f{z\tau}{e^{z\tau}-1}dz\right\|\\
%			&\leq C\int_{\Gamma_{\theta,\kappa}^\tau}e^{z(t_n-s)}\left\|\left(\delta_\tau(e^{-z\tau})+(\delta_\tau(e^{-z\tau}))^{1-\al_0}A\right)^{-1}\f{z\tau}{e^{z\tau}-1}\right\|dz\\
%		\end{aligned}
%	\end{equation}
    % where $-1<l_3'<-l_3''$ and $1-2\alpha_0<l_3''$, and we set $l_3'=2\alpha_0-1-\varepsilon
    % $.%在平板上有计算细节
	% we need to take $l_3'=-1-\varepsilon,~l_3''=1-\varepsilon$ to satisfy $-1<l_3'<1,~1-2\alpha_0<l_3',~l_3'+l_3''<0$. 
	\begin{equation}\label{BoundJconvHatJ}
		\begin{aligned}
			\cE\left[\left\|\overline{F}(t_n)*\hat{J}_n\right\|^2\right] &\leq\cE\bigg[\Big(\int_{0}^{t_n}\left\|\overline{F}(t_n-s)\hat{J}_n(s)\right\|ds\Big)^2\bigg]\\
            &\leq C \cE\bigg[\int_{0}^{t_n}(t_n-s)^{l_4'}\sum_{i=1}^{4}\Big\|\overline{F}(t_n-s)\hat{J}_{n}^{i}\Big\|^2ds\bigg].
		\end{aligned}
	\end{equation}
    Now we take $\sigma=\min\{1-\rho,\f{H}{\al_0}-\rho-\eps\}$. Taking $l_4^{'}=1-\eps,~ l_4^{''}=2\eps-1$. Consequently, $l_4'<1,~l_4^{''}<1,~0<l_4'+l_4'',~2\al_0(1-\sigma)+l_4^{''}<1$ with $0<\eps\ll 1-\al_0(1-\sigma)$, and we utilize \eqref{Bound_FA1_sigma} to obtain
	\begin{equation*}
		\begin{aligned}
		&\int_{0}^{t_n}(t_n-s)^{l_4'}\left\|\overline{F}(t_n-s)A^{1-\sigma}\right\|^2 ds
		=\int_{0}^{t_n}(t_n-s)^{l_4'}\bigg(\int_{\Gamma_{\theta,\kappa}^\tau}|e^{z(t_n-s)}|\cdot|z|^{\alpha_0(1-\sigma)-1}|dz|\bigg)^2 ds\\
		\leq& C\int_{0}^{t_n}(t_n-s)^{l_4'}\int_{\Gamma_{\theta,\kappa}^\tau}|z|^{-l_4^{''}}|e^{2z(t_n-s)}||dz|\cdot\int_{\Gamma_{\theta,\kappa}^\tau}|z|^{2\alpha_0(1-\sigma)+l_4^{''}-2}|dz| ds\\
        \leq& C\int_{0}^{t_n}(t_n-s)^{l_4'+l_4^{''}-1}ds
        \leq C.
		\end{aligned}
	\end{equation*}
     We split $\hat{J}_n$ into four parts $\hat{J}_n=\hat{J}_{n}^{1}+\hat{J}_{n}^{2}+\hat{J}_{n}^{3}+\hat{J}_{n}^{4},$ where
    \begin{equation*}
		\hat{J}_{n}^{1}:=\sum_{k=1}^{n}\int_{t_{k-1}}^{t_{k}}\left[\frac{(t_{n}-s)^{\alpha(t_{n}-s)-1}}{\Gamma(\alpha(t_{n}-s))}-\frac{(t_{n}-s)^{\alpha(\tau_{n,k})-1}}{\Gamma(\alpha(\tau_{n,k}))}\right]H(t_{n}-s)A u(\boldsymbol{x},s)ds,
	\end{equation*}
     \begin{equation*}
		\hat{J}_{n}^{2}:=\sum_{k=1}^{n}\int_{t_{k-1}}^{t_{k}}\frac{(t_{n}-s)^{\alpha(\tau_{n,k})-1}}{\Gamma(\alpha(\tau_{n,k}))}\Big[H(t_{n}-s)-\breve{H}_{n,k}(s)\Big]A u(\boldsymbol{x},s)ds,
	\end{equation*}
      \begin{equation*}
		\hat{J}_{n}^{3}:=\sum_{k=1}^{n}\int_{t_{k-1}}^{t_k}\left(\int_{0}^{1}\left(\f{(t_n-s)^{\al(z(t_n-s))-1}}{\Gamma(\al(z(t_n-s)))}-\f{(t_n-s)^{\al(z\tau_{n,k})-1}}{\Gamma(\al(z\tau_{n,k}))}\right)U(z,t_n-s)dz\right)Au(x,s)ds,
	\end{equation*}
    \begin{equation*}
		\hat{J}_{n}^{4}:=\sum_{k=1}^{n}\int_{t_{k-1}}^{t_k}\left(\int_{0}^{1}\f{(t_n-s)^{\al(z\tau_{n,k})-1}}{\Gamma(\al(z\tau_{n,k}))}\left(U(z,t_n-s)-\breve{U}_{n,k}(z,s)\right)dz\right) Au(x,s)ds.
	\end{equation*}
	Use \eqref{Bound_FA1_sigma} and Theorem \ref{thm_reg} to estimate
	%According to \cite{Zheng}, we have
	\begin{equation*}%\label{BoundJconvHatJ1}
		\begin{aligned}
			&\cE\bigg[\int_{0}^{t_n}(t_n-s)^{l_4'}\Big\|\overline{F}(t_n-s)\hat{J}_{n}^{1}\Big\|^2ds\bigg]\\
			% =&\cE\bigg[\int_{0}^{t_n}(t_n-s)^{l_4'}\bigg\|\overline{F}(t_n-r)\\
   %          &\qquad\times\left(\sum_{k=1}^{n}\int_{t_{k-1}}^{t_{k}}\left[\frac{(t_{n}-s)^{\alpha(t_{n}-s)-1}}{\Gamma(\alpha(t_{n}-s))}-\frac{(t_{n}-s)^{\alpha(\tau_{n,k})-1}}{\Gamma(\alpha(\tau_{n,k}))}\right]H(t_{n}-s)A u(\boldsymbol{x},s)ds\right)\bigg\|^2dr\bigg]\\
			\leq&C\cE\bigg[\int_{0}^{t_n}(t_n-s)^{l_4'}\Big\|\overline{F}(t_n-s)\\
            &\qquad\times\Big(	\sum_{k=1}^{n}\int_{t_{k-1}}^{t_{k}}\left|\ln(t_{n}-s)\right|\int_{s}^{t_{k}}\Big|\partial_{z}\Big(\frac{(t_{n}-s)^{\alpha(t_{n}-z)-1}}{\Gamma(\alpha(t_{n}-z))}\Big)\Big|dzA u(\boldsymbol{x},s)ds\Big)\Big\|^2ds\bigg]\\
			% \leq&C\tau^2\cE\left[\int_{0}^{t_n}(t_n-s)^{l_4'}\left\|\overline{F}(t_n-r)\left(	\sum_{k=1}^{n}\int_{t_{k-1}}^{t_{k}}(t_{n}-s)^{\alpha_*-1-\varepsilon}A u(\boldsymbol{x},s)ds\right)\right\|^2dr\right]\\
			\leq&C\tau^2\Big(\int_{0}^{t_n}(t_n-s)^{l_4'}\left\|\overline{F}(t_n-s)A^{1-\sigma}\right\|^2 ds\Big)\cE\bigg[\Big\|\Big(\int_{0}^{t_{n}}(t_{n}-s)^{\alpha_*-1-\varepsilon}A^{\sigma} u(\boldsymbol{x},s)ds\Big)\Big\|^2\bigg]\\
			\leq&C\tau^{2}\bigg(\int_{0}^{t_{n}}(t_{n}-s)^{2\alpha_*-2-2\varepsilon}(t_{n}-s)^{1-\alpha_*}ds\int_{0}^{t_n}\cE\left[\left\|A^{\sigma} u(\boldsymbol{x},s)\right\|^2\right](t_n-s)^{\alpha_*-1}ds\bigg)
			\leq C\tau^{2}.
		\end{aligned}
	\end{equation*}
	% where $\al_*=\inf\limits_{t\in[0,T]}\al(t)$. 
    Here we assume $t_n - s \leq 1$, and the case $t_n - s > 1$ follows by a similar argument. Under the assumption that $\al'(t),~\al''(t)$ is bounded over $[0,T]$, each of the remaining terms on the right-hand side of \eqref{BoundJconvHatJ} can be bounded by $C\tau^2$, then we obtain $\cE\left[\left\|\overline{F}(t_n)*\hat{J}_n\right\|^2\right]\leq  C\tau^{2}.$ Applying \eqref{Bound_FA1_sigma}, \eqref{ineq:disint} and Lemma \ref{lem:bnkbound}, we estimate

	\begin{equation*}%\label{boundFcuerr}
		\begin{aligned}
			&\cE\left[\left\|\overline{F}(t_n)*[I_\tau Au^n-I_\tau Au(t_n)]\right\|^2\right]\\
			\leq& C\cE\bigg[\Big(\int_0^{t_n}\left\|\overline{F}(t_n-s)A\right\|\cdot\left\|I_\tau u^n-I_\tau u(t_n)\right\|ds\Big)^2\bigg]\\
			% \leq& C\cE\left[\left(\int_0^{t_n}\int_{\Gamma_{\theta,\kappa}^\tau}e^{|z|(t_n-s)}(\tau|z|^{\al_0}+|z|^{\al_0-1})|dz|\cdot\left\|I_\tau u^n-I_\tau u(t_n)\right\|_\mathbb{H}ds\right)^2\right]\\
			\leq& C\cE\bigg[\Big(\int_0^{t_n}\int_{\Gamma_{\theta,\kappa}^\tau}e^{|z|(t_n-s)}|z|^{\al_0-1}|dz|ds\Big)^2\cdot\left\|I_\tau u^n-I_\tau u(t_n)\right\|^2\bigg]\\
%			\leq&C\tau^{2-2\al_0}\tau^{1-2\al_0}\cE\left[\left\|I_\tau u^n-I_\tau u(t_n)\right\|_\mathbb{H}^2\right]\\
			% =&C\cE\left[\left\|\sum_{k=1}^{n}b_{n,k}(u(t_k)-u^k)\right\|_\mathbb{H}^2\right]\\
			% \leq&C\tau^{2-2\al_0}\sum_{k=1}^{n}|b_{n,k}|\sum_{k=1}^{n}|b_{n,k}|\cE\left[\left\|u(t_k)-u^k\right\|_\mathbb{H}^2\right]	
            \leq &C\sum_{k=1}^{n}|b_{n,k}|\cdot\cE\left[\left\|u(t_k)-u^k\right\|^2\right]
            \leq C\tau\sum_{k=1}^{n}t_{n-k+1}^{\al_*-1-\tilde{\eps}}\cE\left[\left\|u(t_k)-u^k\right\|^2\right].
		\end{aligned}
	\end{equation*}
    We have $\zeta_2\le C(\tau^{2\gamma}+\tau\sum_{k=1}^{n}t_{n-k+1}^{\al_*-1-\tilde{\eps}}\cE\left[\left\|u(t_k)-u^k\right\|^2\right])$.
\end{proof}

Next, we establish the error estimate for the semidiscrete scheme \eqref{scheme}.
\begin{thrm}\label{time_dic_err}
	Assume that $\|A^{-\rho}\|_{\mathcal{L}_{2}^{0}}<\infty$ for $\rho\in[0,\frac{H}{\alpha_0})\cap[0,1]$. Let $u(t_n)$ and $u^n$ be the solutions of problem \eqref{model_3_sto} and scheme \eqref{scheme}, respectively. We have error estimate $\cE\left[\|u(t_n)-u^n\|^2\right]\leq C\tau^{2H-2\rho\al_0-\eps}.$
\end{thrm}
\begin{proof}
	Combining Lemmas \ref{lmm:bound_zeta1} and \ref{lmm:bound_zeta2} with \eqref{error_semidiscrete}, we obtain
	\begin{equation*}%\label{bound_v2}
		\begin{aligned}
		\cE\left[\|u(t_n)-u^n\|^2\right]\leq &C\tau^{\min\{2H-2\rho\al_0,2-2\alpha_0(1-\sigma)-\varepsilon,2\gamma\}}
        +C\tau\sum_{k=1}^{n}t_{n-k+1}^{\al_*-1-\tilde{\eps}}\cE\left[\left\|u(t_k)-u^k\right\|^2\right].
		\end{aligned}
	\end{equation*}
    Thus we utilize $\sigma=\min\left\{1-\rho,\frac{H}{\alpha_0}-\rho-\varepsilon\right\}$ and $\gamma=H-\rho\al_0-\eps$ in Theorems \ref{thm_reg} and \ref{thm_holder_reg}, and apply Lemma \ref{lmm:disGrowIneq} to finish the proof.
\end{proof}

\subsection{Error estimate for scheme (\ref{scheme_v})}
To prove error estimate for the scheme \eqref{scheme_v}, we first present the following lemma.
\begin{lmm}\label{lmm:Fn-breFnm1}
Let $F(t)$ and $\breve{F}(t)$ be defined as \eqref{def:F} and \eqref{def:breF}, respectively. Then  $\|F(t_n)-\breve{F}(t_{n-1})\|\leq Ct_n^{-1}\tau$ for $n\geq 1.$
% \begin{equation*}
%     \|F(t_n)-\breve{F}(t_n-\tau)\|\leq Ct_n^{-1}\tau,~n\geq 1.
% \end{equation*}
\end{lmm}
\begin{proof}
For $n=1$, we have
        \begin{equation}\label{eq:Ftau-F0}
		F(\tau)-\breve{F}(0)
		=\frac{1}{2\pi i}\int_{\Gamma_{\theta,\kappa}\setminus\Gamma_{\theta,\kappa}^\tau}e^{z\tau}z^{\al_0-1}(z^{\al_0}+A)^{-1}dz
		+\frac{1}{2\pi i}\int_{\Gamma_{\theta,\kappa}^\tau}e^{z\tau}\widetilde{\mathcal{W}}_\tau(z)dz,
	\end{equation}
	where $\widetilde{\mathcal{W}}_\tau(z):=z^{\al_0-1}(z^{\al_0}+A)^{-1}-e^{-z\tau}(\delta_\tau(e^{-z\tau}))^{\al_0-1}\big((\delta_\tau(e^{-z\tau}))^{\al_0}+A\big)^{-1}.$
	By Lemma \ref{thm_contour}, the first term on the right-hand side of \eqref{eq:Ftau-F0} is bounded by $C$.
	% By Lemma \ref{thm_contour} with $\beta=0$, $\|z^{\al_0-1}(z^{\al_0}+A)^{-1}\|\le C|z|^{-1}$ for $z\in\Sigma_\theta$. On $\Gamma_{\theta,\kappa}\setminus\Gamma_{\theta,\kappa}^\tau$, i.e. $|z|=r\ge R:=\pi/(\tau\sin\theta)$ with $\arg z=\pm\theta$, we have $|e^{z\tau}|=e^{-cr\tau}$ with $c:=-\cos\theta>0$. Hence
	% \begin{equation*}
	% 	\|\mathrm{I}\|\le C\int_{R}^{\infty}e^{-cr\tau}r^{-1}dr=CE_1(cR\tau),
	% \end{equation*}
	% where $E_1(y)=\int_1^\infty e^{-yt}t^{-1}dt$ is the exponential integral. Since $R\tau=\pi/\sin\theta$ is a constant independent of $\tau$, $E_1(cR\tau)=E_1(c\pi/\sin\theta)$ is likewise a constant independent of $\tau$. Therefore
	% \begin{equation*}
	% 	\|\mathrm{I}\|\le C.
	% \end{equation*}
	Recall $\mathcal{E}_\tau(z)$ defined in \eqref{def:Etau} satisfies $\|\mathcal{E}_\tau(z)\|\le C\tau$ for all $z\in\Gamma_{\theta,\kappa}^\tau$. Write
	\begin{equation*}
		\widetilde{\mathcal{W}}_\tau(z)=\mathcal{E}_\tau(z)+(\delta_\tau(e^{-z\tau}))^{\al_0-1}\big((\delta_\tau(e^{-z\tau}))^{\al_0}+A\big)^{-1}\left(\frac{z\tau}{e^{z\tau}-1}-e^{-z\tau}\right).
	\end{equation*}
	Combining Lemma \ref{thm_contour} and Lemma \ref{lem:delta}, we have $\big\|(\delta_\tau(e^{-z\tau}))^{\al_0-1}\big((\delta_\tau(e^{-z\tau}))^{\al_0}+A\big)^{-1}\big\|\le C|z|^{-1}.$
	Since  for all $z\in\Gamma_{\theta,\kappa}^\tau$, $\left|\frac{z\tau}{e^{z\tau}-1}-e^{-z\tau}\right|\le C\tau|z|$, then%z\in\Gamma_{\theta,\kappa}^\tau
	\begin{equation*}
		\left\|(\delta_\tau(e^{-z\tau}))^{\al_0-1}\big((\delta_\tau(e^{-z\tau}))^{\al_0}+A\big)^{-1}\left(\frac{z\tau}{e^{z\tau}-1}-e^{-z\tau}\right)\right\|\le C|z|^{-1}\tau|z|=C\tau.
	\end{equation*}
	Hence $\|\widetilde{\mathcal{W}}_\tau(z)\|\le\|\mathcal{E}_\tau(z)\|+C\tau\le C\tau$ for all $z\in\Gamma_{\theta,\kappa}^\tau,$ then $\|F(t_1)-\breve{F}(t_0)\|\le C$. For $n\geq2$, 
\begin{equation}\label{ineq:Fn-breFnm1}
    \begin{aligned}
        \|F(t_{n})-\breve{F}(t_{n-1})\|
        % &=\|F(t_{n})-F(t_{n-1})+F(t_{n-1})-\breve{F}(t_{n-1})\|\\
        \leq \|F(t_{n})-F(t_{n-1})\|+\|F(t_{n-1})-\breve{F}(t_{n-1})\|,
    \end{aligned}
\end{equation}
then we use Lemma \ref{thm_contour} to estimate
% \begin{equation*}
%     \begin{aligned}
%         \|F(t_{n})-F(t_{n-1})\|\leq&C\left\|\int_{\Gamma_{\theta,\kappa}}\left(e^{zt_n}-e^{zt_{n-1}}\right)z^{\al_0-1}\left(z^{\al_0}+A\right)^{-1}dz\right\|\\
%         \leq&C\tau^\gamma\int_{\Gamma_{\theta,\kappa}}|e^{zt_{n-1}}| |z|^{\gamma-1}|dz|\leq Ct_{n-1}^{-\gamma}\tau^\gamma.
%     \end{aligned}
% \end{equation*}
\begin{equation}\label{ineq:Fn-Fns1}
    \begin{aligned}
        \|F(t_{n})-F(t_{n-1})\|\leq&C\Big\|\int_{\Gamma_{\theta,\kappa}}\left(e^{zt_n}-e^{zt_{n-1}}\right)z^{\al_0-1}\left(z^{\al_0}+A\right)^{-1}dz\Big\|\\
        \leq&C\tau\int_{\Gamma_{\theta,\kappa}}|e^{zt_{n-1}}| |dz|\leq Ct_{n-1}^{-1}\tau\le Ct_{n}^{-1}\tau.
    \end{aligned}
\end{equation}
Using Lemmas \ref{thm_contour} and \ref{lem:delta} yields
\begin{equation}\label{ineq:Ftnm1-breFnm1}
    \begin{aligned}
        &\|F(t_{n-1})-\breve{F}(t_{n-1})\|\\
        % =\left\|\int_{\Gamma_{\theta,\kappa}}e^{zt_{n-1}}z^{\al_0-1}\left(z^{\al_0}+A\right)^{-1}dz-\int_{\Gamma_{\theta,\kappa}^\tau}e^{zt_{n-1}}z^{\al_0-1}\left(z^{\al_0}+A\right)^{-1}dz\right\|\\
        % &\hspace{1cm}+\left\|\int_{\Gamma_{\theta,\kappa}^\tau}e^{zt_{n-1}}z^{\al_0-1}\left(z^{\al_0}+A\right)^{-1}dz-\int_{\Gamma_{\theta,\kappa}^\tau}e^{zt_{n-1}}\left(\delta_\tau(e^{-z\tau})+(\delta_\tau(e^{-z\tau}))^{1-\al_0}A\right)^{-1}dz\right\|
        \leq& \Big\|\int_{\Gamma_{\theta,\kappa}\backslash\Gamma_{\theta,\kappa}^\tau}e^{zt_{n-1}}z^{\al_0-1}\left(z^{\al_0}+A\right)^{-1}dz\Big\|\\
        &+\Big\|\int_{\Gamma_{\theta,\kappa}^\tau}e^{zt_{n-1}}\left(\left(z+z^{1-\al_0}A\right)^{-1}-\left(\delta_\tau(e^{-z\tau})+(\delta_\tau(e^{-z\tau}))^{1-\al_0}A\right)^{-1}\right)dz\Big\|\\
        \leq&C\int_\frac{\pi}{\tau\sin(\theta)}^{\infty}e^{|z|cos(\theta)t_{n-1}}|z|^{-1}|dz|+C\tau
        \leq CE_1\left(-\frac{\pi t_{n-1}\cos\theta}{\tau\sin\theta}\right)+C\tau
        \leq Ct_{n-1}^{-1}\tau\le Ct_{n}^{-1}\tau,
    \end{aligned}
\end{equation}
where 
\begin{equation*}
    \begin{aligned}
        \Big\|\int_{\Gamma_{\theta,\kappa}^\tau}e^{zt_{n-1}}\left(\left(z+z^{1-\al_0}A\right)^{-1}-\left(\delta_\tau(e^{-z\tau})+(\delta_\tau(e^{-z\tau}))^{1-\al_0}A\right)^{-1}\right)dz\Big\|
        % =&\int_{\Gamma_{\theta,\kappa}^\tau}e^{zt_{n-1}}\left\|\left(z+z^{1-\al_0}A\right)^{-1}-\left(\delta_\tau(e^{-z\tau})+(\delta_\tau(e^{-z\tau}))^{1-\al_0}A\right)^{-1}\right\|dz\\
        % \leq &C\tau\int_{\Gamma_{\theta,\kappa}^\tau}e^{zt_{n-1}} dz
        \leq C\tau,
    \end{aligned}
\end{equation*}
and $E_1(\cdot)$ is the exponential integral defined by $E_1(z)=\int_1^\infty e^{-zt}/tdt$ \cite[Definition 5.1.1]{Abramowitz1948handbook}. Incorporating \eqref{ineq:Fn-Fns1} and \eqref{ineq:Ftnm1-breFnm1} into \eqref{ineq:Fn-breFnm1} completes the proof.
\end{proof}

According to \eqref{v_ref} and \eqref{vn_numer}, we obtain the following estimate.
\begin{thrm}\label{time_dic_err_v}
	Assume that $\|A^{-\rho}\|_{\mathcal{L}_{2}^{0}}<\infty$, $G_0\in\mathbb{H}$, and $\rho\in[0,\frac{H}{\alpha_0})\cap[0,1]$. We have the error estimate $\|v(t_n)-v^n\|\leq C\tau^{1-\eps}+Ct_n^{-1}\tau, ~~n\geq 1.$
\end{thrm}
\begin{proof}
Applying \eqref{v_ref} and \eqref{vn_numer} yields
\begin{equation*}%\label{ineq:vtn-vn}
    \begin{aligned}
        \|v(t_n)-v^n\|&\leq \|(F(t_n)-\breve{F}(t_{n-1}))G_0\|
        +2\left\|\overline{F}(t_n)*I_\tau Av^n-F(t_n)*g'(t_n)*Av(t_n)\right\|.
    \end{aligned}
\end{equation*}
By using Lemma \ref{lmm:Fn-breFnm1}, we estimate
$\|(F(t_n)-\breve{F}(t_{n-1}))G_0\|\leq \|G_0\| \cdot \|F(t_n)-\breve{F}(t_{n-1})\|
% \leq Ct_n^{-1}\tau\|G_0\|_{\mathbb{H}}.
\leq Ct_n^{-1}\tau.$

Similar to the bound of $\zeta_1$ and $\zeta_2$ in Theorem \ref{time_dic_err}, taking $q=0,~\sigma=1/2$ in Theorem \ref{thm:Asig_v} and $0<\eps\ll (1-\al_0)/2,~\gamma=1-\eps/2$ in Theorem \ref{thm_holder_reg_det}, applying Theorem \ref{lem:bnkbound}, we have
\begin{equation*}
    \begin{aligned}
        &\left\|\overline{F}(t_n)*I_\tau Av^n-F(t_n)*g'(t_n)*Av(t_n)\right\|\\
        =&\cE\bigg[\Big\|\int_{0}^{t_n}\overline{F}(t_n-s)(g'(s)*Av(s))-F(t_n-s)(g'(s)*Av(s))ds\Big\|\bigg]\\
			&~~+\cE\bigg[\Big\|\overline{F}(t_n)*I_\tau Av^n-\overline{F}(t_n)*g'(t_n)*Av(t_n)\Big\|\bigg]\\
        \leq& C\tau^{1-\eps}
        +C\sum_{k=1}^{n}|b_{n,k}|\cdot\cE\left[\left\|v(t_k)-v^k\right\|\right]
        \leq C\tau^{1-\eps}
        +C\tau\sum_{k=1}^{n}t_{n-k+1}^{\al_*-1-\tilde{\eps}}\cE\left[\left\|v(t_k)-v^k\right\|\right].
    \end{aligned}
\end{equation*}
Using Lemma \ref{lmm:disGrowIneq} completes the proof.
\end{proof}

% \begin{thrm}
% 	Let $G$ be the solution of problem \eqref{model_2} and $G^n = u^n + v^n$, where $u^n$ and $v^n$ are the solutions of models \eqref{scheme} and \eqref{scheme_v}, respectively. Assume that $\|A^{-\rho}\|_{\mathcal{L}^0_2} < \infty$ with $\rho \in [0, \frac{H}{\alpha_0}) \cap [0, 1]$. Assume $G_0 \in \hat{H}^q(D)$ with $q \leq 2$. There holds
% 	\begin{equation*}
% 		(\mathbb{E}\|U(t_n)-G^n\|_\mathbb{H}^2)^{1/2}\leq C\tau^{H-\rho\alpha_0}+C{t_n}^{q\alpha_0/2-1}\tau\|G_0\|_{\hat{H}^q(D)}.
% 	\end{equation*}
% \end{thrm}

\section{Full discretization and error estimate}\label{sec:spatial_discre}
 In this section, we employ the finite element method to propose the fully-discrete schemes and perform corresponding error estimates. 
\subsection{Schemes and auxiliary estimates}
Let $\mathcal{T}_{h}$ be a partition of the domain $D$, where $h$ is the maximum diameter. Define $X_{h}(D):=\{\nu_{h}\in C(\bar{D}):\nu_{h}|_{\mathbf{T}}\in\mathcal{P}^{1}, \forall\mathbf{T}\in\mathcal{T}_{h}, \nu_{h}|_{\partial D}=0\}$, where $\mathcal{P}^{1}$ is the piecewise linear finite element space defined on $\mathcal{T}_{h}$. We introduce $P_h:L^2(D)\to X_h$ and $R_h:H^1_0(D)\to X_h$ by $(P_hu,\nu_h)=(u,\nu_h),~\forall\nu_h\in X_h$ and $(\nabla R_{h} u,\nabla\nu_{h})=(\nabla u,\nabla\nu_{h}),~\forall\nu_{h}\in X_{h}$, respectively. Then $R_h$ satisfies Ritz projection estimate $\|R_h v - v\| + h \|\nabla (R_h v - v)\| \leq C h^s \| v\|_{H^{s}(D)}
$ for $v \in H^s(D) \cap H_0^1(D)$ and $1 \leq s \leq 2$ \cite[Lemma 1.1]{ThoVid}. %p8 Lemma 1.1

% \begin{equation*}
% 	\begin{aligned}
% 		&(P_hu,\nu_h)=(u,\nu_h)\quad\forall\nu_h\in X_h, \\
% 		&(\nabla R_{h} u,\nabla\nu_{h})=(\nabla u,\nabla\nu_{h})\quad\forall\nu_{h}\in X_{h}.
% 	\end{aligned}
% \end{equation*}

Denote $A_h$ by $(A_hu_h,\nu_h)=(\nabla u_h,\nabla\nu_h)$ with $u_h, \nu_h\in X_h$, the fully discrete Galerkin scheme for 
the first equation of \eqref{jy1} is: $\forall t\in(0,T)$, find $u_h^n\in X_h$ such that
\begin{equation}\label{full_discre_1}
	\begin{aligned}&\left(\frac{u_h^n-u_h^{n-1}}{\tau},\nu_h\right)+\sum_{i=0}^{n-1}d_i^{(1-\alpha_0)}(A_hu_h^{n-i},\nu_h)\\
	&\hspace{6em}=\left(\frac{W_Q^H(t_n)-W_Q^H(t_{n-1})}{\tau},\nu_h\right)-\left(I_\tau A_hu^n_h,\nu_h\right),\quad\forall\nu_h\in X_h\subset H^1_0(D),\end{aligned}
\end{equation}
with $u_h^0=0$. Scheme \eqref{full_discre_1} can be written as
\begin{equation}\label{full_discrete_scheme}
\frac{u_{h}^{n}-u_{h}^{n-1}}{\tau}+\sum_{i=0}^{n-1}d_{i}^{(1-\alpha_0)}A_{h}u_{h}^{n-i}+I_\tau A_hu_h^n=P_{h}\frac{W_{Q}^{H}(t_{n})-W_{Q}^{H}(t_{n-1})}{\tau}.
\end{equation}
Proceeding as in the derivation of \eqref{u_nemeric} and applying the Cauchy integral principle, the solution of \eqref{full_discrete_scheme} admits the representation 
\begin{equation}\label{uh_nemeric}
u^n_h=\bar{F}_{h}(t_n)*\left(P_h\overline{\p}_\tau W^H_Q(t_n)-I_\tau A_hu^n_h\right).
\end{equation}
% where $\bar{F}_{h}(t)=\frac{1}{2\pi{i}}\int_{\Gamma_{\theta ,\kappa}^{\tau}}e^{zt}(\delta_{\tau}(e^{-z\tau}))^{\alpha_0-1}((\delta_{\tau}(e^{-z\tau}))^{\alpha_0}+A_{h})^{-1}\frac{z\tau}{e^{z\tau}-1}dz.$
\begin{equation*}\bar{F}_{h}(t)=\frac{1}{2\pi{i}}\int_{\Gamma_{\theta ,\kappa}^{\tau}}e^{zt}(\delta_{\tau}(e^{-z\tau}))^{\alpha_0-1}((\delta_{\tau}(e^{-z\tau}))^{\alpha_0}+A_{h})^{-1}\frac{z\tau}{e^{z\tau}-1}dz.\end{equation*}
SImilarly, we obtain the fully discrete scheme of the second equation of \eqref{jy1},
\begin{equation}\label{full_discre_2}
\frac{v_h^n - v_{h}^{n-1}}{\tau} + \sum_{i=0}^{n-1} d_i^{(1-\alpha_0)} A_h v_h^{n-i} + I_{\tau} A_h v_h^n = 0, \quad v_h^0 = P_h G_0.
\end{equation}
% \begin{equation}
% 	\left\{ 
% 	\begin{aligned}&\left(\frac{v_h^n-v_h^{n-1}}{\tau},\nu_h\right)+\sum_{i=0}^{n-1}d_i^{(1-\alpha)}(A_hv_h^{n-i},\nu_h)=-\left(I_\tau A_hv^n_h,\nu_h\right),\quad\forall\nu_h\in X_h\subset H^1_0(D),\\
% 		&v_h^0=P_h G_0.\end{aligned}
% 	\right.
% \end{equation}
Multiplying both sides of \eqref{full_discre_2} by $\xi^n$ and summing $n$ from $1$ to $\infty$ and using $v_h^0 = P_h G_0$,
\begin{equation*}
  \begin{aligned}
\left(
\delta_\tau(\xi)
+(\delta_\tau(\xi))^{1-\alpha_0}A_h
\right)\sum_{n=1}^\infty v_h^n \xi^n=
\frac{\xi}{\tau}P_hG_0-\sum_{n=1}^{\infty}\sum_{i=0}^{n-1}b_{n,n-i}A_hv_h^{n-i}\xi^n.
\end{aligned}
\end{equation*}
Then the solution of \eqref{full_discre_2} can be represented as
\begin{equation}\label{vn_hnumer}
    \begin{aligned}
    v^n_h=\breve{F}_h(t_{n-1})P_hG_0-\overline{F}_h(t_n)*I_\tau A_hv^n_h,
    \end{aligned}
\end{equation}
where $\bar{F}$ is defined above and $\breve{F}$ is defined by
\begin{equation}\label{def:breFh}
    \breve{F}_h(t):=\f{1}{2\pi i}\int_{\Gamma_{\theta,\kappa}^\tau}e^{zt}\left(\delta_\tau(e^{-z\tau})+(\delta_\tau(e^{-z\tau}))^{1-\al_0}A_h\right)^{-1}dz.
\end{equation}
We now define the fully discrete numerical approximation to \eqref{model_2} as $G_h^n := u_h^n + v_h^n$.

\begin{lmm}(cf. \cite[Lemma 13.1]{JinYan})\label{lem:Ah} %p68 P336
	For $\delta_\tau(\xi)$ defined in \eqref{def:delta_tau} and for all $\theta\in(\frac{\pi}{2},\pi)$,
	\begin{equation*}
		\begin{aligned}
		&\|(z+A_h)^{-1}\|\leq C|z|^{-1},~~\forall z\in\Sigma_\theta.\\
		&\|A_h(\delta_\tau(e^{-z\tau})^{\alpha_0}+A_h)^{-1}\|\leq C,~~
		\|\delta_\tau(e^{-z\tau})^{\alpha_0-1}A_h(\delta_\tau(e^{-z\tau})^{\alpha_0}+A_h)\|\leq C|z|^{-1},\quad\forall z\in\Gamma_{\theta,\kappa}^\tau.
		\end{aligned}
	\end{equation*}
\end{lmm}

%这里由于算子是负拉普拉斯，双线性形式是($(\nabla u,\nabla v)$_L2 )，所以Ritz投影就是椭圆椭圆算子
% \begin{lmm}(cf. \cite{JinYan})\label{lem:RieErr}
% 	For any $y\in H^{2\sigma}(D)\cap H^1_0(D)$, it holds
% 	\begin{equation*}
% 		\left\|R_hy-y\right\|_{L^2(D)}+h\left\|\nabla (R_hy-y)\right\|_{L^2(D)}\leq Ch^{2\sigma}\left\|y\right\|_{H^{2\sigma}(D)}.
% 	\end{equation*}
% \end{lmm}

\begin{lmm}\label{lem:FmFh}
	Let $\aleph_h(t):=\overline{F}(t)-\overline{F}_h(t)P_h$, then $\|\wtNh(z)A^s\|\leq C|z|^{\al_0-1}h^{2-2s}$ for any $z\in\Gamma_{\theta,\kappa}^{\tau}$ and $s\in[0,1/2]$.
\end{lmm}
\begin{proof}
	Following from \cite[(4.6),(4.8)]{NieSunDenSiamJNA}, we have the estimates $\|\wtNh(z)\|\leq C|z|^{\alpha_0-1}h^2$
	and $\|\wtNh(z)\|_{\hat{H}^{-1}(D)\rightarrow L^{2}(D)}\leq C|z|^{\alpha_0-1}h$ for $z\in\Gamma_{\theta,\kappa}^{\tau}$. Through interpolation property, for $s\in[0,1/2]$, we obtain $\|\wtNh(z)\|_{\hat{H}^{-2s}(D)\rightarrow L^{2}(D)}\leq C|z|^{\alpha_0-1}h^{2-2s}$ for $z\in\Gamma_{\theta,\kappa}^{\tau}.$
	The proof is complete.
\end{proof}

%\begin{lmm}\label{lem:RieErr}\cite{JinYan} 
%	For any $y\in H^{r+1}(D)\cap H^1_0(D)$, it holds
%	\begin{equation}
%		\left\|R_hy-y\right\|_{L^2(D)}+h\left\|\nabla (R_hy-y)\right\|_{L^2(D)}\leq Ch^{r+1}\left\|y\right\|_{H^{r+1}(D)}.
%	\end{equation}
%\end{lmm}

%\begin{lmm}\label{lem:RieErr}\cite{JinYan} 
%	For any $y\in H^q(\Omega)\cap H^1_0(\Omega)$ with $q=1,2$, it holds
%	\begin{equation}
%		\left\|R_hy-y\right\|_{L^2(\Omega)}+h\left\|\nabla (R_hy-y)\right\|_{L^2(\Omega)}\leq Ch^q\left\|y\right\|_{H^q(\Omega)}.
%	\end{equation}
%\end{lmm}

%\begin{lmm}\cite{JinYan} %p35,lemma2.6 
%	 Let $v\in L^2(\Omega),z\in\Sigma_\theta,w=(z^\alpha+A)^{-1}v$,and $w_h=(z^\alpha+A_h)^{-1}P_hv$, then
%	\begin{equation}
%		\|w_h-w\|_{L^2(\Omega)}+h\|\nabla(w_h-w)\|_{L^2(\Omega)}\leq Ch^2\|v\|_{L^2(\Omega)}.
%	\end{equation}
%\end{lmm}

% Similar to Theorem \ref{thm_space_dis_err}, we have the following theorem,
% \begin{thrm}\label{thm:space_dis_err_G}
% 	$G^n:=u^n+v^n$ and $G^n_h:=u^n_h+v^n_h$, assume that $\|A^{-\rho}\|_{\mathcal{L}_2^0}$ with $\rho\in[-1,\frac{1}{2}]$, and $G_0\in\hat{H}^q(D)$ with $q\leq2$, then
% 	\begin{equation*}
% 		\begin{aligned}
% 			(\cE\left\|G^n-G^n_h\right\|^2_{\mathbb{H}})^{1/2}\leq Ch^{\min(2-2\rho,\f{2H}{\al_0}-2\rho-\varepsilon,2)}+Ch^2\|G_0\|_{\hat{H}^q(D)}
% 		\end{aligned}
% 	\end{equation*}
% \end{thrm}

% Combining Theorem \ref{thm_space_dis_err} with the results in \cite{JinLazRayZho, NieSunDenSiamJNA} yields the following statements:
\begin{lmm}\label{lmm:breFnm1-breFhnm1}
Let $\breve{F}(t)$ and $\breve{F}_h(t)$ be defined as in \eqref{def:breF} and \eqref{def:breFh}, respectively. Then
$\|\breve{F}(t_{n-1})-\breve{F}_h(t_{n-1})P_h\|\leq Ct_{n}^{-\alpha_0}h^2.$
\end{lmm}
\begin{proof}
Since $\|(z^{\alpha}+A)^{-1}-(z^{\alpha}+A_{h})^{-1}P_{h}\|\leq Ch^{2}$ for all $z\in\Sigma_{\theta}$ \cite[Lemma 3.4]{Bazhlekova2015} and Lemma \ref{lem:delta}, for $n\geq 2$, we have
\begin{equation*}
    \begin{aligned}
        &\|\breve{F}(t_{n-1})-\breve{F}_h(t_{n-1})P_h\|\\
        % \leq&\int_{\Gamma_{\theta,\kappa}^\tau}e^{zt_{n-1}}\left\|\left(\left(\delta_\tau(e^{-z\tau})+(\delta_\tau(e^{-z\tau}))^{1-\al_0}A\right)^{-1}-\left(\delta_\tau(e^{-z\tau})+(\delta_\tau(e^{-z\tau}))^{1-\al_0}A_h\right)^{-1}P_h\right)\right\|dz\\
        \leq&\int_{\Gamma_{\theta,\kappa}^\tau}e^{zt_{n-1}}|(\delta_\tau(e^{-z\tau}))^{\al_0-1}|\left\|\left((\delta_\tau(e^{-z\tau}))^{\alpha_0}+A\right)^{-1}-\left((\delta_\tau(e^{-z\tau}))^{\alpha_0}+A_h\right)^{-1}P_h\right\||dz|\\
        \leq& C\int_{\Gamma_{\theta,\kappa}^\tau}e^{zt_{n-1}}|z|^{\al_0-1}h^2|dz|\leq Ct_{n-1}^{-\alpha_0}h^2\le C2^{\al_0}t_{n}^{-\alpha_0}h^2.
    \end{aligned}
\end{equation*}

For $n=1$, $\|\breve{F}(0)-\breve{F}_h(0)P_h\|\le Ch^2\int_{\Gamma_{\theta,\kappa}^\tau}|z|^{\al_0-1}|dz|\leq C\tau^{-\al_0}h^2.$
% Since
% \begin{equation*}
% \begin{aligned}
%      \int_{\Gamma_{\theta,\kappa}^\tau}|z|^{\al_0-1}|dz|=&2\int_\kappa^{\pi/(\tau\sin\theta)}r^{\al_0-1}dr+\int_{-\theta}^{\theta}\kappa^{\al_0}d\psi\\
%      =&\frac{2}{\al_0}\big((\pi/(\tau\sin\theta))^{\al_0}-\kappa^{\al_0}\big)+2\kappa^{\al_0}\theta \le C\tau^{-\al_0},
% \end{aligned}
% \end{equation*}
%  Hence $\int_{\Gamma_{\theta,\kappa}^\tau}|z|^{\al_0-1}|dz|\le C\tau^{-\al_0}.$
% Therefore $\|\breve{F}(0)-\breve{F}_h(0)P_h\|\le C\tau^{-\al_0}h^2$.
The proof is finished.
\end{proof}
\subsection{Error estimates}
We derive spatial error estimates for the fully discrete schemes \eqref{full_discrete_scheme} and \eqref{full_discre_2}. Combining these estimates with the temporal error estimates, we obtain an overall error estimate for the fully discrete approximation $G_h^n$ of the original problem \eqref{model_2}. We begin with the error estimate for scheme \eqref{full_discrete_scheme}.
\begin{thrm}\label{thm_space_dis_err}
	Let $u^n$ and $u^h_n$ are the solutions of models \eqref{scheme} and \eqref{full_discrete_scheme}, respectively. Assume that $\|A^{-\rho}\|_{\mathcal{L}_2^0}<\infty$ with $\rho\in[0,\frac{H}{\al_0}-\frac{1}{2})\cap[0,\frac{1}{2}]$, then
	$\big(\cE\left[\left\|u^n-u^n_h\right\|^2\right]\big)^{1/2}\leq Ch^{\min(2-2\rho,\f{2H}{\al_0}-2\rho-\varepsilon)}$.
\end{thrm}
\begin{proof}
Subtracting \eqref{uh_nemeric} from \eqref{u_nemeric}, we obtain the expectation% of $\left\|u^n-u^n_h\right\|^2_{\mathbb{H}}$
\begin{equation*}
	\begin{aligned}
		\cE\left[\left\|u^n-u^n_h\right\|^2\right]=&2\cE \left[\left\|\overline{F}(t_n)*\left(\overline{\p}_\tau W^H_Q(t_n)\right)-\bar{F}_{h}(t_n)*\left(P_h\overline{\p}_\tau W^H_Q(t_n)\right)\right\|^2\right]\\
		&+2\cE\left[\left\|\bar{F}_{h}(t_{n})*\left(I_\tau A_hu^n_h\right)-\bar{F}(t_{n})*\left(I_\tau Au^n\right)\right\|^2\right].
	\end{aligned}
\end{equation*}

For $\rho\in[-1,\f{1}{2}],~0<\eps\ll \f{2H}{\al_0}-2\rho$, there holds \cite[Theorem 4.1]{NieSunDenSiamJNA},
\begin{equation*}
	\begin{aligned}
		\left(\cE \left[\left\|\overline{F}(t_n)*\left(\overline{\p}_\tau W^H_Q(t_n)\right)-\bar{F}_{h}(t_n)*\left(P_h\overline{\p}_\tau W^H_Q(t_n)\right)\right\|^2\right]\right)^{1/2}\leq Ch^{\min(2-2\rho,\f{2H}{\al_0}-2\rho-\varepsilon,2)}.
	\end{aligned}
\end{equation*}
Since $A_hR_h=P_hA$, now we estimate
\begin{equation*}
\begin{aligned}
		&\cE\left[\left\|\bar{F}_{h}(t_{n})*(I_\tau A_hu^n_h)-\bar{F}(t_{n})*\left(I_\tau A_hR_hu^n\right)\right\|^2\right]\\
		% =&\cE\bigg[\big\|\bar{F}_{h}(t_{n})*\left(I_\tau A_hu^n_h\right)-\bar{F}_{h}(t_{n})*\left(I_\tau A_hR_hu^n\right)\\
		% &\hspace{6em}+\bar{F}_{h}(t_{n})*(I_\tau P_hAu^n)-\bar{F}(t_{n})*\left(I_\tau Au^n\right)\big\|^2\bigg]\\
		% \leq&\cE\bigg[2\big\|\bar{F}_{h}(t_{n})*\left(I_\tau A_hu^n_h\right)-\bar{F}_{h}(t_{n})*\left(I_\tau A_hR_hu^n\right)\big\|^2_{\mathbb{H}}\\
		% &\hspace{6em}+2\big\|\bar{F}_{h}(t_{n})*\left(I_\tau A_hR_hu^n\right)-\bar{F}(t_{n})*\left(I_\tau Au^n\right)\big\|^2_{\mathbb{H}}\bigg]\\
		% \leq&\cE\bigg[2\big\|\bar{F}_{h}(t_{n})*\left(I_\tau A_hu^n_h\right)-\bar{F}_{h}(t_{n})*\left(I_\tau A_hR_hu^n\right)\big\|^2\\\\\\这段可以打开
		&\hspace{6em}+2\big\|\bar{F}_{h}(t_{n})*\left(I_\tau P_hAu^n\right)-\bar{F}(t_{n})*\left(I_\tau Au^n\right)\big\|^2\bigg]\\
		\leq&2\cE\bigg[\Big\|\left(\bar{F}_{h}(t_{n})A_h\right)*\bigg(\big(\sum_{k=1}^{n}b_{n,k}u_h^k\big)-\big(\sum_{k=1}^{n}b_{n,k}R_hu^k\big)\bigg)\Big\|^2\bigg]\\
		&\hspace{6em}+2\cE\bigg[\bigg\|\left(\bar{F}_{h}(t_{n})P_h-\bar{F}(t_{n})\right)*\big(\sum_{k=1}^{n}b_{n,k}Au^k\big)\bigg\|^2\bigg]\\
		\leq&2\cE\bigg[\Big(\int_{0}^{t_n}\left\|\overline{F}_{h}(t_n-s)A_h\right\|\cdot\Big\|\sum_{k=1}^{n}b_{n,k}\left(u_h^k-R_hu^k\right)\Big\|ds\Big)^2\bigg]\\
		&\hspace{2em}+2\cE\bigg[\Big(\int_{0}^{t_n}\left\|\left(\bar{F}_{h}(t_{n}-s)P_h-\bar{F}(t_{n}-s)\right)A^{1-\sigma}\right\|\cdot\Big\|\Big(\sum_{k=1}^{n}b_{n,k}A^\sigma u^k\Big)\Big\|ds\Big)^2\bigg],
\end{aligned}
\end{equation*}
where $\sigma=\min\{1-\rho, \f{H}{\al_0}-\rho-\f{\eps}{2}\}$ in Theorem \ref{thm_reg}.
According to Lemma \ref{lem:Ah}, %similar to \eqref{BoundFA}, 
\begin{equation*}
	\begin{aligned}
		\left\|\overline{F}_{h}(t_n-s)A_h\right\|
		&\leq C\int_{\Gamma_{\theta,\kappa}^\tau}e^{z(t_n-s)}\left\|\left(\delta_\tau(e^{-z\tau})+(\delta_\tau(e^{-z\tau}))^{1-\al_0}A_{h}\right)^{-1}A_h\f{z\tau}{e^{z\tau}-1}\right\|dz\\
		&\leq C\int_{\Gamma_{\theta,\kappa}^\tau}e^{z(t_n-s)}|z|^{\al_0-1}dz,
	\end{aligned}
\end{equation*}
then applying Lemma \ref{lem:bnkbound}, we obtain
\begin{equation*}
	\begin{aligned}
	&\cE\bigg[\Big(\int_{0}^{t_n}\Big\|\overline{F}_{h}(t_n-s)A_h\Big\|\cdot\bigg\|\sum_{k=1}^{n}b_{n,k}\left(u_h^k-R_hu^k\right)\bigg\|ds\Big)^2\bigg]\\
	\leq&C\Big(\int_{0}^{t_n}\int_{\Gamma_{\theta,\kappa}^\tau}e^{z(t_n-s)}|z|^{\al_0-1}dzds\Big)^2\cE\bigg[\Big\|\sum_{k=1}^{n}b_{n,k}\left(u_h^k-R_hu^k\right)\Big\|^2\bigg]\\
	\leq &C\sum_{k=1}^{n}|b_{n,k}|\cE\left[\left\|u_h^k-R_hu^k\right\|^2\right]\leq C\tau\sum_{k=1}^{n}t_{n-k+1}^{\al_*-1-\tilde{\eps}}\cE\left[\left\|u_h^k-R_hu^k\right\|^2\right].
	\end{aligned}
\end{equation*}

%taking an extra assumption that $\sigma>1/2$,
According to Theorem \ref{time_dic_err} and Lemmas \ref{lem:AsigUn_bound} and \ref{lem:FmFh},
\begin{equation*}%\label{boundFhPhmF}
	\begin{aligned}
		&\cE\bigg[\Big(\int_{0}^{t_n}\Big\|\left(\bar{F}_{h}(t_{n}-s)P_h-\bar{F}(t_{n}-s)\right)A^{1-\sigma}\Big\|\cdot\Big\|\Big(\sum_{k=1}^{n}b_{n,k}A^\sigma u^k\Big)\Big\|ds\Big)^2\bigg]\\
    	\leq&C\cE\bigg[\bigg(\int_{0}^{t_n}\Big\|\int_{\Gamma_{\theta ,\kappa}^{\tau}}e^{z(t_n-s)}(\delta_{\tau}(e^{-z\tau}))^{\alpha_0-1}\big(((\delta_{\tau}(e^{-z\tau}))^{\alpha_0}+A_{h})^{-1}P_h\\
    	&\hspace{3em}-((\delta_{\tau}(e^{-z\tau}))^{\alpha_0}+A)^{-1}\big)A^{1-\sigma}\frac{z\tau}{e^{z\tau}-1}dz\Big\|\cdot\Big\|\Big(\sum_{k=1}^{n}b_{n,k}A^\sigma u^k\Big)\Big\|ds\bigg)^2\bigg]\\
    	\leq& Ch^{4\sigma}\cE\bigg[\Big(\int_{0}^{t_n}\int_{\Gamma_{\theta ,\kappa}^{\tau}}e^{z(t_n-s)}|z|^{\al_0-1}|dz|\Big\|\Big(\sum_{k=1}^{n}b_{n,k}A^{\sigma}u^k\Big)\Big\|ds\Big)^2\bigg]\\
%    	\leq& Ch^{4\sigma}\cE\left[\left(\left\|\sum_{k=1}^{n}b_{n,k}A^\sigma  u^k\right)\right\|_{\mathbb{H}}^2\right]
    	% \leq& Ch^{4\sigma}\left(\sum_{k=1}^{n}|b_{n,k}|\sum_{k=1}^{n}|b_{n,k}|\cE\left[\left\|A^\sigma  u^k\right\|_{\mathbb{H}}^2\right]\right)\leq Ch^{4\sigma}.%~~(\sigma<1-\rho)
        \leq& Ch^{4\sigma}\bigg(\sum_{k=1}^{n}|b_{n,k}|\cE\left[\left\|A^\sigma  u^k\right\|^2\right]\bigg)\leq Ch^{4\sigma}.%~~(\sigma<1-\rho)
	\end{aligned}
\end{equation*}
% which is obtained by \eqref{sumbnk} 
% and
% \begin{equation*} 
% 	\begin{aligned}
% 		\cE\left[\|A^\sigma u^k\|_{\mathbb{H}}^2\right]&=\cE\left[\|A^\sigma (u^k-u(t_k)+u(t_k))\|_{\mathbb{H}}^2\right]\\
% 		&\leq 2\cE\left[\|A^\sigma (u^k-u(t_k))\|_{\mathbb{H}}^2\right]+2\cE\left[\|A^\sigma u(t_k)\|_{\mathbb{H}}^2\right]\leq C.
% 	\end{aligned}
% \end{equation*}
% 按道理时间在怎么离散不影响空间正则性吧，所以虽然是Asigma u^k，仍然可以控制
% According to Theorem \ref{thm_reg_G} and Lemma \ref{lem:RieErr}, 
Consequently,
we obtain
\begin{equation}\label{ineq:un-uhn}
	\begin{aligned}
		&\cE\left[\left\|u^n-u^n_h\right\|^2\right]\leq Ch^{\min(4-4\rho,\f{4H}{\al_0}-4\rho-2\varepsilon)}+C\tau\sum_{k=1}^{n}t_{n-k+1}^{\al_*-1-\tilde{\eps}}\cE\left[\left\|u_h^k-R_hu^k\right\|^2\right].
	\end{aligned}
\end{equation}
Choose $\eps$ sufficiently small such that $0<\eps\ll2(\f{H}{\al_0}-\rho)-1$, which ensures that $2\sigma\in[1,2]$. By the inequality $\cE\left[\left\|R_hu^n-u_h^n\right\|^2\right]\leq 2\cE\left[\left\|R_hu^n-u^n\right\|^2\right]+2\cE\left[\left\|u^n-u_h^n\right\|^2\right]$, together with the Ritz projection estimate and Lemma \ref{lem:AsigUn_bound}, we obtain $\cE\left[\left\|R_hu^n-u^n_h\right\|^2\right]\leq Ch^{4\sigma}\cE[\|A^\sigma u^n\|^2]\leq Ch^{4\sigma}$. Combining these estimates with \eqref{ineq:un-uhn}, we have
\begin{equation*}
	\begin{aligned}
		\cE\left[\left\|R_hu^n-u^n_h\right\|^2\right]\leq Ch^{\min(4-4\rho,\f{4H}{\al_0}-4\rho-2\varepsilon)}+C\tau\sum_{k=1}^{n}t_{n-k+1}^{\al_*-1-\tilde{\eps}}\cE\left[\left\|u_h^k-R_hu^k\right\|^2\right].
	\end{aligned}
\end{equation*}

By Lemma \ref{lmm:disGrowIneq}, we obtain
       $ \left(\cE\left\|R_hu^n-u^n_h\right\|^2\right)^{1/2}\leq Ch^{\min(2-2\rho,\f{2H}{\al_0}-2\rho-\varepsilon)}$, %因为rho>0，所以删除2
	which, together with
% Applying Lemma \ref{lem:RieErr},
 \eqref{ineq:un-uhn}, completes the proof.
\end{proof}

Then we provide spatial error estimate for the fully discrete scheme \eqref{full_discre_2}.
\begin{thrm}\label{thm_space_v_dis_err}
    Let $v^n$ and $v^h_n$ denote the solutions of models \eqref{scheme_v} and \eqref{full_discre_2}, respectively.  Assume that $\|A^{-\rho}\|_{\mathcal{L}_2^0}< \infty$ with $\rho\in[0, \frac{H}{\alpha_0}) \cap[0,\frac{1}{2}]$. If $G_0\in\mathbb{H}$, then we have the error estimate $\|v^n-v^n_h\|\leq Ct_n^{-\al_0}h^2.$
\end{thrm}
    \begin{proof}
Applying \eqref{vn_numer} and \eqref{vn_hnumer} yields
\begin{equation*}%\label{ineq:vtn-vn}
    \begin{aligned}
        \|v^n-v^n_h\|&\leq \|(\breve{F}(t_{n-1})-\breve{F}_h(t_{n-1})P_h)G_0\|
        +\left\|\overline{F}(t_n)*I_\tau Av^n-\overline{F}_h(t_n)*I_\tau A_hv^n_h\right\|.
    \end{aligned}
\end{equation*}% 由于t_{-al_0}>>t_{-sigma \al_0}，所以t_{-sigma \al_0}项被第一项覆盖了
By using Lemma \ref{lmm:breFnm1-breFhnm1}, we estimate
\begin{equation*}
    \begin{aligned}
        \|(\breve{F}(t_{n-1})-\breve{F}_h(t_{n-1})P_h)G_0\|\leq \|G_0\|\cdot\|\breve{F}(t_{n-1})-\breve{F}_h(t_{n-1})P_h\|\leq Ct_n^{-\alpha_0}h^2.
    \end{aligned}
\end{equation*}

% Similar to the proof in Theorem \ref{thm_space_dis_err}  
From Lemmas \ref{lem:bnkbound} and \ref{lem:AsigVn_bound}, we obtain
\begin{equation*}
    \begin{aligned}
        \left\|\overline{F}(t_n)*I_\tau Av^n-\overline{F}_h(t_n)*I_\tau A_hv^n_h\right\|\leq Ct_n^{-\al_0}h^{2}+C\tau\sum_{k=1}^{n}t_{n-k+1}^{\al_*-1-\tilde{\eps}}\cE\left[\left\|v_h^k-R_hv^k\right\|^2\right],
    \end{aligned}
\end{equation*}
where we take $\sigma=1$ in Lemma \ref{lem:AsigVn_bound}. Using Lemma \ref{lmm:disGrowIneq} completes the proof.
\end{proof}

Combining Theorems \ref{thm_space_dis_err} and \ref{thm_space_v_dis_err} yields the following main theorem.
\begin{thrm}\label{thm:space_dis_err_G}
    Assume  $\|A^{-\rho}\|_{\mathcal{L}_2^0}< \infty$ with $\rho\in[0, \frac{H}{\alpha_0}-\frac{1}{2}) \cap[0,\frac{1}{2}]$ and $G_0\in\mathbb{H}$. Then
	\begin{equation*}
		\begin{aligned}
			(\cE\left\|G(t_n)-G^n_h\right\|^2)^{1/2}\leq C(h^{\min(2-2\rho,\f{2H}{\al_0}-2\rho-\varepsilon)}+t_n^{-\al_0}h^{2}+\tau^{H-\rho\al_0-\eps}+t_n^{-1}\tau).
		\end{aligned}
	\end{equation*}
\end{thrm}

\section{Numerical experiments}\label{sec:simulation}
In our setting, we take the number of sample paths $M=100,~D=(0,1),~T = 0.01$ and $\Lambda_k = k^m$ for $k=1,2,\dots$, where the parameter $m\in\mathbb{R}$ together with the spatial dimension $d$ controls the regularity of the noise. The eigenfunctions of $Q$ are $\phi_k(x)=\sqrt{2}\sin(k\pi x)$.

If $m$ is decreased, the lower bound for $\rho$ becomes smaller, meaning that the noise admits higher spatial regularity. Conversely, if $m$ is increased, the decay of $\Lambda_k$ slows down, allowing stronger contributions from high-frequency modes and thus producing a rougher noise field. We approximate $W_Q^H$ by %Refer to \cite{NieSunDenSiamJNA}, 
\begin{equation*}W_Q^H(x,t)\approx\sum_{k=1}^{1000}\sqrt{\Lambda_k}\phi_k(x)W_k^H(t),\end{equation*} 
and the relationship between $m$ and $\rho$ is given by the inequality $\rho>(1+m)d/4$, which stems from the assumption $\|A^{-\rho}\|_{\mathcal{L}_2^0}<\infty$ as established in \cite{LapJFA,LiYau,NieSunDenSiamJNA}.

We take the variable exponent
\begin{equation*}%\label{alp_exmp}
	\alpha({t})=\alpha_T+\left(\alpha_0-\alpha_T\right)\left[1-\frac{{t}}{{T}}-\frac{\sin\left(2\pi(1-{t}/{T})\right)}{2\pi}\right],~~\alpha_T:=\alpha(T).
\end{equation*}

Since the exact solution is unknown, let $\omega_i$ denote the $i$th sample path. We compute
\begin{align*}
	e_{h}=\Big(\frac{1}{M}\sum_{i=1}^{M}\|G_{h}^{N}(\omega_{i})-G_{h/2}^{N}(\omega_{i})\|^{2}\Big)^{1/2},~~
	e_{\tau}=\Big(\frac{1}{M}\sum_{i=1}^{M}\|G^N_{\tau}(\omega_{i})-G^{2N}_{\tau/2}(\omega_{i})\|^{2}\Big)^{1/2},
\end{align*}
and the convergence rates follow from
\begin{equation*}
	\mathrm{Rate}_h=\frac{\ln(e_h/e_{h/2})}{\ln 2},\quad\mathrm{Rate}_\tau=\frac{\ln(e_\tau/e_{\tau/2})}{\ln 2}.
\end{equation*}

The following two examples are presented to compute the convergence rates in time and space, respectively.
\begin{xmpl}\label{exm:1}
    In this example, we investigate the temporal convergence rates of the proposed fully discrete scheme.
	The numerical experiments in this example use the parameters $G_0 = 0$ and $h = 1/100$. \crefrange{table1_tau}{table6_tau} summarize the results for different parameter sets: $(\alpha_0,~\alpha_T)=(0.3,~0.7),~(0.5,~0.2),~(0.8,~0.5)$, $H=0.6,~0.75$, and $m=0,~-0.5,~-1$. When $m = 0$, under the condition $\rho > 1/4$, this leads to a convergence order of $\mathcal{O}(\tau^{H-\alpha_0/4})$, as observed in Tables \ref{table1_tau} and \ref{table2_tau}. Furthermore, an increase in the noise regularity is achieved by setting $m = -0.5~(\rho>1/8)$ and $m = -1~(\rho>0)$, which leads to a corresponding improvement in the convergence rates. These improved rates, detailed in \crefrange{table3_tau}{table6_tau}, are in agreement with the theoretical estimates provided in Theorem \ref{time_dic_err_G}.
\end{xmpl}
\begin{table}[!t]
	\caption{Temporal errors and observed convergence rates for numerical scheme \eqref{scheme} with $H=0.6$,~$m=0~(\rho>1/4)$.}
	\label{table1_tau}
	\centering
	%	\vspace{0.5em}	
	\medskip\small\renewcommand{\arraystretch}{1.15}
	\begin{tabular}{||c|c|c|c|c|c|c||}
		\hline
		& \multicolumn{2}{c|}{$(\alpha_0,~\alpha_T)=(0.3,~0.7)$} &  \multicolumn{2}{c|}{$(\alpha_0,~\alpha_T)=(0.5,~0.2)$}  & \multicolumn{2}{c||}{$(\alpha_0,~\alpha_T)=(0.8,~0.5)$} \\
		\hline
		$T/\tau$ & $e_{\tau}$&  $\mathrm{Rate}_\tau$& $e_{\tau}$ & $\mathrm{Rate}_\tau$& $e_{\tau}$ & $\mathrm{Rate}_\tau$\\
		\hline
		$2^{2}$&	1.21E-3     &    -   & 3.26E-3   &   -     &9.50E-3  &  -  \\
		$2^{3}$&	8.60E-4	    & 0.4960   & 2.12E-3   & 0.6172    &6.39E-3  & 0.5731  \\
		$2^{4}$&	5.94E-4	    & 0.5350   & 1.44E-3   & 0.5649    &4.39E-3  & 0.5400  \\
		$2^{5}$&	4.09E-4	    & 0.5376   & 1.02E-3   & 0.4881    &3.29E-3  & 0.4180  \\
        $2^{6}$&	2.88E-4	    & 0.5072   & 7.26E-4   & 0.4941    &2.36E-3  & 0.4799 \\
		\hline
		predicted rate& &0.5250& &0.4750& &0.4000\\
		\hline
	\end{tabular}
\end{table}

\begin{table}[!t]
	\caption{Temporal errors and observed convergence rates for numerical scheme \eqref{scheme} with $H=0.75$,~$m=0~(\rho>1/4)$.}
	\label{table2_tau}
	\centering
	%	\vspace{0.5em}	
	\medskip\small\renewcommand{\arraystretch}{1.15}
	\begin{tabular}{||c|c|c|c|c|c|c||}
		\hline
		& \multicolumn{2}{c|}{$(\alpha_0,~\alpha_T)=(0.3,~0.7)$} &  \multicolumn{2}{c|}{$(\alpha_0,~\alpha_T)=(0.5,~0.2)$}  & \multicolumn{2}{c||}{$(\alpha_0,~\alpha_T)=(0.8,~0.5)$} \\
		\hline
		$T/\tau$ & $e_{\tau}$&  $\mathrm{Rate}_\tau$& $e_{\tau}$ & $\mathrm{Rate}_\tau$& $e_{\tau}$ & $\mathrm{Rate}_\tau$\\
		\hline
		$2^{2}$&	4.07E-4     &    -    & 1.02E-3   &   -     &3.09E-4  &  -  \\
		$2^{3}$&	2.58E-4	    & 0.6559   & 6.16E-4   & 0.7316    &1.93E-4  & 0.6772  \\
		$2^{4}$&	1.63E-4	    & 0.6653   & 3.77E-4   & 0.7087    &1.21E-4  & 0.6753  \\
		$2^{5}$&	1.01E-4	    & 0.6841   & 2.45E-4   & 0.6243    &8.07E-4  & 0.5858  \\
        $2^{6}$&	6.40E-5	    & 0.6657   & 1.54E-4   & 0.6678    &5.16E-5  & 0.6462  \\
		\hline
		predicted rate& &0.6750& &0.6250& &0.5500\\
		\hline
	\end{tabular}
\end{table}

\begin{table}[!t]
	\caption{Temporal errors and observed convergence rates for numerical scheme \eqref{scheme} with $H=0.6$,~$m=-0.5~(\rho>1/8)$.}
	\label{table3_tau}
	\centering
	%	\vspace{0.5em}	
	\medskip\small\renewcommand{\arraystretch}{1.15}
	\begin{tabular}{||c|c|c|c|c|c|c||}
		\hline
		& \multicolumn{2}{c|}{$(\alpha_0,~\alpha_T)=(0.3,~0.7)$} &  \multicolumn{2}{c|}{$(\alpha_0,~\alpha_T)=(0.5,~0.2)$}  & \multicolumn{2}{c||}{$(\alpha_0,~\alpha_T)=(0.8,~0.5)$} \\
		\hline
		$T/\tau$ & $e_{\tau}$&  $\mathrm{Rate}_\tau$& $e_{\tau}$ & $\mathrm{Rate}_\tau$& $e_{\tau}$ & $\mathrm{Rate}_\tau$\\
		\hline
		$2^{2}$&	1.16E-3     &    -    & 2.81E-3   &   -     &6.61E-3  &  -  \\
		$2^{3}$&	8.13E-4	    & 0.5073   & 1.75E-3   & 0.6775    &4.15E-3  & 0.6712  \\
		$2^{4}$&	5.53E-4	    & 0.5577   & 1.12E-3   & 0.6470    &2.67E-3  & 0.6354  \\
		$2^{5}$&	3.75E-4	    & 0.5593   & 7.65E-4   & 0.5509    &1.85E-3  & 0.5319  \\
        $2^{6}$&	2.59E-4	    & 0.5353   & 5.23E-4   & 0.5492    &1.23E-3  & 0.5873  \\
		\hline
		predicted rate& &0.5625& &0.5375& &0.5000\\
		\hline
	\end{tabular}
\end{table}

\begin{table}[!t]
	\caption{Temporal errors and observed convergence rates for numerical scheme \eqref{scheme} with $H=0.75$,~$m=-0.5~(\rho>1/8)$.}
	\label{table4_tau}
	\centering
	%	\vspace{0.5em}	
	\medskip\small\renewcommand{\arraystretch}{1.15}
	\begin{tabular}{||c|c|c|c|c|c|c||}
		\hline
		& \multicolumn{2}{c|}{$(\alpha_0,~\alpha_T)=(0.3,~0.7)$} &  \multicolumn{2}{c|}{$(\alpha_0,~\alpha_T)=(0.5,~0.2)$}  & \multicolumn{2}{c||}{$(\alpha_0,~\alpha_T)=(0.8,~0.5)$} \\
		\hline
		$T/\tau$ & $e_{\tau}$&  $\mathrm{Rate}_\tau$& $e_{\tau}$ & $\mathrm{Rate}_\tau$& $e_{\tau}$ & $\mathrm{Rate}_\tau$\\
		\hline
		$2^{2}$&	3.88E-4     &    -    & 8.82E-4   &   -     &2.16E-3  &  -  \\
		$2^{3}$&	2.45E-4	    & 0.6659   & 5.10E-4   & 0.7899    &1.28E-3  & 0.7589  \\
		$2^{4}$&	1.53E-4	    & 0.6807   & 2.98E-4   & 0.7774    &7.58E-4  & 0.7551  \\
		$2^{5}$&	9.37E-5	    & 0.7026   & 1.86E-4   & 0.6812    &4.71E-4  & 0.6870  \\
        $2^{6}$&	5.81E-5	    & 0.6896   & 1.12E-4   & 0.7253    &2.80E-4  & 0.7484  \\
		\hline
		predicted rate& &0.7125& &0.6875& &0.6500\\
		\hline
	\end{tabular}
\end{table}

\begin{table}[!t]
	\caption{Temporal errors and observed convergence rates for numerical scheme \eqref{scheme} with $H=0.6$,~$m=-1$~$(\rho>0)$.}
	\label{table5_tau}
	\centering
	%	\vspace{0.5em}	
	\medskip\small\renewcommand{\arraystretch}{1.15}
	\begin{tabular}{||c|c|c|c|c|c|c||}
		\hline
		& \multicolumn{2}{c|}{$(\alpha_0,~\alpha_T)=(0.3,~0.7)$} &  \multicolumn{2}{c|}{$(\alpha_0,~\alpha_T)=(0.5,~0.2)$}  & \multicolumn{2}{c||}{$(\alpha_0,~\alpha_T)=(0.8,~0.5)$} \\
		\hline
		$T/\tau$ & $e_{\tau}$&  $\mathrm{Rate}_\tau$& $e_{\tau}$ & $\mathrm{Rate}_\tau$& $e_{\tau}$ & $\mathrm{Rate}_\tau$\\
		\hline
		$2^{2}$&	1.12E-3     &    -    & 2.51E-3   &   -      &4.85E-3  &  -  \\
		$2^{3}$&	7.83E-4	    & 0.5145   & 1.52E-3   & 0.7271    &2.86E-3  & 0.7638  \\
		$2^{4}$&	5.26E-4	    & 0.5740   & 9.15E-4   & 0.7270    &1.72E-3  & 0.7299  \\
		$2^{5}$&	3.53E-4	    & 0.5758   & 6.02E-4   & 0.6056    &1.10E-3  & 0.6436  \\
        $2^{6}$&	2.40E-4	    & 0.5576   & 3.96E-4   & 0.6033    &6.82E-4  & 0.6929  \\
		\hline
		predicted rate& &0.6000& &0.6000& &0.6000\\
		\hline
	\end{tabular}
\end{table}

\begin{table}[!t]
	\caption{Temporal errors and observed convergence rates for numerical scheme \eqref{scheme} with $H=0.75$,~$m=-1~(\rho>0)$.}
	\label{table6_tau}
	\centering
	%	\vspace{0.5em}	
	\medskip\small\renewcommand{\arraystretch}{1.15}
	\begin{tabular}{||c|c|c|c|c|c|c||}
		\hline
		& \multicolumn{2}{c|}{$(\alpha_0,~\alpha_T)=(0.3,~0.7)$} &  \multicolumn{2}{c|}{$(\alpha_0,~\alpha_T)=(0.5,~0.2)$}  & \multicolumn{2}{c||}{$(\alpha_0,~\alpha_T)=(0.8,~0.5)$} \\
		\hline
		$T/\tau$ & $e_{\tau}$&  $\mathrm{Rate}_\tau$& $e_{\tau}$ & $\mathrm{Rate}_\tau$& $e_{\tau}$ & $\mathrm{Rate}_\tau$\\
		\hline
		$2^{2}$&	3.75E-4     &    -    & 7.88E-4   &   -      &1.60E-3  &  -  \\
		$2^{3}$&	2.36E-4	    & 0.6721   & 4.40E-4   & 0.8387    &8.96E-4  & 0.8351  \\
		$2^{4}$&	1.46E-4	    & 0.6917   & 2.46E-4   & 0.8411    &5.03E-4  & 0.8324  \\
		$2^{5}$&	8.87E-5	    & 0.7169   & 1.48E-4   & 0.7308    &2.94E-4  & 0.7773  \\
        $2^{6}$&	5.43E-5	    & 0.7078   & 8.63E-5   & 0.7789    &1.64E-4  & 0.8399  \\
		\hline
		predicted rate& &0.7500& &0.7500& &0.7500\\
		\hline
	\end{tabular}
\end{table}

\begin{xmpl}\label{exm:2}
We next investigate the spatial convergence rate of the proposed fully discrete scheme. This example employs the initial condition $G_{0} = x(1-x)$ and a fixed time step $\tau = T/1024$. Using the same parameter values $\alpha_0$, $\alpha_T$, $H$, and $m$ as in Example \ref{exm:1}, we present the results in \crefrange{table1_h}{table6_h}. These results are consistent with the theoretical convergence rate $\mathcal{O}\bigl(\max\{h^{2-2\rho},h^{\frac{2H}{\alpha_0}-2\rho-\varepsilon}\}\bigr)$ established in Theorem \ref{thm:space_dis_err_G}.
\end{xmpl}

\begin{table}[!t]
	\caption{Spatial errors and observed convergence rates for numerical schemes \eqref{scheme} and \eqref{scheme_v} with $H=0.6$,~$m=0~(\rho>1/4)$.}
	\label{table1_h}
	\centering
	%	\vspace{0.5em}	
	\medskip\small\renewcommand{\arraystretch}{1.15}
	\begin{tabular}{||c|c|c|c|c|c|c||}
		\hline
		& \multicolumn{2}{c|}{$(\alpha_0,~\alpha_T)=(0.3,~0.7)$} &  \multicolumn{2}{c|}{$(\alpha_0,~\alpha_T)=(0.5,~0.2)$}  & \multicolumn{2}{c||}{$(\alpha_0,~\alpha_T)=(0.8,~0.5)$} \\
		\hline
		$1/h$ & $e_h$&  $\mathrm{Rate}_{h}$ & $e_h$ & $\mathrm{Rate}_{h}$ & $e_h$ & $\mathrm{Rate}_{h}$\\
		\hline
		$2^{4}$&	3.45E-4     &    -  & 9.06E-4   &   -   &6.90E-3  &  -  \\
		$2^{5}$&	1.23E-4	    & 1.4861  & 3.36E-4   & 1.4309  &3.18E-3  & 1.1183  \\
		$2^{6}$&	4.29E-5	    & 1.5205  & 1.22E-4   & 1.4664  &1.39E-3  & 1.1986  \\
		$2^{7}$&	1.52E-5	    & 1.5002  & 4.34E-5   & 1.4876  &5.56E-4  & 1.3163  \\
        $2^{8}$&	5.32E-6	    & 1.5127  & 1.55E-5   & 1.4831  &2.14E-4  & 1.3789  \\
		\hline
		predicted rate& &1.5000& &1.5000& &1.0000\\
		\hline
	\end{tabular}
\end{table}

\begin{table}[!t]
	\caption{Spatial errors and observed convergence rates for numerical schemes \eqref{scheme} and \eqref{scheme_v} with $H=0.75$,~$m=0~(\rho>1/4)$.}
	\label{table2_h}
	\centering
	%	\vspace{0.5em}	
	\medskip\small\renewcommand{\arraystretch}{1.15}
	\begin{tabular}{||c|c|c|c|c|c|c||}
		\hline
		& \multicolumn{2}{c|}{$(\alpha_0,~\alpha_T)=(0.3,~0.7)$} &  \multicolumn{2}{c|}{$(\alpha_0,~\alpha_T)=(0.5,~0.2)$}  & \multicolumn{2}{c||}{$(\alpha_0,~\alpha_T)=(0.8,~0.5)$} \\
		\hline
		$1/h$ & $e_h$&  $\mathrm{Rate}_{h}$ & $e_h$ & $\mathrm{Rate}_{h}$ & $e_h$ & $\mathrm{Rate}_{h}$\\
		\hline
		$2^{4}$&	1.85E-4     &    -  & 3.42E-4     &   -     &2.15E-3  &  -  \\
		$2^{5}$&	6.25E-5	    & 1.5600  & 1.22E-4   & 1.4967  &8.85E-4  & 1.2819  \\
		$2^{6}$&	2.11E-5	    & 1.5678  & 4.27E-5   & 1.5095  &3.54E-4  & 1.3204  \\
		$2^{7}$&	7.35E-6	    & 1.5215  & 1.50E-5   & 1.5045  &1.34E-4  & 1.4029  \\
        $2^{8}$&	2.55E-6	    & 1.5235  & 5.33E-6   & 1.4980  &5.01E-5  & 1.4205  \\
		\hline
		predicted rate& &1.5000& &1.5000& &1.3750\\
		\hline
	\end{tabular}
\end{table}

\begin{table}[!t]
	\caption{Spatial errors and observed convergence rates for numerical schemes \eqref{scheme} and \eqref{scheme_v} with $H=0.6$,~$m=-0.5~(\rho>1/8)$.}
	\label{table3_h}
	\centering
	%	\vspace{0.5em}	
	\medskip\small\renewcommand{\arraystretch}{1.15}
	\begin{tabular}{||c|c|c|c|c|c|c||}
		\hline
		& \multicolumn{2}{c|}{$(\alpha_0,~\alpha_T)=(0.3,~0.7)$} &  \multicolumn{2}{c|}{$(\alpha_0,~\alpha_T)=(0.5,~0.2)$}  & \multicolumn{2}{c||}{$(\alpha_0,~\alpha_T)=(0.8,~0.5)$} \\
		\hline
		$1/h$ & $e_h$&  $\mathrm{Rate}_{h}$ & $e_h$ & $\mathrm{Rate}_{h}$ & $e_h$ & $\mathrm{Rate}_{h}$\\
		\hline
		$2^{4}$&	2.13E-4     &    -  & 5.09E-4   &   -   &3.53E-3  &  -  \\
		$2^{5}$&	6.46E-5	    & 1.7191  & 1.63E-4   & 1.6409  &1.39E-3  & 1.3416  \\
		$2^{6}$&	1.92E-5	    & 1.7529  & 5.08E-5   & 1.6820  &5.22E-4  & 1.4187  \\
		$2^{7}$&	5.73E-6	    & 1.7415  & 1.56E-5   & 1.7046  &1.81E-4  & 1.5248  \\
        $2^{8}$&	1.70E-6	    & 1.7499  & 4.78E-6   & 1.7069  &6.05E-5  & 1.5836  \\
		\hline
		predicted rate& &1.7500& &1.7500& &1.2500\\
		\hline
	\end{tabular}
\end{table}

\begin{table}[!t]
	\caption{Spatial errors and observed convergence rates for numerical schemes \eqref{scheme} and \eqref{scheme_v} with $H=0.75$,~$m=-0.5~(\rho>1/8)$.}
	\label{table4_h}
	\centering
	%	\vspace{0.5em}	
	\medskip\small\renewcommand{\arraystretch}{1.15}
	\begin{tabular}{||c|c|c|c|c|c|c||}
		\hline
		& \multicolumn{2}{c|}{$(\alpha_0,~\alpha_T)=(0.3,~0.7)$} &  \multicolumn{2}{c|}{$(\alpha_0,~\alpha_T)=(0.5,~0.2)$}  & \multicolumn{2}{c||}{$(\alpha_0,~\alpha_T)=(0.8,~0.5)$} \\
		\hline
		$1/h$ & $e_h$&  $\mathrm{Rate}_{h}$ & $e_h$ & $\mathrm{Rate}_{h}$ & $e_h$ & $\mathrm{Rate}_{h}$\\
		\hline
		$2^{4}$&	1.33E-4     &    -   & 2.20E-4  &   -    &1.14E-3  &  -  \\
		$2^{5}$&	3.78E-5	    & 1.8209  & 6.56E-5  & 1.7466  &4.05E-4  & 1.4995  \\
		$2^{6}$&	1.06E-5	    & 1.8299  & 1.94E-5  & 1.7593  &1.39E-4  & 1.5401  \\
		$2^{7}$&	3.05E-6	    & 1.7992  & 5.76E-6  & 1.7520  &4.54E-5  & 1.6150  \\
        $2^{8}$&	8.81E-7	    & 1.7934  & 1.72E-6  & 1.7447  &1.46E-5  & 1.6350 \\
		\hline
		predicted rate& &1.7500& &1.7500& &1.6250\\
		\hline
	\end{tabular}
\end{table}

\begin{table}[!t]
	\caption{Spatial errors and observed convergence rates for numerical schemes \eqref{scheme} and \eqref{scheme_v} with $H=0.6$,~$m=-1~(\rho>0)$.}
	\label{table5_h}
	\centering
	%	\vspace{0.5em}	
	\medskip\small\renewcommand{\arraystretch}{1.15}
	\begin{tabular}{||c|c|c|c|c|c|c||}
		\hline
		& \multicolumn{2}{c|}{$(\alpha_0,~\alpha_T)=(0.3,~0.7)$} &  \multicolumn{2}{c|}{$(\alpha_0,~\alpha_T)=(0.5,~0.2)$}  & \multicolumn{2}{c||}{$(\alpha_0,~\alpha_T)=(0.8,~0.5)$} \\
		\hline
		$1/h$ & $e_h$&  $\mathrm{Rate}_{h}$ & $e_h$ & $\mathrm{Rate}_{h}$ & $e_h$ & $\mathrm{Rate}_{h}$\\
		\hline
		$2^{4}$&	1.54E-4     &    -   & 3.16E-4   &   -   &1.89E-3  &  -  \\
		$2^{5}$&	4.17E-5	    & 1.8876  & 9.00E-5   & 1.8122  &6.45E-4  & 1.5469  \\
		$2^{6}$&	1.10E-5	    & 1.9168  & 2.49E-5   & 1.8514  &2.10E-4  & 1.6190  \\
		$2^{7}$&	2.92E-6	    & 1.9180  & 6.81E-6   & 1.8726  &6.43E-5  & 1.7077  \\
        $2^{8}$&	7.67E-7	    & 1.9281  & 1.85E-6   & 1.8829  &1.90E-5  & 1.7580  \\
		\hline
		predicted rate& &2.0000& &2.0000& &1.5000\\
		\hline
	\end{tabular}
\end{table}

\begin{table}[!t]
	\caption{Spatial errors and observed convergence rates for numerical schemes \eqref{scheme} and \eqref{scheme_v} with $H=0.75$,~$m=-1~(\rho>0)$.}
	\label{table6_h}
	\centering
	%	\vspace{0.5em}	
	\medskip\small\renewcommand{\arraystretch}{1.15}
	\begin{tabular}{||c|c|c|c|c|c|c||}
		\hline
		& \multicolumn{2}{c|}{$(\alpha_0,~\alpha_T)=(0.3,~0.7)$} &  \multicolumn{2}{c|}{$(\alpha_0,~\alpha_T)=(0.5,~0.2)$}  & \multicolumn{2}{c||}{$(\alpha_0,~\alpha_T)=(0.8,~0.5)$} \\
		\hline
		$1/h$ & $e_h$&  $\mathrm{Rate}_{h}$ & $e_h$ & $\mathrm{Rate}_{h}$ & $e_h$ & $\mathrm{Rate}_{h}$\\
		\hline
		$2^{4}$&	1.14E-4     &    -   & 1.69E-4   &   -    &6.54E-4  &  -  \\
		$2^{5}$&	2.96E-5	    & 1.9512  & 4.50E-5   & 1.9106  &2.02E-4  & 1.6952  \\
		$2^{6}$&	7.59E-6	    & 1.9634  & 1.18E-6   & 1.9273  &6.07E-5  & 1.7341  \\
		$2^{7}$&	1.95E-6	    & 1.9594  & 3.10E-6   & 1.9300  &1.75E-5  & 1.7944  \\
        $2^{8}$&	5.01E-7	    & 1.9629  & 8.13E-7   & 1.9328  &4.97E-6  & 1.8153  \\
		\hline
		predicted rate& &2.0000& &2.0000& &1.8750\\
		\hline
	\end{tabular}
\end{table}

\begin{figure}[htbp]
	\centering
	\setlength{\tabcolsep}{0pt}
	
	\begin{tabular}{@{}c@{\hspace{0.03\textwidth}}c@{}}
		\begin{subfigure}[b]{0.46\textwidth}
			\centering
			\includegraphics[width=\textwidth]{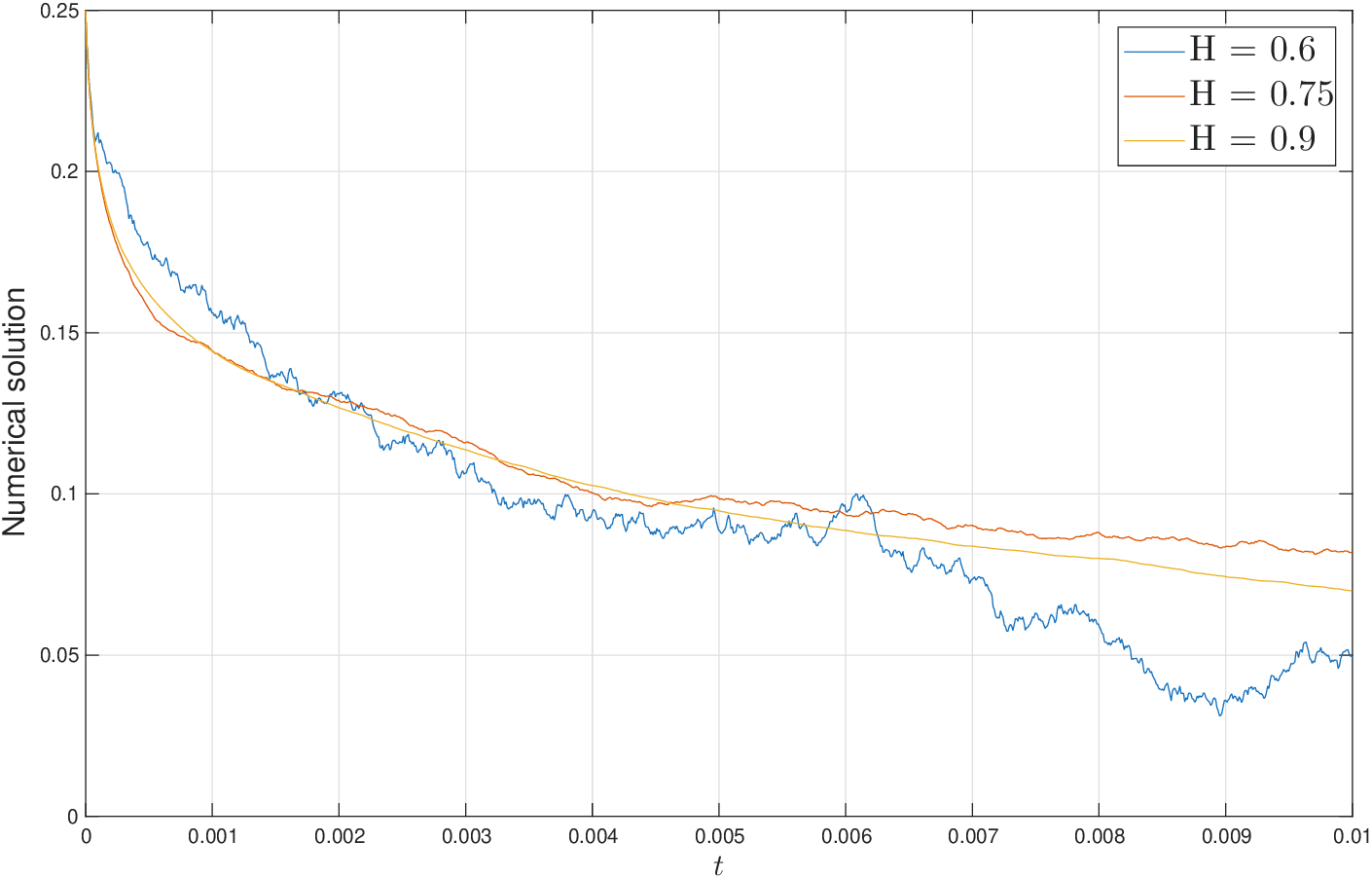}
			\caption{}
			\label{fig:a}
		\end{subfigure}
		&
		\begin{subfigure}[b]{0.46\textwidth}
			\centering
			\includegraphics[width=\textwidth]{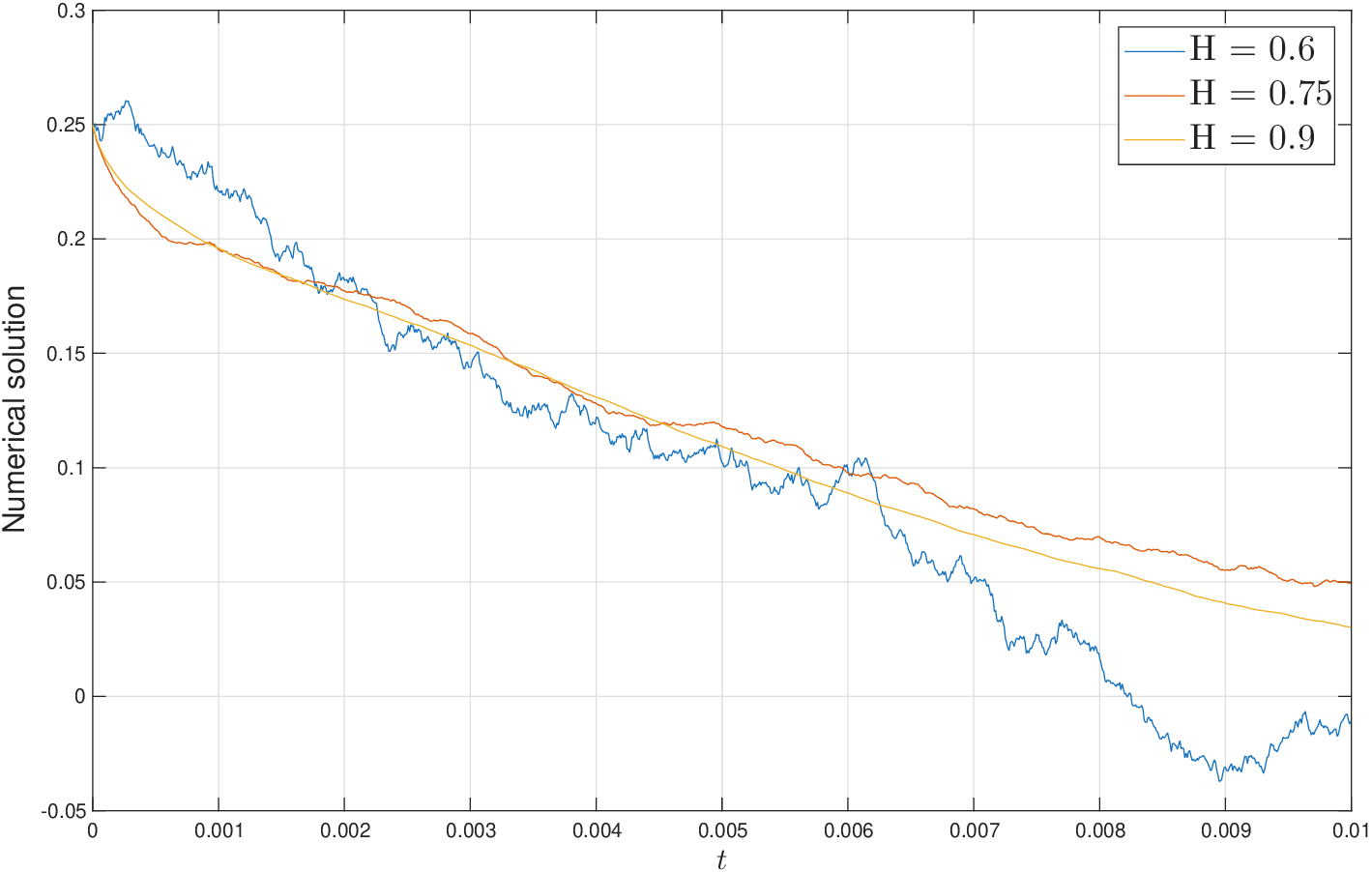}
			\caption{}
			\label{fig:b}
		\end{subfigure}
		\\[1em]
		\begin{subfigure}[b]{0.46\textwidth}
			\centering
			\includegraphics[width=\textwidth]{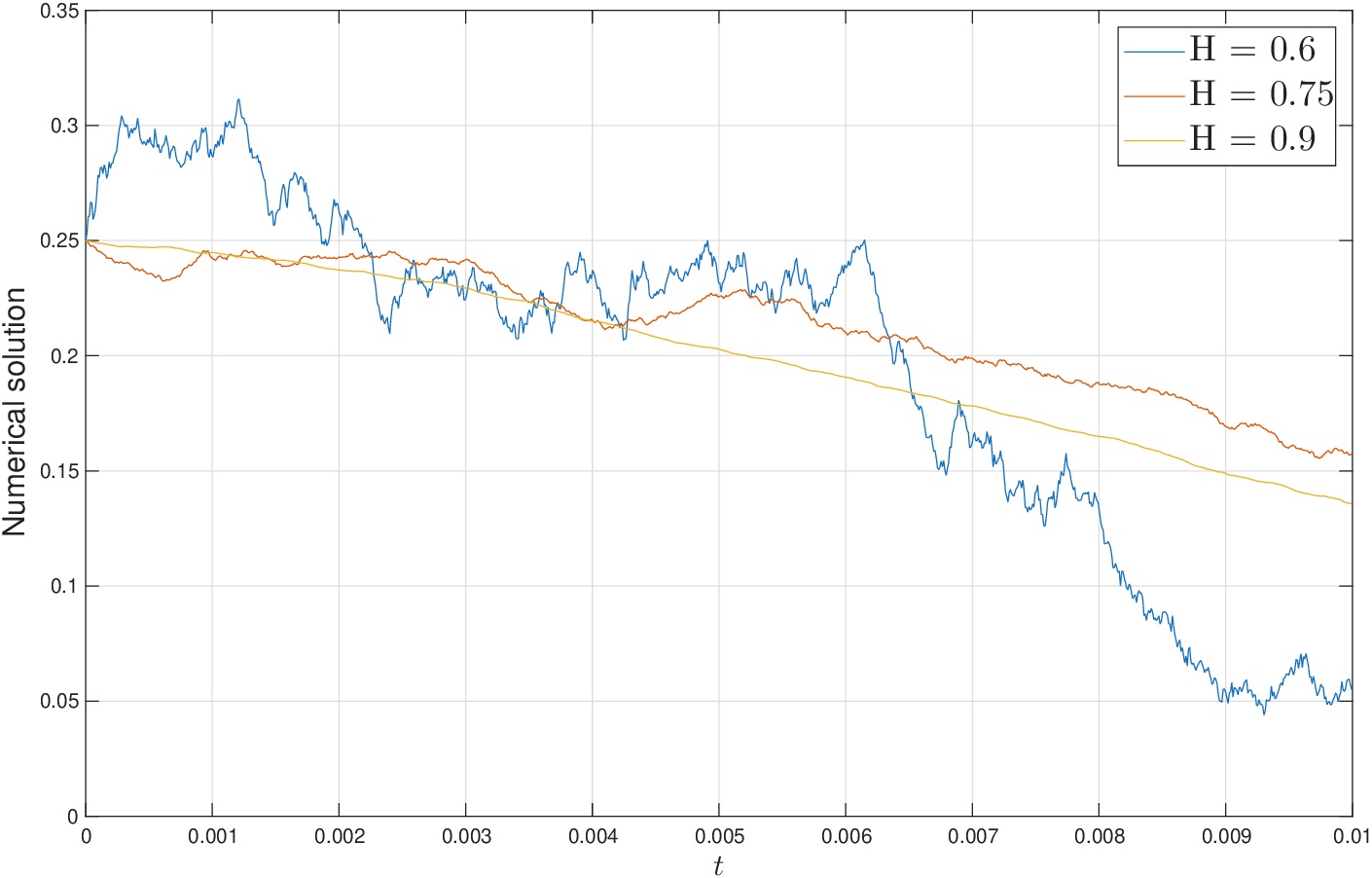}
			\caption{}
			\label{fig:c}
		\end{subfigure}
		&
		\begin{subfigure}[b]{0.46\textwidth}
			\centering
			\includegraphics[width=\textwidth]{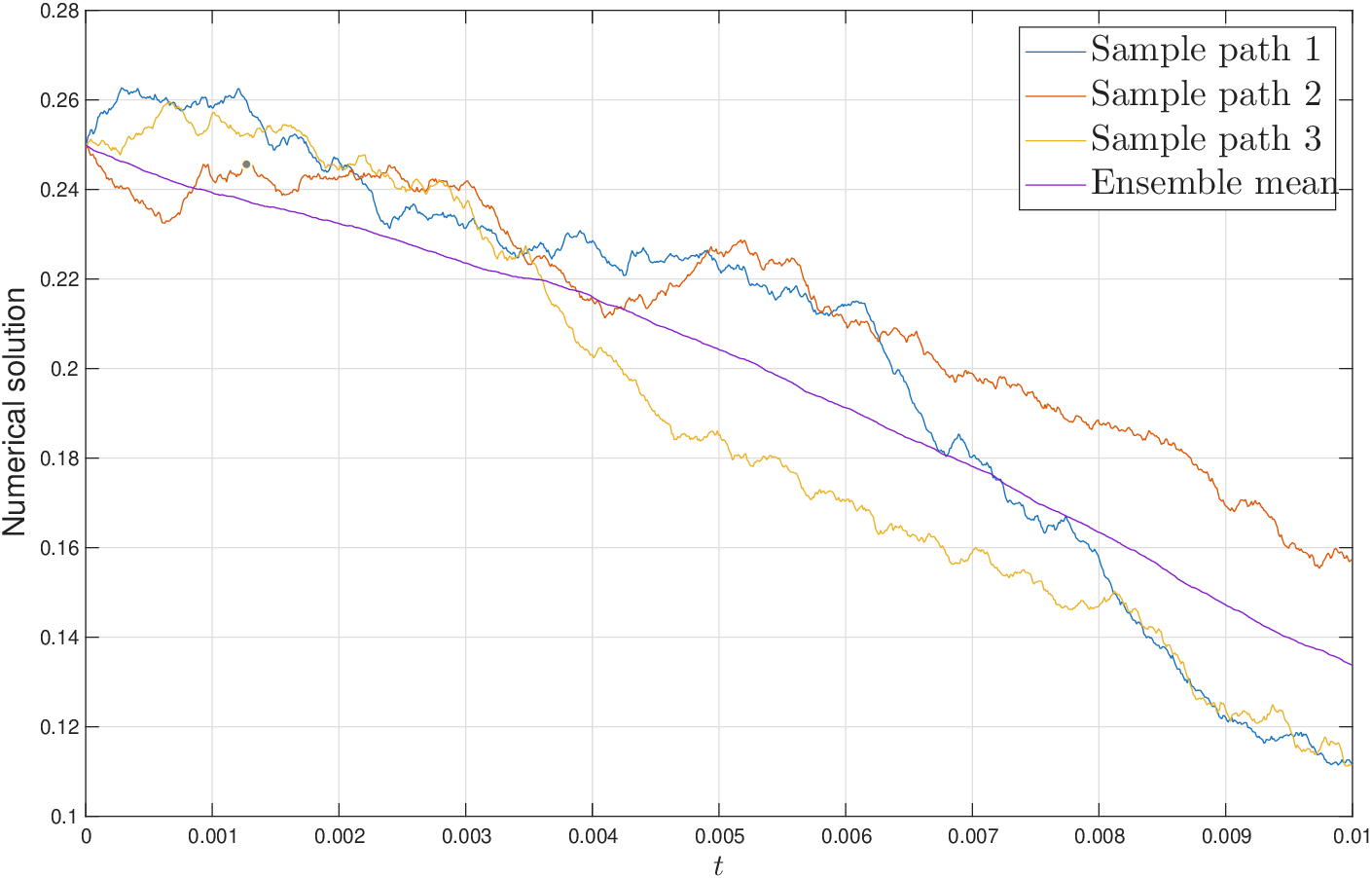}
			\caption{}
			\label{fig:d}
		\end{subfigure}
		    \end{tabular}
    \caption{Numerical solution of $G(0.5,t)$ with $h = 1/100$, $\tau = T/1024$, and $m = -0.5$. \textbf{(a)} $(\alpha_0,~\alpha_T) = (0.3,~0.7)$; \textbf{(b)} $(\alpha_0,~\alpha_T) = (0.5,~0.2)$; \textbf{(c)} $(\alpha_0,~\alpha_T) = (0.8,~0.5)$; \textbf{(d)} Sample paths and mean numerical solution with $(\alpha_0,~\alpha_T) = (0.8,~0.5)$, $H=0.75,~M=100$. }
    \label{fig:numerical_solution}
\end{figure}
In Example \ref{exm:2}, we also present several sample paths of the numerical solution for problem \eqref{model_2}. As shown in Figures \ref{fig:numerical_solution}(\subref{fig:a}), \ref{fig:numerical_solution}(\subref{fig:b}) and \ref{fig:numerical_solution}(\subref{fig:c}), the parameters are chosen as $(\alpha_0,~\alpha_T)=(0.3,~0.7),~(0.5,~0.2),~(0.8,~0.5)$. Setting $x = 0.5$, $h = 1/100$, $\tau = T/1024$ and $m = -0.5$, we compare the numerical solutions for different Hurst index $H = 0.6,~0.75,~0.9$. The results visually confirm the theoretical property of fractional Brownian motion: larger values of $H$ yield smoother solution trajectories, while smaller values of $H$ lead to stronger fluctuations, reflecting the distinct correlation structures of fractional Gaussian noise. Figure \ref{fig:numerical_solution}(\subref{fig:d}) displays three individual sample paths alongside the ensemble average obtained from 100 independent realizations. The figure qualitatively illustrates the inherent variability of individual trajectories and demonstrates how the ensemble mean effectively smooths out the sample-induced fluctuations, yielding a more stable representative mean path.

\section{Conclusion}\label{sec:conclu}
We consider a multiscale subdiffusion equation driven by fractional Gaussian noise, which describes anomalous diffusion phenomena characterizing mean squared displacement with time-varying sublinear growth rates, perturbed by noise with autocorrelation structure. We prove the well-posedness and solution regularity of the problem and develop and analyze a fully discrete numerical scheme. The sharpness of the error estimate is verified by extensive numerical experiments.

\bibliographystyle{siamplain}
\bibliography{references}

@book{Abramowitz1948handbook,
  title={Handbook of mathematical functions with formulas, graphs, and mathematical tables},
  author={Abramowitz, Milton and Stegun, Irene A},
  volume={55},
  year={1948},
  publisher={US Government printing office}
}

@book{Adams2003sobolev,
  title={Sobolev spaces},
  author={Adams, Robert A and Fournier, John JF},
  volume={140},
  year={2003},
  publisher={Elsevier}
}

@article{Bazhlekova2015,
  title     = {An analysis of the {{R}}ayleigh–{{S}}tokes problem for a generalized second-grade fluid},
  author    = {Bazhlekova, E. and Jin, B. and Lazarov, R. and others},
  journal   = {Numer. Math.},
  volume    = {131},
  pages     = {1--31},
  year      = {2015},
  doi       = {10.1007/s00211-014-0685-2}
}

@article{BenJoh,
  title={Fractional {{B}}rownian motions, fractional noises and applications},
  author={Mandelbrot, Benoit B and Van Ness, John W},
  journal={SIAM review},
  volume={10},
  number={4},
  eid={422--437},
  year={1968},
  publisher={SIAM}
}

@article{EfeLeuLiPunVab,
  title   = {Nonlocal transport equations in multiscale media: Modeling, dememorization, and discretizations},
  author  = {Efendiev, Yalchin and Leung, Wing Tat and Li, Wenyuan and Pun, Sai-Mang and Vabishchevich, Petr N.},
  year    = {2023},
  journal = {J. Comput. Phys.},
  volume  = {472},
  eid   = {111555},
}

@article{AlKar,
  title = {Numerical Approximation of Semilinear Subdiffusion Equations with Nonsmooth Initial Data},
  author = {Al-Maskari, Mariam and Karaa, Samir},
  year = {2019},
  month = jan,
  journal = {SIAM J. Numer. Anal.},
  volume = {57},
  number = {3},
  eid = {1524--1544}
}

@book{Chow2007stochastic,
  title={Stochastic partial differential equations},
  author={Chow, Pao-Liu},
  year={2007},
  publisher={Chapman and Hall/CRC}
}

@article{CleDa,
  author = {Cl{\'e}ment, P. and Da Prato, G.},
  title = {Some results on stochastic convolutions arising in Volterra equations perturbed by noise},
  journal = {Atti Accad. Naz. Lincei Cl. Sci. Fis. Mat. Natur. Rend. Lincei Mat. Appl.},
  volume = {7},
  number = {3},
  eid =  {147--153},
  year = {1996}
}

@article{DicGao,
  author  =  {Dick, Josef and Gao, Hecong and McLean, William and Mustapha, Kassem},
  title   = {Time-fractional diffusion equations with randomness, and efficient numerical estimations of expected values},
  journal = {Numer. Algorithms},
  volume  = {101},
  eid   = {1865--1883},
  year    = {2026},
}

@article{FenYaoLiWan,
  title = {An Inverse Source Problem for the Stochastic Multiterm Time-Fractional Diffusion-Wave Equation},
  author = {Feng, Xiaoli and Yao, Qiang and Li, Peijun and Wang, Xu},
  year = {2025},
  month = jun,
  journal = {SIAM/ASA J. Uncertain. Quantif.},
  volume = {13},
  number = {2},
  eid = {518--542}}

@article{GliEliRanChe,
  title = {From Inverse-Cascade to Subdiffusive Dynamic Scaling in Driven Disordered {{B}}ose Fluids},
  author = {Gliott, Elisabeth and Ran{\c c}on, Adam and Cherroret, Nicolas},
  year = {2024},
  month = dec,
  journal = {Phys. Rev. Lett.},
  volume = {133},
  number = {23},
  eid = {233403}
  }

@article{Gunzburger2019,
  title     = {Convergence of finite element solutions of stochastic partial integro-differential equations driven by white noise},
  author    = {Gunzburger, M. and Li, B. and Wang, J.},
  journal   = {Numer. Math.},
  volume    = {141},
  pages     = {1043--1077},
  year      = {2019}
}

@article{GunzMathComp,
  title={Sharp convergence rates of time discretization for stochastic time-fractional {PDE}s subject to additive space-time white noise},
  author={Gunzburger, Max and Li, Buyang and Wang, Jilu},
  journal={Math. Comput.},
  volume={88},
  number={318},
  eid={1715--1741},
  year={2019}
}

@article{HuAliEfeLeu,
  title   = {Partially explicit time discretization for time fractional diffusion equation},
  author  = {Hu, Jiuhua and Alikhanov, Anatoly A. and Efendiev, Yalchin and Leung, Wing Tat},
  year    = {2022},
  journal = {Fract. Calc. Appl. Anal.},
  volume  = {25},
  number  = {5},
  eid   = {1908--1924}
}

@article{IdoEdw,
  title = {Physical Nature of Bacterial Cytoplasm},
  author = {Golding, Ido and Cox, Edward C.},
  journal = {Phys. Rev. Lett.},
  volume = {96},
  issue = {9},
  eid = {098102},
  numpages = {4},
  year = {2006},
  month = {Mar}
}

@book{Jin,
  title={Fractional differential equations},
  author={Jin, Bangti},
  volume = {206},
  series = {Applied Mathematical Sciences},
  year={2021},
  publisher={Springer},
  address   = {Cham, Switzerland}
}

@book{JinYan,
  author    = {Jin, Bangti and Zhou, Zhi},
  title     = {Numerical Treatment and Analysis of Time-Fractional Evolution Equations},
  year      = {2023},
  series    = {Applied Mathematical Sciences},
  volume    = {214},
  publisher = {Springer},
  address   = {Cham, Switzerland}
}

@article{JinYanZho,
title = {Numerical approximation of stochastic time-fractional diffusion},
author = {Jin, Bangti and Yan, Yubin and Zhou, Zhi},
volume = {53},
eid = {1245--1268},
journal = {ESAIM: M2AN},
issn = {0764-583X},
publisher = {EDP Sciences},
number = {4},
year={2019}
}

@article{KovPri,
  title = {Strong Order of Convergence of a Fully Discrete Approximation of a Linear Stochastic {{V}}olterra Type Evolution Equation},
  author = {Kov{\'a}cs, Mih{\'a}ly and Printems, Jacques},
  year = {2014},
  month = jan,
  journal = {Math. Comp.},
  volume = {83},
  number = {289},
  eid =  {2325--2346}
}

@article{LapJFA,
  title = {Dirichlet and {{N}}eumann Eigenvalue Problems on Domains in {{Euclidean Spaces}}},
  author = {Laptev, A},
  year = {1997},
  month = dec,
  journal = {J. Funct. Anal.},
  volume = {151},
  number = {2},
  eid = {531--545}
}

@article{Lub1,
  title = {Convolution Quadrature and Discretized Operational Calculus. {{I}}},
  author = {Lubich, C.},
  year = {1988},
  month = jan,
  journal = {Numer. Math.},
  volume = {52},
  number = {2},
  eid =  {129--145}
}

@article{Lub2,
  title = {Convolution Quadrature and Discretized Operational Calculus. {{II}}},
  author = {Lubich, C.},
  year = {1988},
  month = jul,
  journal = {Numer. Math.},
  volume = {52},
  number = {4},
  eid =  {413--425}
}

@article{LiYau,
  title = {On the {{S}}chr{\"o}dinger Equation and the Eigenvalue Problem},
  author = {Li, Peter and Yau, Shing-Tung},
  year = {1983},
  month = sep,
  journal = {Comm. Math. Phys.},
  volume = {88},
  number = {3},
  eid = {309--318}
}

@article{MijNan,
  title = {Space--Time Fractional Stochastic Partial Differential Equations},
  author = {Mijena, Jebessa B. and Nane, Erkan},
  year = {2015},
  month = sep,
  journal = {Stochastic Process. Appl.},
  volume = {125},
  number = {9},
  eid =  {3301--3326}
}

@article{NieDeng2022,
title = {A Unified Convergence Analysis for the Fractional Diffusion Equation Driven by Fractional {{G}}aussian Noise with {{H}}urst Index ${{H}}\in(0,1)$},
author = {Nie, Daxin and Deng, Weihua},
journal = {SIAM J. Numer. Anal.},
volume = {60},
number = {3},
eid = {1548-1573},
year = {2022}
}

@article{NieSunDeng,
	author = {Nie, Daxin and Sun, Jing and Deng, Weihua},
	title = {Numerical approximation for stochastic nonlinear fractional diffusion equation driven by rough noise},
	journal = {ESAIM: M2AN},
	year = 2025,
	volume = 59,
	number = 1,
	eid = "389-418"
}

@article{NieSunDenSiamJNA,
author = {Nie, Daxin and Sun, Jing and Deng, Weihua},
title = {Strong Convergence Order for the Scheme of Fractional Diffusion Equation Driven by Fractional {{G}}aussian Noise},
journal = {SIAM J. Numer. Anal. },
volume = {60},
number = {4},
eid = {1879-1904},
year = {2022}
}

@article{PhiRal,
year = {2023},
month = {jun},
publisher = {IOP Publishing},
volume = {25},
number = {6},
eid = {063003},
author = {Meyer, Philipp G and Metzler, Ralf},
title = {Stochastic processes in a confining harmonic potential in the presence of static and dynamic measurement noise},
journal = {New J. Phys.}
}

@article{SerMen,
  title = {Non-Markovian Model for Transport and Reactions of Particles in Spiny Dendrites},
  author = {Fedotov, Sergei and M\'endez, Vicen\ifmmode \mbox{\c{c}}\else \c{c}\fi{}},
  journal = {Phys. Rev. Lett.},
  volume = {101},
  issue = {21},
  eid = {218102},
  numpages = {4},
  year = {2008},
  month = {Nov},
  publisher = {American Physical Society}
}

@article{SunZhaCheRee,
title = {Use of a variable-index fractional-derivative model to capture transient dispersion in heterogeneous media},
journal = {J. Contam. Hydrol.},
volume = {157},
eid = {47-58},
year = {2014},
issn = {0169-7722},
author = {Sun, HongGuang and Zhang, Yong and Chen, Wen and Reeves, Donald M. }
}

@article{SunNieDeng2023,
	author = {Sun, Jing and Nie, Daxin and Deng, Weihua},
	title = {An Efficient Numerical Algorithm for the Model Describing the Competition Between Super- and Sub-diffusions Driven by Fractional {{B}}rownian Sheet},
	journal = {J. Sci. Comput.},
	year = 2023,
	volume = 96,
	eid = 10
}

@book{ThoVid,
  title = {Galerkin Finite Element Methods for Parabolic Problems},
  author = {Thom{\'e}e, Vidar},
  year = {2006},
  series = {Springer Series in Computational Mathematics},
  edition = {2nd ed},
  number = {v. 25},
  publisher = {Springer},
  address = {Berlin ; New York}
}

@article{WuYan2025,
  title={Milstein Scheme for a Stochastic Semilinear Subdiffusion Equation Driven by Fractionally Integrated Multiplicative Noise},
  author={Wu, Xiaolei and Yan, Yubin},
  journal={Fractal Fract.},
  volume={9},
  number={5},
  eid = {314},
  year={2025},
  publisher={MDPI}
}

@article{XuFenShaSeeCha,
  title = {Subdiffusion of a Sticky Particle on a Surface},
  author = {Xu, Q. and Feng, L. and Sha, R. and Seeman, N. C. and Chaikin, P. M.},
  year = {2011},
  month = jun,
  journal = {Phys. Rev. Lett.},
  volume = {106},
  number = {22},
  eid = {228102}
  }

@book{Yul,
  author    = {Mishura, Yuliya S.},
  title     = {Stochastic Calculus for Fractional {{B}}rownian Motion and Related Processes},
  year      = {2008},
  series    = {Lecture Notes in Mathematics},
  volume    = {1929},
  publisher = {Springer},
  address   = {Berlin}
}

@article{ZengZhaKar,
  title   = {A generalized spectral collocation method with tunable accuracy for variable-order fractional differential equations},
  author  = {Zeng, Fanhai and Zhang, Zhongqiang and Karniadakis, George Em},
  journal = {SIAM J. Sci. Comput.},
  volume  = {37},
  number  = {6},
  eid   = {A2710--A2732},
  year    = {2015}
}

@article{ZengZhaKar2016,
  title   = {Fast difference schemes for solving high-dimensional time-fractional subdiffusion equations},
  author  = {Zeng, Fanhai and Zhang, Zhongqiang and Karniadakis, George Em},
  journal = {J. Comput. Phys.},
  volume  = {307},
  eid   = {15--33},
  year    = {2016}
}

@book{ZhaKar,
  title     = {Numerical Methods for Stochastic Partial Differential Equations with White Noise},
  author    = {Zhang, Zhongqiang and Karniadakis, George Em},
  series    = {Applied Mathematical Sciences},
  volume    = {196},
  publisher = {Springer},
  address   = {Cham},
  year      = {2017}
}

@article{ZheACOM,
  title = {Optimal control of variable-exponent subdiffusion},
  author = {Li, Yiqun and Liu, Mengmeng and Qiu, Wenlin},
   year    = {2025},
  journal = {arXiv preprint},
  volume  = {},
  number  = {},
  pages   = {},
  URL = {arXiv:2505.17678}
}

@article{ZheCMS,
  author  = {Qiu, Wenlin and Guo, Tao and Li, Yiqun and Guo, Xu and Zheng, Xiangcheng},
  title   = {A multiscale {{A}}bel kernel and application in viscoelastic problem},
  journal = {Commun. Math. Sci.},
  volume  = {24},
  number  = {7},
  eid   = {2047--2067},
  year    = {2026},
}

@Article{Zheng,
author = {Zheng , Xiangcheng},
title = {Two Methods Addressing Variable-Exponent Fractional Initial and Boundary Value Problems and {{A}}bel Integral Equation},
journal = {CSIAM Trans. Appl. Math.},
year = {2025},
volume = {6},
number = {4},
eid = {666--710}
}

@article {ZheMMS,
    AUTHOR = {Zheng, Xiangcheng and Li, Yiqun and Qiu, Wenlin},
     TITLE = {Local modification of subdiffusion by initial {F}ickian
              diffusion: multiscale modeling, analysis, and computation},
   JOURNAL = {Multiscale Model. Simul.},
    VOLUME = {22},
      YEAR = {2024},
    NUMBER = {4},
     eid =  {1534--1557},
      ISSN = {1540-3459,1540-3467},
   MRCLASS = {65M12 (35B65 35K20 35R11)},
  MRNUMBER = {4827876},
}
\end{document}